%% file: DimensionOfMDSplinesOnTMeshes.tex
\documentclass[1p,fleqn,11pt]{elsarticle}

\input{preamble}
\input{my_commands}

\usepackage{color}
\usepackage[colorinlistoftodos]{todonotes}
\definecolor{lightgray}{gray}{0.80}

\makeatletter
\def\ps@pprintTitle{%
	\let\@oddhead\@empty
	\let\@evenhead\@empty
	\def\@oddfoot{}%
	\let\@evenfoot\@oddfoot}
\makeatother

\begin{document}
	
\input{my_tikz_commands}

\title{Polynomial spline spaces of non-uniform bi-degree on T-meshes: \\Combinatorial bounds on the dimension}
\corref{cor1}
\author[ices]{Deepesh Toshniwal}
\ead{deepesh@ices.utexas.edu}
\author[inria]{Bernard Mourrain}
\author[ices]{Thomas J. R. Hughes}
\cortext[cor1]{Corresponding author}
\address[ices]{Institute for Computational Engineering and Sciences, University of Texas at Austin, USA}
\address[inria]{Universit\'e C\^ote d'Azur, Inria Sophia Antipolis M\'editerran\'ee, Sophia Antipolis, France}

\begin{abstract}
	Polynomial splines are ubiquitous in the fields of computer aided geometric design and computational analysis.
	Splines on T-meshes, especially, have the potential to be incredibly versatile since local mesh adaptivity enables efficient modeling and approximation of local features.
	Meaningful use of such splines for modeling and approximation requires the construction of a suitable spanning set of linearly independent splines, and a theoretical understanding of the spline space dimension can be a useful tool when assessing possible approaches for building such splines.
	Here, we provide such a tool.
	Focusing on T-meshes, we study the dimension of the space of bivariate polynomial splines, and we discuss the general setting where local mesh adaptivity is combined with local polynomial degree adaptivity.
	The latter allows for the flexibility of choosing non-uniform bi-degrees for the splines, i.e., different bi-degrees on different faces of the T-mesh.
	In particular, approaching the problem using tools from homological algebra, we generalize the framework and the discourse presented by Mourrain (2014) for uniform bi-degree splines.
	We derive combinatorial lower and upper bounds on the spline space dimension and subsequently outline sufficient conditions for the bounds to coincide.
	
\end{abstract}

\begin{keyword}
	Smooth splines \sep T-meshes \sep Non-uniform degrees \sep Dimension formula \sep Homological algebra
\end{keyword}

\maketitle

\input{ch_introduction}

\input{ch_preliminaries}
\input{ch_topological_complexes}
\input{ch_topological_complexes_constant}

\input{ch_topological_complexes_ideal}
\input{ch_dimension_formula}
%
\input{ch_stable_dimension_configs}
\input{ch_examples}

\input{ch_conclusions}

\section*{References}


\end{document}

%% file: preamble.tex
\usepackage{calc}
\usepackage{fancyhdr}
\usepackage{amsmath}
\usepackage{amssymb}
\usepackage{amsfonts}
\usepackage[normalem]{ulem}
\usepackage{graphicx}
\usepackage{amsthm}
\usepackage{amscd}
\usepackage{hyperref} 
\usepackage{lscape}
\usepackage{morefloats}
\usepackage{stmaryrd}
\usepackage{psfrag}
\usepackage{etex}
\usepackage[dvipsnames,table]{xcolor}
\usepackage{booktabs}
\usepackage{latexsym}
\usepackage{bm}
\usepackage{geometry}
\usepackage{color}
\usepackage[hang,small]{caption}
\usepackage{subcaption}
\usepackage{xparse}
\usepackage{xifthen}
\usepackage{empheq}
\usepackage{blkarray}
\usepackage{algorithm}
\usepackage[noend]{algpseudocode}
\usepackage{pgfplotstable}
\usepackage{pgfplots}
\pgfplotsset{compat=1.14}
\pgfplotsset{
	colormap={my basis colormap}{
		rgb255=(0, 114, 189);
		rgb255=(54, 106, 148);
		rgb255=(108, 98, 107);
		rgb255=(163, 91, 66);
		rgb255=(217, 83, 25);
	}
}
\pgfplotsset{
	colormap={my parula}{
		rgb255=(53.0655, 42.406, 134.9460);
		rgb255=(20.2336, 132.6061, 211.9511);
		rgb255=(55.5463, 184.8859, 157.9118);
		rgb255=(208.7187, 186.8470,  89.1966);
		rgb255=(248.9565, 250.6905, 13.7190);
	}
}
\pdfpageattr {/Group << /S /Transparency /I true /CS /DeviceRGB>>}

\usepackage{tikz}
\usepgfmodule{oo}
\usetikzlibrary{calc}
\usetikzlibrary{shapes}
\usetikzlibrary{plotmarks}
\usetikzlibrary{backgrounds}
\usetikzlibrary{decorations.pathmorphing}
\usetikzlibrary{decorations.markings}
\usetikzlibrary{external}
\usetikzlibrary{arrows,fit,matrix,positioning,shapes.geometric}
\usetikzlibrary{cd}

\usepackage{subcaption}
\usepackage{ulem}
\usepackage[inline]{enumitem}
\usepackage{multirow}
\usepackage{tcolorbox}
\usepackage{accents}
\usepackage{scalerel}
\usepackage{stackengine,wasysym}
\usepackage{tensor}
\usepackage{varwidth,calc}
\usepackage{wrapfig}
\usepackage{chngcntr}
\usepackage{apptools}
\AtAppendix{\counterwithin{theorem}{section}}

\makeatletter
\newcommand{\breakcaption}{\@dblarg\emit@breakcaption}
\long\def\emit@breakcaption[#1]#2{%
	\expandafter\caption\expandafter[\expandafter\emit@removeafter#1\\\@nil]{%
		\begin{varwidth}[t]{\textwidth-\widthof{\figurename\space\thefigure:\space}}
			#2
		\end{varwidth}%
	}%
}
\def\emit@removeafter#1\\#2\@nil{#1}
\makeatother

\numberwithin{equation}{section}

\renewcommand{\arraystretch}{1.5}

\usepackage{xargs}                      
\usepackage[colorinlistoftodos,prependcaption,textsize=tiny]{todonotes}
\newcommandx{\unsure}[2][1=]{\todo[linecolor=red,backgroundcolor=red!25,bordercolor=red,#1]{#2}}
\newcommandx{\change}[2][1=]{\todo[linecolor=blue,backgroundcolor=blue!25,bordercolor=blue,#1]{#2}}
\newcommandx{\info}[2][1=]{\todo[linecolor=OliveGreen,backgroundcolor=OliveGreen!25,bordercolor=OliveGreen,#1]{#2}}
\newcommandx{\improvement}[2][1=]{\todo[linecolor=Plum,backgroundcolor=Plum!25,bordercolor=Plum,#1]{#2}}
\newcommandx{\thiswillnotshow}[2][1=]{\todo[disable,#1]{#2}}

\makeatletter
\def\BState{\State\hskip-\ALG@thistlm}
\makeatother

\usepackage{diagrams}
\usepackage{color}
\newarrow{Dto}{<}{-}{-}{-}{>}
\usepackage[all]{xy}
\diagramstyle[labelstyle=\scriptstyle]
\usepackage{mathtools, leftidx}

\usepackage[nodots]{numcompress}

\newtheorem{theorem}{Theorem}[section]
\newtheorem{proposition}[theorem]{Proposition}
\newtheorem{lemma}[theorem]{Lemma}
\newtheorem{corollary}[theorem]{Corollary}
\newtheorem{remark}[theorem]{Remark}
\newtheorem{example}[theorem]{Example}
\newtheorem{definition}[theorem]{Definition}

\newtheorem{configuration}[theorem]{Configuration}
\newtheorem{assumption}[theorem]{Assumption}

\DeclareMathAlphabet{\mathpzc}{OT1}{pzc}{m}{it}

%% file: my_commands.tex
\newcommand{\mbf}[1]{{\boldsymbol{#1}}}
\newcommand{\RR}{{\mathbb{R}}}
\newcommand{\NN}{{\mathbb{N}}}
\newcommand{\PP}{{\mathcal{P}}}
\newcommand{\ZZ}{{\mathbb{Z}}}
\newcommand{\ZZP}{\mathbb{Z}_{\geq0}}
\DeclareMathOperator{\rank}{rank}
\newcommand{\rankwp}[1]{\rank{\left(#1\right)}}
\newcommand{\dimwp}[1]{\dim{\left(#1\right)}}
\DeclareMathOperator{\im}{im}
\newcommand{\imwp}[1]{\im{\left(#1\right)}}
\newcommand{\kerwp}[1]{\ker{\left(#1\right)}}
\newcommand{\moplus}{\mathop{\oplus}}

\newcommand{\isomorphic}{\cong}

\newcommand{\face}{\sigma}
\newcommand{\edge}{\tau}
\newcommand{\vertex}{\gamma}

\newcommand{\ms}{\rho}
\newcommand{\sms}{\lambda}

\DeclareMathOperator{\MS}{MS}
\newcommand{\MSA}[2]{\MS_{#1}^{#2}}
\newcommand{\IMSA}[2]{\interior{\MS}_{#1}^{#2}}
\newcommand{\MSH}[2]{\tensor*[^h]{\MS}{^{#2}_{#1}}}
\newcommand{\MSV}[2]{\tensor*[^v]{\MS}{^{#2}_{#1}}}
\newcommand{\IMSH}[2]{\tensor*[^h]{\interior{\MS}}{^{#2}_{#1}}}
\newcommand{\IMSV}[2]{\tensor*[^v]{\interior{\MS}}{^{#2}_{#1}}}

\newcommand{\interior}[1]{\accentset{\circ}{#1}}
\newcommand{\tmesh}{\mathcal{T}}
\newcommand{\tmeshF}{\tmesh_2}
\newcommand{\tmeshE}{\tmesh_1}
\newcommand{\tmeshV}{\tmesh_0}
\newcommand{\tmeshEH}{\tensor*[^h]{\tmesh}{_1}}
\newcommand{\tmeshEV}{\tensor*[^v]{\tmesh}{_1}}
\newcommand{\tmeshInterior}{\interior{\tmesh}}
\newcommand{\tmeshInteriorE}{\tmeshInterior_1}
\newcommand{\tmeshInteriorV}{\tmeshInterior_0}
\newcommand{\tmeshInteriorEH}{\tensor*[^h]{\tmeshInterior}{_1}}
\newcommand{\tmeshInteriorEV}{\tensor*[^v]{\tmeshInterior}{_1}}
\newcommand{\tmeshComponent}[2]{{\tmesh}_{#1}^{#2}}
\newcommand{\tmeshComponentF}[2]{{\tmesh}_{2#1}^{#2}}
\newcommand{\tmeshComponentE}[2]{{\tmesh}_{1#1}^{#2}}
\newcommand{\tmeshComponentV}[2]{{\tmesh}_{0#1}^{#2}}
\newcommand{\tmeshComponentEH}[2]{\tensor*[^h]{\tmesh}{^{#2}_{1#1}}}
\newcommand{\tmeshComponentEV}[2]{\tensor*[^v]{\tmesh}{^{#2}_{1#1}}}

\newcommand{\tmeshInteriorComponentE}[2]{\interior{\tmesh}_{1#1}^{#2}}
\newcommand{\tmeshInteriorComponentV}[2]{\interior{\tmesh}_{0#1}^{#2}}

\newcommand{\domain}{\Omega}
\newcommand{\domainInterior}{\interior{\domain}}
\newcommand{\domainBnd}{\partial\domain}

\newcommand{\domainComponent}[2]{\domain_{#1}^{#2}}
\newcommand{\nComponent}[1]{N_{#1}}

\newcommand{\ncells}[1]{\mathfrak{t}_{#1}}
\newcommand{\incells}[1]{\interior{\mathfrak{t}}_{#1}}
\newcommand{\nlevels}{\mathfrak{l}}

\newcommand{\DEG}{\mathcal{D}}
\newcommand{\adegree}{m}
\newcommand{\abdegree}{\mbf{\adegree}}

\newcommand{\degree}{\Delta m}
\newcommand{\bdegree}{{\mbf{\degree}}}

\newcommand{\degreev}{\Delta n}
\newcommand{\bdegreev}{{\mbf{\degreev}}}

\newcommand{\polyTotal}[1]{\PP_{#1}}
\newcommand{\pRing}{R}
\newcommand{\pRingH}{S}
\newcommand{\ideal}[1][I]{\mathfrak{#1}}
\newcommand{\bideg}[1]{\mbf{\delta}(#1)}

\newcommand{\shiftideal}[3]{\ideal[#1]{}_{#2}(#3)}

\newcommand{\weight}[3]{\omega_{\pB{#1},#3}^{#2}}

\newcommand{\pA}[1]{\|#1\|} 
\newcommand{\pB}[1]{[#1]} 
\newcommand{\degA}[3]{#1_{#2, \pA{#3}}} 
\newcommand{\degAa}[2]{#1_{\pA{#2}}}
\newcommand{\degB}[3]{#1_{#2, \pB{#3}}} 
\newcommand{\degBa}[2]{#1_{\pB{#2}}}

\DeclareMathOperator{\In}{In}
\newcommand{\initial}[1]{{\In #1}}

\newcommand{\smooth}{r}
\newcommand{\bsmooth}{\mbf{r}}

\newcommand{\splSpace}{\mathcal{R}}
\newcommand{\splSpaceH}{\mathcal{S}}

\newcommand{\elem}[1]{[#1]}
\newcommand{\faceE}[1]{\elem{\face_{#1}}}
\newcommand{\edgeE}[1]{\elem{\edge_{#1}}}
\newcommand{\vertexE}[1]{\elem{\vertex_{#1}}}
\newcommand{\msE}[1]{\elem{\ms_{#1}}}

\newcommand{\halfEdge}[2]{\elem{#1 | #2}}

\newcommand{\krond}[2]{\varepsilon_{#1, #2}}
\newcommand{\kronFE}[2]{\krond{\face_{#1}}{\edge_{#2}}}
\newcommand{\kronEV}[2]{\krond{\edge_{#1}}{\vertex_{#2}}}
\newcommand{\boundary}{\partial}

\newcommand{\idealComplex}{\mathcal{I}}
\newcommand{\constantComplex}{\mathcal{C}}
\newcommand{\quotientComplex}{\mathcal{Q}}
\newcommand{\homoDim}{h}
\newcommand{\holes}[2]{\pi_{#1}^{#2}}
\DeclareMathOperator{\ornt}{\varepsilon}
\newcommand{\euler}[1]{\chi \left( #1 \right)}

%% file: my_tikz_commands.tex
\newcommand{\includetikz}[1]{%
	\tikzsetnextfilename{./tikz/images/#1}%
	\input{#1}%
}

\newcommand\mybox[2][]{\tikz[overlay]\node[fill=blue!20,inner sep=2pt, anchor=text, rectangle, rounded corners=1mm,#1] {#2};\phantom{#2}}

\tikzset{->-/.style={decoration={
			markings,
			mark=at position #1 with {\arrow{>}}},postaction={decorate}}}

\definecolor{myBlue}{rgb} {0,0.4470,0.7410}
\definecolor{myRed}{rgb} {0.8500,0.3250,0.0980}
\definecolor{myGray1}{rgb} {0,0,0}
\definecolor{myGray2}{rgb} {0.6,0.6,0.6}
\definecolor{myGray3}{rgb} {0.45,0.45,0.45}
\definecolor{myGray4}{rgb} {0.8,0.8,0.8}

\tikzset{
	bThickness/.style={line width=#1\pgflinewidth},
	bThickness/.default={2},
}

\tikzset{
	eThickness/.style={line width=#1\pgflinewidth},
	eThickness/.default={0.5},
}

%% file: ch_introduction.tex
\section{Introduction}

Standard B-spline parameterizations of surfaces in Computer Aided
Geometric Design are defined on a grid of nodes over a rectangular
domain.
These representations are also the basis of Isogeometric
Analysis which provides high order finite element methods in
numerical simulations
\cite{CottrellIsogeometricAnalysisIntegration2009}.
However, grid structures do not allow
complex shapes to be easily resolved.
They also preclude the flexibility of performing local refinements 
for improving the error in numerical simulations.
To address these issues, meshes with T-junctions  --- also called {\em T-meshes} --- and polynomial and rational splines on such meshes have been investigated for performing both geometric
modeling and isogeometric analysis. 
Classically, uniform degree splines, i.e., piecewise polynomial functions of uniform degree on the faces, have been studied and developed on T-meshes with the intent of using T-junctions for locally increasing the resolution offered by the spline space.
An alternate strategy to improve the approximation power of splines is to increase the degree in a localized manner.
In this paper, motivated by applications for isogeometric finite element methods, we study the space of piecewise polynomials functions on a T-mesh with different bi-degrees on its faces and different regularities across its edges.
In particular, we analyze the dimension of these functional spaces, thus providing a tool that can help identify when a given set of linearly independent splines spans the full space.

Over the last decades, several works focused on the construction
of spline functions and the analysis of spline function spaces on
T-meshes have appeared, mainly motivated by applications in isogeometric analysis.
So called T-splines, which are B-spline functions defined on domains
with a T-mesh structure,
have been investigated for their flexibility of representing shapes \cite{SederbergTsplinesTNURCCs2003},
for isogeometric analysis \cite{BazilevsIsogeometricanalysisusing2010} and
for functional approximation \cite{schumaker_approximation_2012}.
However, linear dependencies of the blending functions involved in
T-spline constructions have been observed \cite{BuffaLinearindependenceTspline2010}.
To remedy this problem, a special sub-family of T-splines, called Analysis Suitable T-splines
has been developed, by imposing sufficient constraints on the T-mesh
\cite{li_linear_2012,scott_local_2012,BressanCharacterizationanalysissuitableTsplines2015}.

The construction of so-called LR-splines defined on
T-meshes and based on knot sub-grids has been proposed in
\cite{dokken_polynomial_2013}.
Their use in isogeometric analysis has been further investigated
in \cite{JohannessenIsogeometricanalysisusing2014}, including an
analysis of the linear independency of the blending functions \cite{bressan_properties_2013}.
Another type of splines, so-called hierarchical B-splines, have been
investigated in
\cite{forsey_hierarchical_1988,kraft97,deng_polynomial_2008,GiannelliTHBsplinestruncatedbasis2012}.
They are defined by recursive subdivisions of quadrangular faces,
producing nested spaces of splines functions and providing simple
schemes for performing local refinements.

In general, the dimension of the spaces of splines on T-meshes can be \emph{unstable}, i.e., it can depend on the global geometry of the T-mesh \cite{li_instability_2011,mourrain2014dimension}.
Since any efficient constructive approach must rely only on local data for building spline functions, this instability in the dimension necessitates identification of configurations where the spline space dimension is a priori guaranteed to be stable.
In this direction, a detailed study of spline spaces on general T-meshes has been
presented in \cite{mourrain2014dimension} using homological
techniques, which go back to \cite{billera1988homology}.
Results from \cite{mourrain2014dimension} were used in \cite{dokken2013} to devise a refinement strategy for LR-splines that ensures that the entire spline space is spanned by LR B-splines at each stage of refinement.
The dimension of Tchebycheffian spline spaces over planar T-meshes, which involve
non-polynomial functions, have been investigated in
\cite{BraccodimensionTchebycheffianspline2016},
\cite{BraccoTchebycheffiansplinespaces2019},
exploiting the same homological techniques as in \cite{mourrain2014dimension}.

In all the works referenced above, only uniform degree splines are considered on the T-meshes.
Here, we analyze in detail splines spaces over general T-meshes when non-uniform polynomial bi-degrees are chosen on the faces, thus accounting for local degree adaptivity in conjunction with local mesh adaptivity.
We provide combinatorial lower and upper bounds on the dimensions of such spline spaces and outline sufficient conditions for the bounds to coincide.
These sufficient conditions are equivalent to geometric conditions that need to be satisfied by the T-meshes.
The approach is based on homological techniques and generalizes the
framework presented in \cite{mourrain2014dimension} to the case of
non-uniform polynomial bi-degree distributions.
As part of the approach, we perform a degree-based decomposition of the mesh into nested regions and this allows us to untangle the contributions of different bi-degrees to the spline space dimension.
The main results on the lower and upper bounds of the dimension of
these spline spaces (Theorems \ref{thm:lower_bound_special}, \ref{thm:upper_bound_general} and \ref{thm:homo_dim_zero}) involve homological invariants of the nested regions associated to the different bi-degrees.
As mentioned previously, the theoretical results presented here can be used to identify when a given set of linearly independent splines spans the full spline space.
Conversely, given a constructive approach that aims to produce linearly independent splines over T-meshes using only local data, computation of the associated spline space dimension can help identify cases where the splines produced by the approach cannot be linearly independent.
This is crucial for devising constructive approaches that can be robustly employed for performing isogeometric analysis.

The layout of the paper is as follows.
We start by introducing preliminary concepts and notation about T-meshes and non-uniform bi-degree spline spaces on such meshes in Section \ref{sec:tmesh_splines}.
Thereafter, we introduce the topological complexes that form the main object of our analysis in Section \ref{sec:topology}; in particular, Section \ref{ss:approach_summary} provides an overview of our approach to the problem at hand.
Sections \ref{ss:constant_homology} and \ref{sec:ideals} take a closer look at the topological complexes introduced in Section \ref{sec:topology}, and the results presented therein are used in Section \ref{sec:dim_formula} to provide bounds on the spline space dimension (Theorems \ref{thm:lower_bound_special} and \ref{thm:upper_bound_general}).
Section \ref{sec:stable_dimension} contains Theorem \ref{thm:homo_dim_zero}, which outlines sufficient conditions for the bounds derived in Section \ref{sec:dim_formula} to coincide.
We also discuss the notion of \emph{maximal segment weights}, generalized from \cite{mourrain2014dimension}.
This notion helps provide a geometric criterion that is useful when computing the spline space dimension.
Finally, Section \ref{sec:examples} provides examples of the theory developed here.
We would like to mention here that computations using Macaulay2 \cite{M2} went hand-in-hand with the research presented here.

%% file: ch_preliminaries.tex
\section{Planar T-meshes and polynomials}\label{sec:tmesh_splines}
In the following we define the basic concepts associated with planar T-meshes, and thereafter present some preliminary results on polynomials. We will proceed as in \cite{mourrain2014dimension}, albeit in the setting of non-uniform degree spline spaces.

\input{ch_preliminaries_tmeshes}

\input{ch_preliminaries_splines}

\input{ch_preliminaries_splines_homogenous}

\input{ch_preliminaries_polynomials}

%% file: ch_preliminaries_tmeshes.tex
\subsection{T-meshes}
\begin{definition}[T-mesh]\label{def:tmesh}
	A T-mesh $\tmesh$ of $\RR^2$ is defined as:
	\begin{itemize}
		\item a finite set $\tmeshF$ of closed axis-aligned rectangles $\face$ of $\RR^2$, called $2$-cells or faces,
		\item a finite set $\tmeshE$ of closed axis-aligned segments $\edge$, called $1$-cells or edges, included in $\cup_{\face\in\tmeshF}\boundary\face$, and,
		\item a finite set, $\tmeshV$, of points $\vertex$, called $0$-cells or vertices, included in $\cup_{\edge\in\tmeshE}\boundary\edge$,
	\end{itemize}
	such that
	\begin{itemize}
		\item $\face \in \tmeshF \Rightarrow \boundary\face$ is a finite union of elements of $\tmeshE$,
		\item $\face, \face' \in \tmeshF \Rightarrow \face \cap \face' = \boundary\face \cap \boundary\face'$ is a finite union of elements of $\tmeshE \cup \tmeshV$, and,
		\item $\edge, \edge' \in \tmeshE \text{ with } \edge \neq \edge' \Rightarrow \edge\cap \edge' = \boundary\edge \cap \boundary\edge' \subset \tmeshV$.
	\end{itemize}
	The domain of the T-mesh is assumed to be connected and is defined as $\domain := \cup_{\face\in\tmeshF}\face \subset \RR^2$.
\end{definition}

Sets of horizontal and vertical edges will be denoted by $\tmeshEH$ and $\tmeshEV$, respectively.
Edges of the T-mesh are called interior edges if they intersect $\domainInterior$, and boundary edges otherwise. The set of interior edges will be denoted by $\tmeshInteriorE$; and the sets of interior horizontal and vertical edges will be denote by $\tmeshInteriorEH$ and $\tmeshInteriorEV$, respectively. Similarly, if a vertex is in $\domainInterior$ it will be called an interior vertex, and a boundary vertex otherwise. The set of interior vertices will be denoted by $\tmeshInteriorV$.
We will denote the number of $i$-cells with $\ncells{i} := \# \tmesh_i$; the number of interior $i$-cells with $\incells{i} := \# \tmeshInterior_i$; and so on.

\begin{example}\label{ex:tmesh}
An example T-mesh is shown in Figure \ref{fig:tmesh_example}.
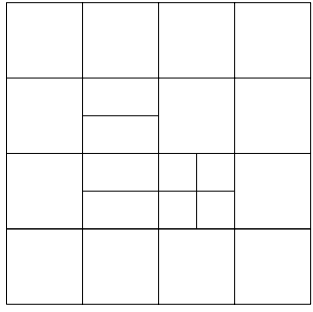
\begin{figure}[h]
	\centering
	\begin{subfigure}{\linewidth}
		\centering
	\tikzsetnextfilename{./tikz/images/tmesh_ex2_woutColors}%
	\input{tmesh_ex2_woutColors}%

	\end{subfigure}
	\caption{An example of the kind of T-meshes we will consider in this document.}
	\label{fig:tmesh_example}
\end{figure}
\end{example}

\begin{assumption}\label{ass:simplyConnectedDomain}
	The domain $\domain$ is simply connected, and $\domainInterior$ is connected.
\end{assumption}

%% file: tmesh_ex2_woutColors.tex
\begin{tikzpicture}
\foreach \i in {0,...,4}{
	\foreach \j in {1,...,4}{
		\node (a\i\j) at (\i,\j) {};
	}
}
\foreach \i in {1,...,5}{
	\foreach \j in {0,...,1}{
		\node (b\i\j) at (\i-1,\j) {};
	}
}

\draw[step=1,eThickness] (a01) grid (a44);
\draw[step=1,eThickness] (b10) grid (b51);

\draw[eThickness] ($(a21)!0.5!(a31)$) -- ($(a22)!0.5!(a32)$);
\draw[eThickness] ($(a12)!0.5!(a13)$) -- ($(a22)!0.5!(a23)$);
\draw[eThickness] ($(a11)!0.5!(a12)$) -- ($(a31)!0.5!(a32)$);

\end{tikzpicture}

%% file: ch_preliminaries_splines.tex
\subsection{Splines on T-meshes}
We will now define spaces of piecewise-polynomial splines on the planar T-meshes introduced above. To do so, we will first define a map that specifies relative polynomial degrees on the faces of $\tmesh$, and a second map that specifies the smoothness across its edges. Note that these maps are assumed to be known/fixed throughout this document and, when needed, we will omit mentioning them explicitly in order to simplify notation.

\begin{definition}[Degree deficit distribution]\label{def:deg_deficit}
	A degree deficit distribution on $\tmesh$ is a map
	\begin{equation*}
		\begin{split}
			\bdegree~:~ &\tmeshF \rightarrow \ZZP^2\;,\\
								&\face \mapsto \bdegree(\face)\;.
		\end{split}
	\end{equation*}
	It is assumed that $\DEG := \left\{ \bdegree(\face)~:~ \face \in \tmeshF \right\}$ can be totally ordered using the relation $\leq_\DEG$ defined as
	\begin{equation*}
		(a_1, a_2) \leq_\DEG (b_1, b_2) \Leftrightarrow a_1 \leq b_1\;\wedge\;a_2 \leq b_2\;,
	\end{equation*}
	and that $(0, 0) \in \DEG$.
\end{definition}
Given a degree-deficit distribution as defined above, we build the following sequence $\bdegree_i$, $i = 0, \dots, \nlevels$,
\begin{equation}
	 \min \DEG = (0, 0) =: \bdegree_0 < \bdegree_1 < \cdots < \bdegree_{\nlevels} := \max \DEG\;,
	 \label{eq:deg_def_seq}
\end{equation}
such that
\begin{equation}
	 \bdegree_i - \bdegree_{i-1} =: \bdegreev_i \in \{ (1, 0), (0, 1), (1, 1) \}\;.
	 \label{eq:deg_def_diff}
\end{equation}
Note that all comparisons carried out above are with respect to the ordering in Definition \ref{def:deg_deficit}.
We will denote the components of $\bdegree_i$ and $\bdegreev_i$ with $(\degree_{i1}, \degree_{i2})$ and $(\degreev_{i1}, \degreev_{i2})$, respectively.

The map $\bdegree$ will help specify the bi-degree of polynomials on a face $\face \in \tmeshF$.
Given $\abdegree \in \ZZP^2$, we define the bi-degrees $\abdegree_\square$, $\square \in \tmeshF \cup \tmeshE \cup \tmeshV$, as
\begin{equation}
	\abdegree_\square := \abdegree - \bdegree(\square)\;,
\end{equation}
where the induced degree deficits on $\edge\in \tmeshE$ and $\vertex
\in \tmeshV$ are defined as
\begin{equation}
	\bdegree(\edge) := \min_{\edge \subset \face} \bdegree(\face)\;,\qquad
	\bdegree(\vertex) := \min_{\vertex \in \face} \bdegree(\face)\;.
\end{equation}

Let $\pRing := \RR[s,t]$ be the polynomial ring with coefficients in $\RR$.
We define $\polyTotal{\adegree_1\adegree_2} \equiv \polyTotal{(\adegree_1,\adegree_2)}
\subset \pRing$ as the $\RR$-linear vector space of polynomials of bi-degree $\leq (\adegree_1, \adegree_2)$ spanned by the monomials $s^i t^j$, $0 \leq i \leq \adegree_1,\;0\leq j \leq \adegree_2$.
If any of $\adegree_1, \adegree_2$ are negative, then $\polyTotal{\adegree_1\adegree_2} := 0$.
\begin{definition}[Smoothness distribution]
	A smoothness distribution on $\tmesh$ is a map
	\begin{equation*}
		\begin{split}
			\bsmooth~:~&\tmeshInteriorE \rightarrow \ZZP\;,\\
										&\edge \mapsto \bsmooth(\edge)\;,
		\end{split}
	\end{equation*}
	such that $\edge, \edge' \in \tmeshInteriorEH$ (or both in $\tmeshInteriorEV$) and $\edge \cap \edge' \neq \emptyset \Rightarrow \bsmooth(\edge) = \bsmooth(\edge')$.
\end{definition}
The map $\bsmooth$ will help us define the smoothness across all interior edges, i.e., for $\edge \in \tmeshInteriorE$, splines will be required to be at least $C^{\bsmooth(\edge)}$ smooth across $\edge$. For $\vertex \in \tmeshInteriorV$ such that $\{\vertex \} = \edge_h \cap \edge_v$, $(\edge_h, \edge_v) \in \tmeshInteriorEH \times \tmeshInteriorEV$, the smoothness in horizontal and vertical directions is defined, respectively, as
\begin{equation}
	\bsmooth_{\vertex,h} = \bsmooth(\edge_v)\;,\qquad
	\bsmooth_{\vertex,v} = \bsmooth(\edge_h)\;.
	\label{eq:smoothnessAtVertex}
\end{equation}

\begin{definition}[Spline space]
	Given T-mesh $\tmesh$, degree deficit and smoothness distributions $\bdegree$ and $\bsmooth$, respectively, and $\abdegree \in \ZZP^2$, we define the spline space $\splSpace^\bsmooth_{\bdegree, \abdegree}(\tmesh)$ as
	\begin{equation}
		\begin{split}
			\splSpace^\bsmooth_{\bdegree, \abdegree} \equiv \splSpace^\bsmooth_{\bdegree, \abdegree}(\tmesh) := \bigg\{f~:~ &\forall \face \in \tmeshF,~~f|_\face \in \polyTotal{\abdegree_\face} = \polyTotal{\abdegree-\bdegree(\face)}\;,\\
			&\forall \edge \in \tmeshInteriorE,~~f \text{~is~} C^{\bsmooth(\edge)}\text{~across~}\edge\}\;.
		\end{split}
	\end{equation}
\end{definition}

\begin{figure}
	\centering
	\begin{subfigure}{\linewidth}
		\centering
	\tikzsetnextfilename{./tikz/images/tmesh_ex2}%
	\input{tmesh_ex2}%

	\end{subfigure}
	\caption{A T-mesh $\tmesh$ where the degree deficit on white faces is $(2,2)$, on blue elements is $(1,1)$, and on red elements is $(0, 0)$.}
	\label{fig:tmesh_example_degreeDist}
\end{figure}
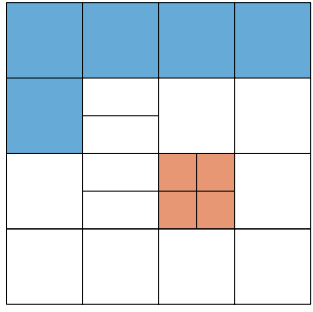

\begin{example}\label{ex:tmesh_degreeDist}
	Corresponding to the T-mesh shown in Example \ref{ex:tmesh}, an example degree deficit distribution has been shown in Figure \ref{fig:tmesh_example_degreeDist}. The set $\DEG$ is given by
	\begin{equation*}
		\DEG = \{ (0, 0), (1, 1), (2, 2) \}\;,
	\end{equation*}
	and the sequence $\bdegree_i$ can be chosen to be
	\begin{equation*}
		(0,0) = \bdegree_0 < (1, 1) = \bdegree_1 < (2, 1) = \bdegree_2 < (2, 2) = \bdegree_3\;.
	\end{equation*}
	Clearly, the above choice of the sequence is not unique. We could have alternatively chosen the shorter sequence
	\begin{equation*}
		(0,0) = \bdegree_0 < (1, 1) = \bdegree_1 < (2, 2) = \bdegree_2\;.
	\end{equation*}
\end{example}

We will employ the following algebraic characterization of smoothness of piecewise-polynomial splines \cite{billera_dimension_1991}.
\begin{lemma}[\cite{billera1988homology,mourrain2014dimension}]\label{lem:smoothness}
	For $\face, \face' \in \tmeshF$, let $\face \cap \face' = \edge \in \tmeshInteriorE$. A piecewise polynomial function equaling $p \in \pRing$ and $p'\in\pRing$ on $\face$ and $\face'$, respectively, is at least $\smooth$ times continuously differentiable across $\edge$ if and only if 
	\begin{equation*}
		p - p' \in \left(l ^{\smooth+1}\right)\;,
	\end{equation*}
	where $l \in \pRing$ is a non-zero linear polynomial vanishing on $\edge$.
\end{lemma}

%% file: tmesh_ex2.tex
\begin{tikzpicture}
	\foreach \i in {0,...,4}{
		\foreach \j in {1,...,4}{
			\node (a\i\j) at (\i,\j) {};
		}
	}
	\foreach \i in {1,...,5}{
		\foreach \j in {0,...,1}{
			\node (b\i\j) at (\i-1,\j) {};
		}
	}

	\fill[myBlue,fill opacity=0.6] (a03) rectangle (a44);
	\fill[myBlue,fill opacity=0.6] (a02) rectangle (a13);
	\fill[myRed,fill opacity=0.6] (a21) rectangle (a32);

	\draw[step=1,eThickness] (a01) grid (a44);
	\draw[step=1,eThickness] (b10) grid (b51);
	
	\draw[eThickness] ($(a21)!0.5!(a31)$) -- ($(a22)!0.5!(a32)$);
	\draw[eThickness] ($(a12)!0.5!(a13)$) -- ($(a22)!0.5!(a23)$);
	\draw[eThickness] ($(a11)!0.5!(a12)$) -- ($(a31)!0.5!(a32)$);
	
\end{tikzpicture}

%% file: ch_preliminaries_splines_homogenous.tex
\subsection{Homogenized problem}
We will translate the problem of investigating the dimension of $\splSpace^\bsmooth_{\bdegree, \abdegree}$ to the homogeneous setting \cite{billera1988homology, billera_dimension_1991}. To this end, let us introduce the ring of bi-homogeneous polynomials $\pRingH := \pRing[u, v] = \RR[s, t, u, v]$ which is interpreted as the extension of $\pRing$ by the variables $u$ and $v$ that homogenize $s$ and $t$, respectively. We denote the associated vector space of bi-homogeneous polynomials of bi-degree exactly $\abdegree = (\adegree_1, \adegree_2) \in \ZZP^2$ with $\pRingH_{\abdegree} \equiv \pRingH_{\adegree_1\adegree_2} \subset \pRingH$. This vector space is spanned by the monomials $s^iu^{\adegree_1-i}t^jv^{\adegree_2-j}$, $0 \leq i \leq \adegree_1$, $0 \leq j \leq \adegree_2$.
If any of $\adegree_1, \adegree_2$ are negative then $\pRingH_{\abdegree} := 0$.
The ring $\pRingH$ is naturally graded by $\ZZ^2$,
\begin{equation}
	\pRingH = \bigoplus_{(i, j) \in \ZZ^2} \pRingH_{ij}\;,\qquad
	\pRingH_{ij}\pRingH_{kl} = \pRingH_{(i+k)(j+l)}\;,
\end{equation}
and its graded pieces are shifted in the usual manner: $\pRingH(-i, -j)_{kl} = \pRingH_{(k-i)(l-j)}$.

For convenience, let us define the following notation $\mbf{s} := (s, t)$, $\mbf{u} := (u, v)$, and for any tuple $(a, b)$,
$
	(a, b)^{(i,j)} := a^i b^j\;,$ for $i, j \in \ZZP.
$ 
Using the above, we define the vector space associated to $\square \in \tmeshF \cup \tmeshE \cup \tmeshV$ as
\begin{equation}
	\pRingH_\square := \left( \mbf{u}^{\bdegree(\square)} \right) = \mbf{u}^{\bdegree(\square)} \pRingH(-\bdegree(\square))\;.
\end{equation}
In particular, given $\abdegree \in \ZZP^2$, we denote its $\abdegree^{th}$ graded piece as $\pRingH_{\square, \abdegree} = \mbf{u}^{\bdegree(\square)}\pRingH(-\bdegree(\square))_{\abdegree}$.
An algebraic characterization of smoothness for bi-homogeneous piecewise polynomial functions follows in the vein of Lemma \ref{lem:smoothness}, and is stated below. The module of bi-homogeneous splines of interest is defined immediately thereafter.
\begin{lemma}\label{lem:smoothnessH}
	For $\face, \face' \in \tmeshF$, let $\face \cap \face' = \edge \in \tmeshInteriorE$. A bi-homogeneous piecewise polynomial function equaling $p \in \pRingH$ and $p'\in\pRingH$ on $\face$ and $\face'$, respectively, is at least $\bsmooth(\edge)$ times continuously differentiable across $\edge$ if and only if 
	\begin{equation*}
	p - p' \in \left( l_\edge^{\bsmooth(\edge)+1} \right)\;,
	\end{equation*}
	where $l_\edge$ is a non-zero $u$-homogeneous (resp. $v$-homogeneous) linear polynomial vanishing on $\edge \in \tmeshEV$ (resp. $\edge \in \tmeshEH$).
\end{lemma}

\begin{definition}[Module of bi-homogeneous splines]
	Given T-mesh $\tmesh$, degree and smoothness distributions $\bdegree$ and $\bsmooth$, respectively, we define the module of bi-homogeneous splines $\splSpaceH^\bsmooth_{\bdegree}(\tmesh)$ as
	\begin{equation}
	\begin{split}
	\splSpaceH^\bsmooth_{\bdegree} \equiv \splSpaceH^\bsmooth_{\bdegree}(\tmesh) := \bigg\{f~:~ &\forall \face \in \tmeshF,~~f|_\face \in \pRingH_\face\;,\\
	&\forall \edge \in \tmeshInteriorE,~~f \text{~is~} C^{\bsmooth(\edge)}\text{~across~}\edge\}\;.
	\end{split}
	\end{equation}
\end{definition}
For a given $\abdegree \in \ZZP^2$, the above $\pRingH$-module of splines is of interest precisely because its $\abdegree$-th graded piece is isomorphic to $\splSpace^\bsmooth_{\bdegree, \abdegree}$.
\begin{theorem}
	\begin{equation*}
		\dimwp{\splSpaceH^\bsmooth_{\bdegree}}_{\abdegree} = \dimwp{\splSpace^\bsmooth_{\bdegree, \abdegree}}\;.
	\end{equation*}
\end{theorem}

%% file: ch_topological_complexes.tex
\section{Topological complexes}\label{sec:topology}
In this section we will describe the tools from homology that we will use for computing the dimension of graded pieces of $\splSpaceH^\bsmooth_\bdegree$ but first let us introduce the relevant notation.

First, we define the ideal
\begin{equation}
	\shiftideal{L}{\pB{i}}{-(j, k)} \equiv \shiftideal{L}{\pB{i}}{-j, -k} :=
	\begin{dcases}
		0\;, & i = \nlevels+1\\ 
		\mbf{u}^{\bdegree_i}\pRingH(-\bdegree_i -(j, k))\;, & 0 \leq i \leq \nlevels
	\end{dcases}\;,
\end{equation}
and denote $\ideal[L]{}_{\pB{i}} \equiv \shiftideal{L}{\pB{i}}{0,0}$.
Then, from Equation \eqref{eq:deg_def_diff}, we have for $i = 1, \dots, \nlevels$,
\begin{equation}
	\shiftideal{L}{\pB{i}}{ -j, -k} = \mbf{u}^{\bdegreev_i}\shiftideal{L}{\pB{i-1}}{-(j, k) - \bdegreev_i}\;.
	\label{eq:ideal_i_im1}
      \end{equation}
We also define the (shifted) quotient
\begin{equation}
	\shiftideal{M}{\pB{i}}{ -j, -k} := \shiftideal{L}{\pB{i-1}}{ -j,
		-k} / \shiftideal{L}{\pB{i}}{ -j, -k}\;.
\end{equation}
For $\edge \in \tmeshInteriorE$, it will be convenient to set $\Delta_\edge := l_\edge^{\bsmooth(\edge)+1}$, where $l_\edge$ is the homogeneous linear polynomial from Lemma \ref{lem:smoothnessH}.
We will use $\Delta_\edge$ to define the edge and vertex associated ideals $\ideal_\edge$ and $\ideal_\vertex$,
\begin{equation}
	\ideal_\edge := \left(\Delta_\edge \mbf{u}^{\bdegree(\edge)}\right) \subset \pRingH_{\edge}\;,\qquad
	\ideal_\vertex := \sum_{\vertex \in \edge}\left(\Delta_\edge \mbf{u}^{\bdegree(\edge)}\right) \subset \pRingH_{\vertex}\;.
\end{equation}
In general, while $\ideal_\edge = \pRingH_\edge \cap \left( \Delta_\edge \right)$, we have $\ideal_\vertex = \pRingH_{\vertex} \cap \sum_{\vertex \in \edge}\left(\Delta_\edge\right)$ only when $\vertex$ is a crossing vertex.

\subsection{Definitions}
Orienting the $i$-cells in a compatible manner, $i = 0, 1, 2$, they will generate $\pRingH$-modules, and we will index the generators with the respective faces, edges and vertices. These indexed generators will be denoted with $\faceE{}, \edgeE{}$ and $\vertexE{}$.
We will assume that all oriented $2$-cells have been assigned a counter-clockwise orientation. For $\edge \in \tmeshE$ with end points $\vertex,\vertex' \in \tmeshV$, the generator associated will be represented as $\edgeE{} = [\vertex\vertex']$, with $[\vertex'\vertex] = -[\vertex\vertex']$ defining the oppositely oriented edge. In the following sections we will only be interested in homology relative to $\domainBnd$. Therefore, we will assume that $\edgeE{} = 0$ and $\vertexE{} = 0$ when $\edge \subset \domainBnd$ and $\vertex \in \domainBnd$, respectively.

Let $\krond{\theta}{\phi} \equiv \ornt([\theta],[\phi])$ denote the orientation function, i.e., the function that takes as inputs one $n$-dimensional cell $[\theta]$, and one $(n-1)$-dimensional cell, $[\phi]$, and returns
\begin{itemize}
	\item $0$ if $\theta \cap \phi = \emptyset$;
	\item $-1$ if $\theta \cap \phi = \phi$ and the orientation endowed by $[\theta]$ upon its boundary is incompatible with orientation of $[\phi]$; and,
	\item $+1$ if $\theta \cap \phi = \phi$ and the orientation endowed by $[\theta]$ upon its boundary is compatible with orientation of $[\phi]$.
\end{itemize}
We will consider the usual boundary maps, $\boundary$, defined as
\begin{equation}
	\begin{split}
		\boundary(\faceE{}) = \sum_{\edge\in\tmeshInteriorE} \kronFE{}{}\edgeE{}\;,\qquad
		\boundary(\edgeE{}) = \sum_{\vertex\in\tmeshInteriorV} \kronEV{}{}\vertexE{}\;,\qquad
		\boundary(\vertexE{}) = 0\;.
	\end{split}
\end{equation}
Then, for an element $p = \sum_\face \faceE{}p_\face$ of the $\pRingH$-module $\moplus_{\face \in \tmeshF}\faceE{} \pRingH_\face$, its image under the action of $\boundary$ will be
\begin{equation}
\begin{split}
\boundary\left( \sum_{\face \in \tmeshF} \faceE{} p_\face \right) = \sum_{\edge \in \tmeshInteriorE}\edgeE{} \left(\sum_{\face \in \tmeshF} \kronFE{}{}p_{\face}\right)\;.
\end{split}
\end{equation}
It is clear that $\sum_\face \kronFE{}{}p_{\face} \in \pRingH_\edge
\ (=\mbf{u}^{\bdegree(\edge)}\, S)$. Therefore, by Lemma \ref{lem:smoothnessH}, for $p$ to be in smoothness class $C^\bsmooth$ we require the following,
\begin{equation}
\begin{split}
\forall \edge \in \tmeshInteriorE~,~~\sum_{\face \in \tmeshF} \kronFE{}{}p_{\face} \in \ideal_\edge\;.
\end{split}
\end{equation}
Then, $\splSpaceH^\bsmooth_{\bdegree}$ contains all splines $f$ (in all bi-degrees) such that their polynomial pieces satisfy the above requirement, with $p_\face = f|_\face$. In other words, for $\abdegree \in \ZZP^2$,
\begin{equation}
\dimwp{\splSpace^\bsmooth_{\bdegree, \abdegree}} =
\dimwp{\splSpaceH^\bsmooth_\bdegree}_\abdegree =
\dimwp{\kerwp{\overline{\boundary}}}_\abdegree\;,
\label{eq:spline_space_iso_homology}
\end{equation}
where $\overline{\boundary}$, given below, is obtained by composing $\boundary$ with the natural quotient map,
\begin{equation}
\overline{\boundary}~:~\moplus_{\face \in \tmeshF} \faceE{} \pRingH_{\face} \rightarrow \moplus_{\edge \in \tmeshInteriorE} \edgeE{} \pRingH_{\edge}/\ideal_{\edge}\;.
\end{equation}

\subsection{Degree-deficit based topological complexes}
In light of the above reasoning, we consider the following chain complex of $\pRingH$-modules as the object of our analysis, with the top homology module of $\quotientComplex$ equaling $\splSpaceH^\bsmooth_\bdegree$.
\begin{equation}
\begin{tikzcd}
	\quotientComplex~:~ & \moplus\limits_{\face \in \tmeshF} \faceE{} \pRingH_{\face} \arrow[r] & \moplus\limits_{\edge \in \tmeshInteriorE} \edgeE{} \pRingH_{\edge}/\ideal_\edge \arrow[r] & \moplus\limits_{\vertex \in \tmeshInteriorV} \vertexE{} \pRingH_{\vertex}/\ideal_\vertex \arrow[r] & 0\;.
\end{tikzcd}
\label{eq:quotient_complex}
\end{equation}
We analyze $\quotientComplex$ by performing a decomposition that untangles the individual contributions from the different degree-deficits from Equation \eqref{eq:deg_def_seq} to the dimension of $\splSpaceH^\bsmooth_\bdegree$.

First, for $\square \in \tmeshF \cup \tmeshE \cup \tmeshV$, we define
\begin{align}
	&\degA{\pRingH}{\square}{i} = \pRingH_{\square}/(\pRingH_{\square}\cap \ideal[L]{}_{\pB{i}})\;, &\qquad &\degA{\ideal}{\square}{i} = \ideal_{\square} \cdot \pRingH_{\square, \pA{i}}\;,\\
	&\degB{\pRingH}{\square}{i} = \degA{\pRingH}{\square}{i} \cap \ideal[L]{}_{\pB{i-1}}\;, &\qquad&\degB{\ideal}{\square}{i} = \ideal_{\square} \cdot \degB{\pRingH}{\square}{i}\;.
\end{align}
We also define the complex $\degBa{\quotientComplex}{i}$ as 
\begin{equation}
\begin{tikzcd}
  \moplus\limits_{\face \in \tmeshF} \faceE{} \degB{\pRingH}{\face}{i} \arrow[r]& \moplus\limits_{\edge \in \tmeshInteriorE} \edgeE{} \degB{\pRingH}{\edge}{i} / \degB{\ideal}{\edge}{i} \arrow[r] & \moplus\limits_{\vertex \in \tmeshInteriorV} \vertexE{} \degB{\pRingH}{\vertex}{i} / \degB{\ideal}{\vertex}{i}
\arrow[r] & 0\;, 
\end{tikzcd}
\end{equation}
and the complex $\degAa{\quotientComplex}{i}$ as
\begin{equation}
\begin{tikzcd}
   \moplus\limits_{\face \in \tmeshF} \faceE{} \degA{\pRingH}{\face}{i} \arrow[r]& \moplus\limits_{\edge \in \tmeshInteriorE} \edgeE{} \degA{\pRingH}{\edge}{i} / \degA{\ideal}{\edge}{i} \arrow[r] & \moplus\limits_{\vertex \in \tmeshInteriorV} \vertexE{} \degA{\pRingH}{\vertex}{i} / \degA{\ideal}{\vertex}{i}
\arrow[r] & 0\;.
\end{tikzcd}
\end{equation}
Let us make a few observations about the setup so far and the motivations behind it:
\begin{itemize}
\item By construction of these complexes, we have a sequence of complexes 
\begin{equation}
	\begin{tikzcd}
	0 \arrow[r] & \degBa{\quotientComplex}{i} \arrow[r] & \degAa{\quotientComplex}{i} \arrow[r] & \degAa{\quotientComplex}{i-1} \arrow[r] & 0\;. 
	\end{tikzcd}
\label{eq:exact_stair}
\end{equation}
In other words, the maps induced on the corresponding components of $\degBa{\quotientComplex}{i},\; \degAa{\quotientComplex}{i},\; \degAa{\quotientComplex}{i-1}$ form complexes.
As we will see in Proposition \ref{prop:exact_stair}, these complexes are exact.
                      
\item With regards to the analysis of $\quotientComplex$ from Equation \eqref{eq:quotient_complex}, it is clear from the above definitions that $\degAa{\quotientComplex}{\nlevels+1} = \quotientComplex$. We will perform its analysis using the previous sequence of complexes.

	\item It can be observed that each module in $\degAa{\quotientComplex}{0}$ is identically zero, thereby implying $\degAa{\quotientComplex}{1} = \degBa{\quotientComplex}{1}$. Therefore, it is only necessary to analyze the complexes
	\begin{equation*}
	\degBa{\quotientComplex}{\nlevels+1}\;,\;
	\degBa{\quotientComplex}{\nlevels}\;,\;
	\dots\;,\;
	\degBa{\quotientComplex}{2}\;,\;
	\degBa{\quotientComplex}{1}\;.
	\end{equation*}
	\item Finally, for a given $\abdegree \in \ZZP^2$, it can be observed that the dimension of the $\abdegree$-th graded piece of the top homology module of ${\quotientComplex}_{[\nlevels+1]}$ is equal to the dimension of $\splSpace^\bsmooth_{\mbf{0}, \abdegree - \bdegree_{\nlevels}}$, which is the largest uniform-degree spline space contained in $\splSpace^\bsmooth_{\bdegree, \abdegree}$. Therefore, in a rough sense, the contributions from the complexes $\degBa{\quotientComplex}{\nlevels}\;,\; \dots\;,\;\degBa{\quotientComplex}{1}$ represent the incremental changes in the dimension of $\splSpace^\bsmooth_{\mbf{0}, \abdegree - \bdegree_{\nlevels}}$ that result from the introduction of non-uniformity in bi-degrees. The net effect of these incremental changes on the dimension of  $\splSpace^\bsmooth_{\mbf{0}, \abdegree - \bdegree_{\nlevels}}$ will intuitively bring us close to the quantity of interest, $\dimwp{\splSpace^\bsmooth_{\bdegree, \abdegree}} = \dimwp{\splSpaceH^\bsmooth_{\bdegree}}_\abdegree$.
\end{itemize}

\begin{proposition}\label{prop:exact_stair}
The following is a short exact sequence of complexes
  \begin{equation}
	\begin{tikzcd}
	0 \arrow[r] & \degBa{\quotientComplex}{i} \arrow[r] & \degAa{\quotientComplex}{i} \arrow[r] & \degAa{\quotientComplex}{i-1} \arrow[r] & 0\;.
	\end{tikzcd}
  \end{equation}
\end{proposition}                    
\begin{proof}
	We have to prove that the following is a short exact sequence for all $\square \in \tmeshF \cup \tmeshE \cup \tmeshV$:
	\begin{equation*}
		\begin{tikzcd}[column sep=normal]
			0 \arrow[r] & \degB{\pRingH}{\square}{i} / \degB{\ideal}{\square}{i} \arrow[r] & \degA{\pRingH}{\square}{i} / \degA{\ideal}{\square}{i} \arrow[r] \arrow[r] & \degA{\pRingH}{\square}{i-1} / \degA{\ideal}{\square}{i-1} \arrow[r] & 0\;,
		\end{tikzcd}
	\end{equation*}
	where $\degB{\ideal}{\face}{i} := 0$.
	
	Given $i$, if $\pRingH_{\square} \subsetneq \ideal[L]{}_{\pB{i-1}}$, then we must have $\pRingH_{\square} \subseteq \ideal[L]{}_{\pB{i}}$, which would imply that all spaces of the above complex are identically $0$. Therefore, the only non-trivial case to consider is when $\pRingH_{\square} \supseteq \ideal[L]{}_{\pB{i-1}}$. Explicitly, the non-trivial cases to analyze yield the following complexes:
	\begin{equation*}
	\begin{tikzcd}[column sep=small]
		0 \arrow[r] & \ideal[L]{}_{\pB{i-1}} / \ideal[L]{}_{\pB{i}} \arrow[r] & \pRingH_{\face} / \ideal[L]{}_{\pB{i}} \arrow[r] & \pRingH_{\face} / \ideal[L]{}_{\pB{i-1}} \arrow[r] & 0\;, \qquad  (a)\\
		0 \arrow[r] & \ideal[L]{}_{\pB{i-1}}/\left( \ideal_\edge \cap \ideal[L]{}_{\pB{i-1}} + \ideal[L]{}_{\pB{i}} \right) \arrow[r] & \pRingH_{\edge}/\left( \ideal_\edge + \ideal[L]{}_{\pB{i}} \right) \arrow[r] & \pRingH_{\edge}/\left( \ideal_\edge + \ideal[L]{}_{\pB{i-1}} \right) \arrow[r] & 0\;,\qquad (b)\\
		0 \arrow[r] & \ideal[L]{}_{\pB{i-1}}/\left( \ideal_\vertex \cap \ideal[L]{}_{\pB{i-1}} + \ideal[L]{}_{\pB{i}} \right) \arrow[r] & \pRingH_{\vertex}/\left( \ideal_\vertex + \ideal[L]{}_{\pB{i}} \right) \arrow[r] & \pRingH_{\vertex}/\left( \ideal_\vertex + \ideal[L]{}_{\pB{i-1}} \right) \arrow[r] & 0\;. \qquad (c)
	\end{tikzcd}
	\end{equation*}
	It is easy to see that $(a)$ is a short exact sequence. The same observation follows for $(b)$ and $(c)$ since, for $\square \in \tmeshE \cup \tmeshV$, by definition we have
	\begin{equation*}
		\ideal[L]{}_{\pB{i-1}} \cap \left( \ideal_\square + \ideal[L]{}_{\pB{i}} \right) = \left( \ideal_\square \cap \ideal[L]{}_{\pB{i-1}} + \ideal[L]{}_{\pB{i}} \right)\;,
	\end{equation*}
	which transforms $(b)$ and $(c)$ into
	\begin{equation*}
	\begin{tikzcd}
		0 \arrow[r] & \left(\ideal_\square + \ideal[L]{}_{\pB{i-1}}\right)/\left( \ideal_\square + \ideal[L]{}_{\pB{i}} \right) \arrow[r] & \pRingH_{\square}/\left( \ideal_\square + \ideal[L]{}_{\pB{i}} \right) \arrow[r] & \pRingH_{\square}/\left( \ideal_\square + \ideal[L]{}_{\pB{i-1}} \right) \arrow[r] & 0\;.
	\end{tikzcd}
	\end{equation*}
\end{proof}

As stated above, $\quotientComplex$ can be studied by studying the complexes $\degBa{\quotientComplex}{i}$, $i = 1, 2, \dots, \nlevels+1$.
We do so by analyzing the following short exact sequence of chain complexes for each $i$.
\begin{equation}
	\begin{tikzcd}
		~ & ~ & 0 \arrow[d,""] & 0 \arrow[d,""] & ~ \\
		\degBa{\idealComplex}{i}~:~ & 0\arrow[r] \arrow[d]& \moplus\limits_{\edge \in \tmeshInteriorComponentE{,\pB{i}}{}} \edgeE{} \degB{\ideal}{\edge}{i} \arrow[r] \arrow[d] & \moplus\limits_{\vertex \in \tmeshInteriorComponentV{,\pB{i}}{}} \vertexE{} \degB{\ideal}{\vertex}{i} \arrow[r] \arrow[d] & 0 \\
		\degBa{\constantComplex}{i}~:~ & \moplus\limits_{\face \in \tmeshComponentF{,\pB{i}}{}} \faceE{} \degB{\pRingH}{\face}{i} \arrow[r] \arrow[d] & \moplus\limits_{\edge \in \tmeshInteriorComponentE{, \pB{i}}{}} \edgeE{} \degB{\pRingH}{\edge}{i} \arrow[r] \arrow[d] & \moplus\limits_{\vertex \in \tmeshInteriorComponentV{, \pB{i}}{}} \vertexE{} \degB{\pRingH}{\vertex}{i} \arrow[r] \arrow[d] & 0 \\
		\degBa{\quotientComplex}{i}~:~ & \moplus\limits_{\face \in \tmeshComponentF{, \pB{i}}{}} \faceE{} \degB{\pRingH}{\face}{i} \arrow[r] & \moplus\limits_{\edge \in \tmeshInteriorComponentE{, \pB{i}}{}} \edgeE{} \degB{\pRingH}{\edge}{i}/\degB{\ideal}{\edge}{i} \arrow[r] \arrow[d] & \moplus\limits_{\vertex \in \tmeshInteriorComponentV{, \pB{i}}{}} \vertexE{} \degB{\pRingH}{\vertex}{i}/\degB{\ideal}{\vertex}{i} \arrow[r] \arrow[d] & 0 \\
		~ & ~ & 0 & 0 & ~
	\end{tikzcd}
\label{eq:base_complex}
\end{equation}
Note that the morphisms above are obtained in the obvious way by composing (restrictions of) $\boundary$ with quotient maps.
In Equation \eqref{eq:base_complex}, $\tmeshComponentF{, \pB{i}}{}, \tmeshInteriorComponentE{, \pB{i}}{}$ and $\tmeshInteriorComponentV{, \pB{i}}{}$ are the active components of the mesh with respect to the index $i$; this notion was hinted at in the proof of Proposition \ref{prop:exact_stair} and is defined next.
\begin{definition}[Active T-mesh]\label{def:activeTmesh}
	For $i = 1, 2, \dots, \nlevels+1$, the active T-mesh $\tmeshComponent{\pB{i}}{}$ is composed of
	\begin{itemize}
		\item $\tmeshComponentF{,\pB{i}}{} \subseteq \tmeshF$ such that $\face \in \tmeshComponentF{,\pB{i}}{} \xLeftrightarrow{\text{def.}} \pRingH_\face \supseteq \ideal[L]{}_{\pB{i-1}}$;
		\item $\tmeshComponentE{,\pB{i}}{} \subseteq \tmeshE$ such that $\edge \in \tmeshComponentE{,\pB{i}}{} \xLeftrightarrow{\text{def.}} \pRingH_\edge \supseteq \ideal[L]{}_{\pB{i-1}}$; and,
		\item $\tmeshComponentV{,\pB{i}}{} \subseteq \tmeshV$ such that $\vertex \in \tmeshComponentV{,\pB{i}}{} \xLeftrightarrow{\text{def.}} \pRingH_\vertex \supseteq \ideal[L]{}_{\pB{i-1}}$.
	\end{itemize}
	The domain of this active T-mesh, $\domainComponent{\pB{i}}{}$, is defined to be $\cup_{\face\in\tmeshComponentF{,\pB{i}}{}}\face \subset \RR^2$.
\end{definition}
The symbols for interior edges, vertices, horizontal and vertical edges etc. are all appended with a subscript of $\pB{i}$ when talking about the active mesh $\tmeshComponent{\pB{i}}{}$; see Equation \eqref{eq:base_complex}. Note that ``interior'' will always mean interior with respect to $\domain$. It should be noted that $\tmeshComponentE{,\pB{i}}{}$ is exactly the set of edges that are contained in $\cup_{\face\in\tmeshComponentF{,\pB{i}}{}}\boundary\face$. Similarly, $\tmeshComponentV{,\pB{i}}{}$ is exactly the set of vertices that are contained in $\cup_{\edge\in\tmeshComponentE{,\pB{i}}{}}\boundary\edge$.

\begin{remark}
	Note that Equation \eqref{eq:deg_def_diff} may introduce more active meshes than strictly necessary.
	However, we choose the degree-deficit sequence in compliance with Equation \eqref{eq:deg_def_diff} because it simplifies the analysis later on.
	In particular, the results that are affected by this simplification are Lemmas \ref{lem:new_relations} and \ref{lem:new_relations_transversal} (and those that depend on these lemmas).
\end{remark}

\begin{example}\label{ex:tmesh_activeMesh}
	Consider the setup in Example \ref{ex:tmesh_degreeDist}, and let us choose the shorter sequence of degree deficits provided therein.
	Then, the associated active meshes with respect to $i = 1, 2, 3$ are shown in Figure \ref{fig:tmesh_levels}. The bottom, middle and top layers correspond to $\tmeshComponent{\pB{3}}{}$, $\tmeshComponent{\pB{2}}{}$ and $\tmeshComponent{\pB{1}}{}$, respectively. In other words, the layer corresponding to $\tmeshComponent{\pB{i}}{}$ is such that for the faces $\face$, edges $\edge$ and vertices $\vertex$ contained in it, we have the containment $\pRingH_\square \supseteq \ideal[L]{}_{\pB{i-1}}$, $\square \in \{ \face, \edge, \vertex \}$.
\end{example}

\begin{figure}
	\begin{subfigure}{\linewidth}
		\centering
		\hspace{2.2cm}
	\tikzsetnextfilename{./tikz/images/tmesh_levels}%
	\input{tmesh_levels}%

	\end{subfigure}
	\caption{Active regions of the T-mesh $\tmesh$ from Figure \ref{fig:tmesh_example} for different indices $i \in \{1, 2, 3\}$; see Example \ref{ex:tmesh_activeMesh} for reference. Generators of $H_0(\domainComponent{\pB{i}}{}, \boundary\domainComponent{\pB{i}}{}\cap\boundary\domain)$ have been shown as gray disks.
		Boundaries of the meshes have been emphasized in bold and solid lines correspond to the active boundary, i.e., $\boundary\domain \cap \domainComponent{\pB{i}}{}$.
	}
	\label{fig:tmesh_levels}
\end{figure}
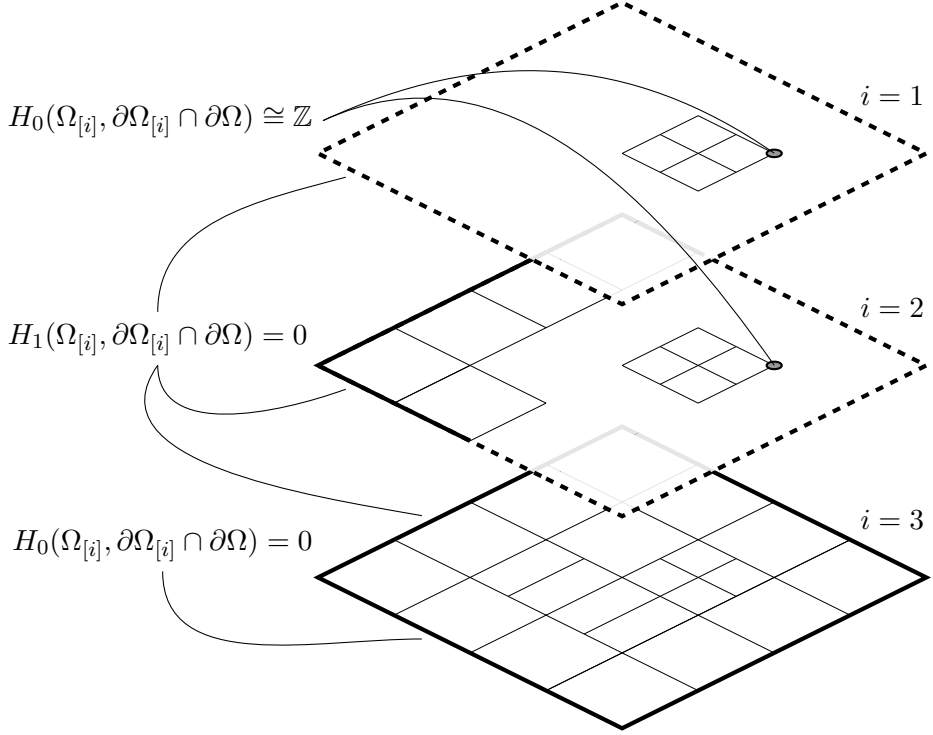

\subsection{Summary of approach}\label{ss:approach_summary}
Given $\abdegree \in \ZZP^2$, let $\euler{\mathcal{A}}_\abdegree$ be the Euler characteristic of the $\abdegree$-th graded piece of the complex $\mathcal{A}~:~0 \rightarrow A_{k} \rightarrow A_{k-1} \rightarrow \cdots \rightarrow A_0 \rightarrow 0$,
\begin{equation}
\begin{split}
	\euler{\mathcal{A}}_\abdegree = \sum_{j=0}^{k} (-1)^j \dimwp{H_j(\mathcal{A})}_\abdegree = \sum_{j=0}^{k} (-1)^j \dimwp{A_j}_\abdegree\;.
\end{split}
\label{eq:euler}
\end{equation}
Then the Euler characteristic of $\quotientComplex = \degAa{\quotientComplex}{\nlevels+1}$ helps quantify the homological contribution to the dimension of $\splSpaceH^\bsmooth_\bdegree$.
\begin{definition}[Homological contribution to dimension]\label{def:homological_contribution}
	Given $\abdegree \in \ZZP^2$, we define the homological contribution to the dimension of $\splSpaceH^\bsmooth_\bdegree$ in bi-degree $\abdegree$ as
	\begin{equation*}
	\begin{split}
		\homoDim^\bsmooth_{\bdegree, \abdegree} = 	\dimwp{\splSpaceH^\bsmooth_\bdegree}_\abdegree -\euler{\quotientComplex}_\abdegree\;,
	\end{split}
	\end{equation*}
so that we have 
$	\dimwp{\splSpaceH^\bsmooth_\bdegree}_\abdegree = \euler{\quotientComplex}_\abdegree + \homoDim^\bsmooth_{\bdegree, \abdegree}$.
\end{definition}

It will be shown in Section \ref{ss:vec_dim_formulas} that the Euler characteristic of $\quotientComplex$ is computable exactly using the rightmost expression in Equation \eqref{eq:euler}. This leaves only the computation (or estimation) of $\homoDim^\bsmooth_{\bdegree, \abdegree}$ for the purpose of determining (bounds on) the dimension of $\splSpaceH^\bsmooth_\bdegree$. We approach this task as follows. First, using the short exact sequence from Equation \eqref{eq:exact_stair}, we build the long exact sequence of homology modules
\begin{equation}
	\begin{tikzcd}[column sep=tiny]
		0 \arrow[r] & H_2\left( \degBa{\quotientComplex}{i} \right) \arrow[r] & H_2\left( \degAa{\quotientComplex}{i} \right)  \arrow[r] & H_2\left( \degAa{\quotientComplex}{i-1} \right) \arrow[r, "\hat{\boundary}_i"] &
	H_1\left( \degBa{\quotientComplex}{i} \right) \arrow[r] & \cdots \arrow[r] & H_0\left( \degAa{\quotientComplex}{i-1} \right) \arrow[r] & 0
	\end{tikzcd}
	\label{eq:long_ex_seq_homology}
\end{equation}
with $\hat{\boundary}_1 \equiv 0$. The above implies that
\begin{equation*}
	\euler{\degAa{\quotientComplex}{i}}_\abdegree = \euler{\degBa{\quotientComplex}{i}}_\abdegree + \euler{\degAa{\quotientComplex}{i-1}}_\abdegree\;.
\end{equation*}
Summing up the terms for $i=1\ldots, \nlevels+1$, we obtain
\begin{equation*}
	\begin{split}
		\euler{{\quotientComplex}}_\abdegree = \euler{\degAa{\quotientComplex}{\nlevels+1}}_\abdegree = \sum_{i = 1}^{\nlevels+1} \euler{\degBa{\quotientComplex}{i}}_\abdegree
	\end{split}
\end{equation*}
since $\degAa{\quotientComplex}{0}=0$ and $\degAa{\quotientComplex}{\nlevels+1}={\quotientComplex}$.
That is, the Euler characteristic of $\quotientComplex$ decomposes additively in the above manner.
The next results simply states a similar additive decomposition does not hold for the $2$-homologies.

\begin{lemma}\label{lem:degree_decomposition_inexactness}
	\begin{equation*}
		\dimwp{\splSpaceH^\bsmooth_\bdegree}_\abdegree = \sum_{i = 1}^{\nlevels+1} \dimwp{H_2\left(\degBa{\quotientComplex}{i}\right)}_\abdegree - \dimwp{\im {\hat{\boundary}_i}}_\abdegree\;.
	\end{equation*}
\end{lemma}
\begin{proof}
	From Equation \eqref{eq:long_ex_seq_homology}, we have the following exact sequence for all $i \ge 1$,
	\begin{equation*}
		\begin{tikzcd}[column sep=normal]
			0 \arrow[r] & H_2\left( \degBa{\quotientComplex}{i} \right) \arrow[r] & H_2\left( \degAa{\quotientComplex}{i} \right)  \arrow[r] & H_2\left( \degAa{\quotientComplex}{i-1} \right) \arrow[r, "\hat{\boundary}_i"] &
			\im{\hat{\boundary}_i} \arrow[r] & 0\;.
		\end{tikzcd}
	\end{equation*}
	The above implies that
	\begin{equation*}
		\dimwp{H_2\left( \degAa{\quotientComplex}{i} \right)}_\abdegree = \dimwp{H_2\left( \degAa{\quotientComplex}{i-1} \right)}_\abdegree + \dimwp{H_2\left(\degBa{\quotientComplex}{i}\right)}_\abdegree - \dimwp{\im{\hat{\boundary}_i}}_\abdegree\;.
	\end{equation*}
	The claim follows upon summing up the terms for $i=1\ldots, \nlevels+1$.
\end{proof}

Using Lemma \ref{lem:degree_decomposition_inexactness}, we can simplify the expression for $\euler{{\quotientComplex}}_\abdegree$,
\begin{equation*}
	\begin{split}
          \euler{\quotientComplex}_\abdegree &= \sum_{i = 1}^{\nlevels+1} \left(\dimwp{H_2\left(  \degBa{\quotientComplex}{i} \right)}_\abdegree
          -\dimwp{H_1\left(  \degBa{\quotientComplex}{i} \right)}_\abdegree + \dimwp{H_0\left(  \degBa{\quotientComplex}{i} \right)}_\abdegree\right)\;,\\
          & = \dim (\splSpaceH^\bsmooth_{\bdegree})_{\abdegree} +\sum_{i=1}^{\nlevels+1}\left(\dimwp{\im{\hat{\boundary}_i}}_\abdegree -\dimwp{H_1\left(  \degBa{\quotientComplex}{i} \right)}_\abdegree + \dimwp{H_0\left(  \degBa{\quotientComplex}{i} \right)}_\abdegree\right)\;.\\
	\end{split}
\end{equation*}
A final simplification is afforded by the following assumption.
\begin{assumption}\label{ass:no_holes}
	 All complexes $\degBa{\constantComplex}{i}$ are without holes, i.e., $H_1\left( \degBa{\constantComplex}{i}  \right) = 0$ for all $i$.
\end{assumption}
\begin{proposition}\label{prop:homo_dim}
	\begin{equation*}
		\begin{split}
			\homoDim^\bsmooth_{\bdegree, \abdegree}&:=\dimwp{\splSpaceH^\bsmooth_{\bdegree}}_{\abdegree} - \euler{\quotientComplex}_\abdegree\;,\\
			&=  \sum_{i = 1}^{\nlevels+1}\dimwp{H_0\left(  \degBa{\idealComplex}{i} \right)}_\abdegree - \dimwp{H_0\left(  \degBa{\constantComplex}{i} \right)}_\abdegree - \dimwp{\im{\hat{\boundary}_i}}_\abdegree\;.
		\end{split}
	\end{equation*}
\end{proposition}
\begin{proof}
	The proof follows from Lemma \ref{lem:degree_decomposition_inexactness} and the diagram in Equation \eqref{eq:base_complex}. Indeed, following Assumption \ref{ass:no_holes}, we obtain the long exact sequence of homology modules
	\begin{equation*}
	\begin{tikzcd}
		0 \arrow[r] & H_1\left( \degBa{\quotientComplex}{i} \right) \arrow[r] & H_0\left( \degBa{\idealComplex}{i} \right)  \arrow[r] & H_0\left( \degBa{\constantComplex}{i} \right) \arrow[r] &
		H_0\left( \degBa{\quotientComplex}{i} \right) \arrow[r] & 0\;,
	\end{tikzcd}
	\end{equation*}
	which yields the claim when combined with Equation \eqref{eq:euler}.
\end{proof}
Therefore, the final problem is the computation (or estimation) of the dimension of ${\im{\hat{\boundary}_i}}_\abdegree$ and the difference in the dimensions of ${H_0\left(  \degBa{\idealComplex}{i} \right)}_\abdegree$ and ${H_0\left( \degBa{\constantComplex}{i} \right)}_\abdegree$  for all $i$.
The next example shows that in general $\dimwp{\im{\hat{\boundary}_i}}_\abdegree$ is not equal to zero.
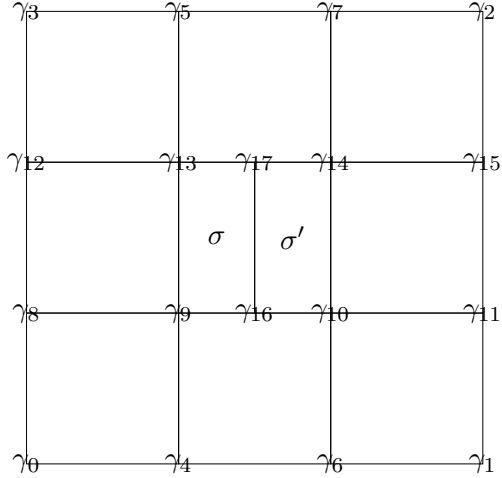
\begin{figure}
	\centering
	\tikzsetnextfilename{./tikz/images/tmesh_counterExample}%
	\input{tmesh_counterExample}%

	\caption{A mesh that serves to counter the expectation that the spline space dimension can be additively decomposed in the same manner as the Euler characteristic of $\quotientComplex$; see Example \ref{ex:dim_additive_counter} for details.}
	\label{fig:tmesh_counter}
\end{figure}
\begin{example}\label{ex:dim_additive_counter}
	Consider the T-mesh in Figure \ref{fig:tmesh_counter}. Assume that the degree deficit on all faces touching the boundary is $(1, 1)$, and on the remaining faces is $(0, 0)$. 
	Choose the asssociated degree-deficit sequence to be $\bdegree_0 = (0,0) < \bdegree_1 = (1, 1)$.
	Let us ask for $C^2$ smoothness across all edges, and let $\abdegree = (5, 5)$. Then, based on the definitions, it can be checked that
	\begin{equation*}
		\dimwp{\splSpaceH^\bsmooth_{\bdegree}}_\abdegree = 81\;.
	\end{equation*}
	However, we can also compute (using Macaulay2, for instance) that
	\begin{equation*}
		\dimwp{H_2\left( \degBa{\quotientComplex}{2} \right)}_\abdegree = 81\;, \qquad
		\dimwp{H_2\left( \degBa{\quotientComplex}{1} \right)}_\abdegree = 1\;.
	\end{equation*}
	Thus, $\dimwp{\splSpaceH^\bsmooth_{\bdegree}}_\abdegree \neq \sum_{i=1}^2 \dimwp{H_2\left( \degBa{\quotientComplex}{i} \right)}_\abdegree$.
	Equivalently, we have that $\dimwp{\im{\hat{\boundary}_2}} = 1$.

	In particular, the spline generating $H_2\left( \degBa{\quotientComplex}{2} \right)$ in degree $\abdegree$ is $\elem{\face} p + \elem{\face'} q$ where
	\begin{equation*}
		p = t^5 u(3s - u)^3 (9s-5u)\;,\qquad
		q = t^5 u(3s - 2u)^3 (9s-4u)\;.
	\end{equation*}
	It can be checked that $p -q = 27t^5u^2(2s-u)^3 \in \ideal_\edge$, $\edge = \vertex_{16}\vertex_{17}$, and that
	$p, q \in \ideal_\edge + \left( uv \right)$
	where $\edge$ is the edge $\vertex_{13}\vertex_{17}$ or $\vertex_{9}\vertex_{16}$.
\end{example}

\subsection{Dimension formulas for relevant vector spaces}\label{ss:vec_dim_formulas}
Before proceeding, we present combinatorial formulas for the dimensions of the different vector spaces that have appeared so far and are relevant for our analysis.
In the following, $a_+ := \max\{a, 0\}$ for $a \in \ZZP$, and $\abdegree = (\adegree, \adegree') \in \ZZP^2$.
We will also need the bi-smoothness associated to an interior edge $\edge$,
\begin{equation}
\bideg{\edge} :=
\begin{dcases}
(0, \bsmooth(\edge) + 1)\;, & \edge \in \tmeshInteriorEH\\
(\bsmooth(\edge) + 1, 0)\;, & \edge \in \tmeshInteriorEV
\end{dcases}\;.
\end{equation}
For an interior vertex $\vertex \in \edge_h \cap \edge_v$, $(\edge_h, \edge_v) \in \tmeshInteriorEH \times \tmeshInteriorEV$, we define $\bideg{\vertex} := \bideg{\edge_h} + \bideg{\edge_v}$.
\begin{proposition}\label{prop:vec_space_dim}
	Let $\edge \in \tmeshInteriorComponentE{,\pB{i}}{}$, and let $\vertex \in \tmeshInteriorComponentV{,\pB{i}}{}$ such that $\vertex = \edge_h \cap \edge_v$, $(\edge_h, \edge_v) \in \tmeshInteriorEH \times \tmeshInteriorEV$.
	Then the following hold, where for $(c)-(e)$ we assume that $1 \leq i \leq \nlevels+1$.
	\begin{align*}
	&\dimwp{\pRingH(-j, -k)}_\abdegree = 
	\left( \adegree-j+1\right)_+\times\left( \adegree'-k+1\right)_+\;. \tag*{$(a)$}\\
	&\dimwp{\shiftideal{L}{\pB{i}}{-j,-k}}_\abdegree = 
	\begin{dcases}
	\dimwp{\pRingH(-j, -k)}_{\abdegree-\bdegree_i}, & 0 \leq i \leq \nlevels\\
	0, & i = \nlevels+1
	\end{dcases}\;. \tag*{$(b)$}\\
	&\dimwp{\shiftideal{M}{\pB{i}}{ -j, -k}}_\abdegree = \dimwp{\shiftideal{L}{\pB{i-1}}{ -j, -k}}_\abdegree - \dimwp{\shiftideal{L}{\pB{i}}{ -j, -k}}_\abdegree\;. \tag*{$(c)$}\\
	&\dimwp{\degB{\ideal}{\edge}{i}}_\abdegree = \dimwp{\pRingH(-\bideg{\edge}-\bdegree_{i-1})}_\abdegree -\dimwp{\pRingH(-\bideg{\edge}-\bdegree_{i})}_\abdegree\;. \tag*{$(d)$}\\
	&\dimwp{\degB{\ideal}{\vertex}{i}}_\abdegree = \dimwp{\pRingH(-\bideg{\edge_h}-\bdegree_{i-1})}_\abdegree + \dimwp{\pRingH(-\bideg{\edge_v}-\bdegree_{i-1})}_\abdegree\\
	&\qquad\qquad\qquad + \dimwp{\pRingH(-\bideg{\vertex}-\bdegree_{i})}_\abdegree
	- \dimwp{\pRingH(-\bideg{\vertex}-\bdegree_{i-1})}_\abdegree \\
	&\qquad\qquad\qquad - \dimwp{\pRingH(-\bideg{\edge_h}-\bdegree_{i})}_\abdegree
	- \dimwp{\pRingH(-\bideg{\edge_v}-\bdegree_{i})}_\abdegree\;.\tag*{$(e)$}
	\end{align*}
\end{proposition}
\begin{proof}
	\renewcommand{\arraystretch}{1.5}
	Claims made in parts $(a)$, $(b)$ and $(c)$ follow by definition.
	Parts $(d)$ and $(e)$ follow from the following exact sequences, respectively,
	\begin{equation*}
	\begin{tikzcd}[column sep=small]
	\pRingH(-\bideg{\edge}-\bdegree_{i})  \arrow[r] &
	\begin{array}{c}
	\pRingH(-\bideg{\edge}-\bdegree_{i-1}) \\[-2mm]
	\moplus\\[-2mm]
	\pRingH(-\bdegree_{i})
	\end{array}
	\arrow[r] & \ideal_\edge\cap\ideal[L]{}_{\pB{i-1}} + \ideal[L]{}_{\pB{i}}\;,
	\end{tikzcd}
	\end{equation*}
	\begin{equation*}
	\begin{tikzcd}[column sep=small]
	\pRingH(-\bideg{\vertex}-\bdegree_{i})  \arrow[r] &
	\begin{array}{c}
	\pRingH(-\bideg{\edge_h}-\bdegree_{i})\\[-2mm]
	\moplus\\[-2mm]
	\pRingH(-\bideg{\edge_v}-\bdegree_{i})\\[-2mm]
	\moplus\\[-2mm]
	\pRingH(-\bideg{\vertex}-\bdegree_{i-1})
	\end{array}
	\arrow[r] &
	\begin{array}{c}
	\pRingH(-\bideg{\edge_h}-\bdegree_{i-1}) \\[-2mm]
	\moplus\\[-2mm]
	\pRingH(-\bideg{\edge_v}-\bdegree_{i-1}) \\[-2mm]
	\moplus\\[-2mm]
	\pRingH(-\bdegree_{i})
	\end{array}
	\arrow[r] &
	\ideal_\vertex\cap\ideal[L]{}_{\pB{i-1}} + \ideal[L]{}_{\pB{i}}\;.
	\end{tikzcd}
	\end{equation*}
\end{proof}

We can simplify the expression in Proposition \ref{prop:vec_space_dim}(e) for special choices of $\bdegree, \bsmooth, \abdegree$.
\begin{corollary}\label{cor:vertex_ideal_dim}
	Let $1 \leq i \leq \nlevels$ and $\vertex \in \tmeshInteriorComponentV{,\pB{i}}{}$ such that $\vertex = \edge_h \cap \edge_v$, $(\edge_h, \edge_v) \in \tmeshInteriorEH \times \tmeshInteriorEV$.
	\begin{enumerate}[label=(\alph*)]
		\item If $\abdegree - \bdegree_{i}$ is entry-wise greater than or equal to $(\bsmooth(\edge_v), \bsmooth(\edge_h))$, then
		\begin{equation*}
		\begin{split}
			\dimwp{\degB{\ideal}{\vertex}{i}}_\abdegree &= (\adegree-\degree_{(i-1)1}+1)\degreev_{i2}+(\adegree'-\degree_{(i-1)2}+1)\degreev_{i1} - \degreev_{i1}\degreev_{i2}\\
			&= \dimwp{\ideal[M]{}_{\pB{i}}}_\abdegree = \dimwp{\ideal[L]{}_{\pB{i-1}}/\ideal[L]{}_{\pB{i}}}_\abdegree\;.
		\end{split}
		\end{equation*}
		
		\item If $\abdegree - \bdegree_{i-1}$ is entry-wise smaller than or equal to $(\bsmooth(\edge_v), \bsmooth(\edge_h))$, then
		\begin{equation*}
		\begin{split}
			\dimwp{\degB{\ideal}{\vertex}{i}}_\abdegree &= 0\;.
		\end{split}
		\end{equation*}
	\end{enumerate}
\end{corollary}

The next result follows from the definition of the quotient modules $\shiftideal{M}{\pB{i}}{ -j, -k}$ and Equation \ref{eq:deg_def_diff}.
\begin{lemma}
	The following hold for $\mbf{b} \in \ZZP^2$, where $1 \leq i \leq \nlevels$ for $(a)$ and $1 \leq i \leq \nlevels$ for $(b)-(d)$.
	\begin{align*}
		&\dimwp{\mbf{s}^{\mbf{b}}\shiftideal{M}{\pB{i}}{ -j, -k}}_\abdegree = \dimwp{\shiftideal{M}{\pB{i}}{ -j, -k}}_{\abdegree-\mbf{b}}\;. \tag*{$(a)$}\\
		&\dimwp{u\shiftideal{M}{\pB{i}}{ -j, -k}}_\abdegree = \dimwp{u\shiftideal{L}{\pB{i-1}}{ -j, -k}}_{\abdegree}\\
		&\qquad\qquad\qquad\qquad\qquad - \dimwp{u^{(1-\degreev_{i1})_+}\shiftideal{L}{\pB{i}}{ -j, -k}}_{\abdegree}\;. \tag*{$(b)$}\\
		&\dimwp{v\shiftideal{M}{\pB{i}}{ -j, -k}}_\abdegree = \dimwp{v\shiftideal{L}{\pB{i-1}}{ -j, -k}}_{\abdegree}\\
		&\qquad\qquad\qquad\qquad\qquad - \dimwp{v^{(1-\degreev_{i2})_+}\shiftideal{L}{\pB{i}}{ -j, -k}}_{\abdegree}\;. \tag*{$(c)$}\\
		&\dimwp{uv\shiftideal{M}{\pB{i}}{ -j, -k}}_\abdegree = 0\;. \tag*{$(d)$}
	\end{align*}
\end{lemma}

The next result has been adapted from \cite{mourrain2014dimension} in a form relevant for our analysis; see \cite{mourrain2014dimension} for its proof.
\begin{proposition}\label{prop:space_sum_dim}
	Consider $\mbf{b} \in \ZZP^2$, $l$ distinct numbers $a_1, \dots, a_l \in \RR$, and $d_1, \dots, d_l \in \ZZP$. Denote with $f_k$ the linear polynomials $t-a_kv$.
	Then, we have
	\begin{align*}
		&\dimwp{\sum_{k=1}^l f_k^{d_k}\pRingH(0, -d_k)}_\abdegree
		= (m+1) \times \min\left(m'+1,~\sum_{k=1}^l m'-d_k+1\right)\;. \tag*{$(a)$}\\
		&\dimwp{\sum_{k=1}^l f_k^{d_k}\ideal[L]{}_{\pB{i}}(-\mbf{b}-(0, d_k))}_\abdegree
		= \dimwp{\sum_{k=1}^l f_k^{d_k}\pRingH(0, -d_k)}_{\abdegree-\bdegree_i - \mbf{b}}\;. \tag*{$(b)$}\\
		&\dimwp{\sum_{k=1}^l f_k^{d_k}\ideal[M]{}_{\pB{i}}(-\mbf{b}-(0, d_k))}_\abdegree
		=\\
		&\qquad\qquad\qquad\qquad\dimwp{\left(\sum_{k=1}^l f_k^{d_k}\ideal[L]{}_{\pB{i-1}}(-\mbf{b}-(0, d_k))\right)\big/\ideal[L]{}_{\pB{i}}(-\mbf{b}-(0, d_k))}_\abdegree\;. \tag*{$(c)$}
	\end{align*}
	Symmetric claims can be made if the linear polynomials are instead chosen to be $s - a_k u$.
\end{proposition}

%% file: tmesh_levels.tex
\begin{tikzpicture}[remember picture]
	\begin{scope}[yshift=-80,every node/.append style={
		yslant=0.5,xslant=-1},yslant=0.5,xslant=-1]
		\foreach \i in {0,...,4}{
			\foreach \j in {1,...,4}{
				\node (a0\i\j) at (\i,\j) {};
			}
		}
		\foreach \i in {1,...,5}{
			\foreach \j in {0,...,1}{
				\node (b0\i\j) at (\i-1,\j) {};
			}
		}
		
		\draw[step=1,eThickness] (a001) grid (a044);
		\draw[step=1,eThickness] (b010) grid (b051);
		
		\draw[eThickness] ($(a021)!0.5!(a031)$) -- ($(a022)!0.5!(a032)$);
		\draw[eThickness] ($(a012)!0.5!(a013)$) -- ($(a022)!0.5!(a023)$);
		\draw[eThickness] ($(a011)!0.5!(a012)$) -- ($(a031)!0.5!(a032)$);
		
		\draw[bThickness] (a001.center) -- (b010.center) -- (b050.center) -- (b051.center) -- (a041.center) -- (a044.center) -- (a004.center) -- (a001.center);
	\end{scope}
	
	\begin{scope}[yshift=0,every node/.append style={
		yslant=0.5,xslant=-1},yslant=0.5,xslant=-1]
		\foreach \i in {0,...,4}{
			\foreach \j in {1,...,4}{
				\node (a1\i\j) at (\i,\j) {};
			}
		}
		\foreach \i in {1,...,5}{
			\foreach \j in {0,...,1}{
				\node (b1\i\j) at (\i-1,\j) {};
			}
		}
	
		\fill[white,fill opacity=0.9] (a101) rectangle (a144);
		\fill[white,fill opacity=0.9] (b110) rectangle (b151);
		\draw[dashed,bThickness] (a101.center) -- (b110.center) -- (b150.center) -- (b151.center) -- (a141.center) -- (a144.center) -- (a104.center) -- (a101.center);
		\draw[bThickness] (a102.center) -- (a104.center) -- (a144.center) -- (a143.center);
		
		\draw[step=1,eThickness] (a103) grid (a144);
		\draw[step=1,eThickness] (a102) grid (a113);
		\draw[step=1,eThickness] (a121) grid (a132);
		
		\draw[eThickness] ($(a121)!0.5!(a131)$) -- ($(a122)!0.5!(a132)$);
		\draw[eThickness] ($(a121)!0.5!(a122)$) -- ($(a131)!0.5!(a132)$);
		
		\draw[thick,fill=myGray2] (a131) circle (2pt);
	\end{scope}
	
	\begin{scope}[yshift=80,every node/.append style={
		yslant=0.5,xslant=-1},yslant=0.5,xslant=-1]
		\foreach \i in {0,...,4}{
			\foreach \j in {1,...,4}{
				\node (a2\i\j) at (\i,\j) {};
			}
		}
		\foreach \i in {1,...,5}{
			\foreach \j in {0,...,1}{
				\node (b2\i\j) at (\i-1,\j) {};
			}
		}
	
		\fill[white,fill opacity=0.9] (a201) rectangle (a244);
		\fill[white,fill opacity=0.9] (b210) rectangle (b251);
		\draw[dashed,bThickness] (a201.center) -- (b210.center) -- (b250.center) -- (b251.center) -- (a241.center) -- (a244.center) -- (a204.center) -- (a201.center);
		
		\draw[step=1,eThickness] (a221) grid (a232);
		
		\draw[eThickness] ($(a221)!0.5!(a231)$) -- ($(a222)!0.5!(a232)$);
		\draw[eThickness] ($(a221)!0.5!(a222)$) -- ($(a231)!0.5!(a232)$);
		
		\draw[thick,fill=myGray2] (a231) circle (2pt);
	\end{scope}

\end{tikzpicture}
\begin{tikzpicture}[remember picture,overlay]
	\node[above left] (h1) at (a104) {$H_1(\domain_{[i]}, \boundary\domain_{[i]}\cap\boundary\domain) = 0$};
	\draw[thin] (h1.north) to[out=90,in=-165] ($(a204.west)!0.5!(a203.west)$);
	\draw[thin] (h1.south) to[out=-90,in=200] ($(a104.west)!0.5!(a103.west)$);
	\draw[thin] (h1.south) to[out=-125,in=170] ($(a014.north)!0.5!(a024.north)$);
	\node[left = 0.8cm of a214] (h012) {$H_0(\domain_{[i]}, \boundary\domain_{[i]}\cap\boundary\domain) \isomorphic \ZZ$};
	\draw[thin] (h012.east) to[out=25,in=145] (a231.center);
	\draw[thin] (h012.east) to[out=25,in=125] (a131.center);
	\node[left = 0.8cm of a014] (h00) {$H_0(\domain_{[i]}, \boundary\domain_{[i]}\cap\boundary\domain) = 0$};
	\draw[thin] (h00.south) to[out=-90,in=-180] ($(a002.west)!0.5!(a003.west)$);
	\node[above right] (i2) at (a241) {$i = 1$};
	\node[above right] (i1) at (a141) {$i = 2$};
	\node[above right] (i0) at (a041) {$i = 3$};
\end{tikzpicture}

%% file: tmesh_counterExample.tex
\begin{tikzpicture}[scale=1.5]
	 \begin{scope}
		\node (v0) at (0.0000000000000000,0.0000000000000000) {$\vertex_{0}$};
		\node (v1) at (4.0000000000000000,0.0000000000000000) {$\vertex_{1}$};
		\node (v2) at (4.0000000000000000,4.0000000000000000) {$\vertex_{2}$};
		\node (v3) at (0.0000000000000000,4.0000000000000000) {$\vertex_{3}$};
		\node (v4) at (1.3333333333333333,0.0000000000000000) {$\vertex_{4}$};
		\node (v5) at (1.3333333333333333,4.0000000000000000) {$\vertex_{5}$};
		\node (v6) at (2.6666666666666665,0.0000000000000000) {$\vertex_{6}$};
		\node (v7) at (2.6666666666666665,4.0000000000000000) {$\vertex_{7}$};
		\node (v8) at (0.0000000000000000,1.3333333333333333) {$\vertex_{8}$};
		\node (v9) at (1.3333333333333333,1.3333333333333333) {$\vertex_{9}$};
		\node (v10) at (2.6666666666666665,1.3333333333333333) {$\vertex_{10}$};
		\node (v11) at (4.0000000000000000,1.3333333333333333) {$\vertex_{11}$};
		\node (v12) at (0.0000000000000000,2.6666666666666665) {$\vertex_{12}$};
		\node (v13) at (1.3333333333333333,2.6666666666666665) {$\vertex_{13}$};
		\node (v14) at (2.6666666666666665,2.6666666666666665) {$\vertex_{14}$};
		\node (v15) at (4.0000000000000000,2.6666666666666665) {$\vertex_{15}$};
		\node (v16) at (2.0000000000000000,1.3333333333333333) {$\vertex_{16}$};
		\node (v17) at (2.0000000000000000,2.6666666666666665) {$\vertex_{17}$};
		
		\node at ($(v9.center)!0.5!(v17.center)$) {$\face$};
		\node at ($(v16.center)!0.5!(v14.center)$) {$\face'$};

		 \draw[step=1,eThickness] (v4.center) -- (v9.center) -- (v8.center) -- (v0.center) -- cycle;
		 \draw[step=1,eThickness] (v6.center) -- (v10.center) -- (v16.center) -- (v9.center) -- (v4.center) -- cycle;
		 \draw[step=1,eThickness] (v1.center) -- (v11.center) -- (v10.center) -- (v6.center) -- cycle;
		 \draw[step=1,eThickness] (v9.center) -- (v13.center) -- (v12.center) -- (v8.center) -- cycle;
		 \draw[step=1,eThickness] (v16.center) -- (v17.center) -- (v13.center) -- (v9.center) -- cycle;
		 \draw[step=1,eThickness] (v11.center) -- (v15.center) -- (v14.center) -- (v10.center) -- cycle;
		 \draw[step=1,eThickness] (v13.center) -- (v5.center) -- (v3.center) -- (v12.center) -- cycle;
		 \draw[step=1,eThickness] (v17.center) -- (v14.center) -- (v7.center) -- (v5.center) -- (v13.center) -- cycle;
		 \draw[step=1,eThickness] (v15.center) -- (v2.center) -- (v7.center) -- (v14.center) -- cycle;
		 \draw[step=1,eThickness] (v10.center) -- (v14.center) -- (v17.center) -- (v16.center) -- cycle;
	 \end{scope}
\end{tikzpicture}

%% file: ch_topological_complexes_constant.tex
\section[]{Homology of $\degBa{\constantComplex}{i}$}\label{ss:constant_homology}
In this section we collect main results characterizing the homology of the chain complex $\degBa{\constantComplex}{i}$. We will use its following equivalent form that follows from the proof of Proposition \ref{prop:exact_stair},
\begin{equation*}
	\begin{tikzcd}[column sep=small]
		\degBa{\constantComplex}{i}~:~ \moplus\limits_{\face \in \tmeshComponentF{,\pB{i}}{}} \faceE{} \ideal[M]{}_{\pB{i}} \arrow[r]& \moplus\limits_{\edge \in \tmeshInteriorComponentE{, \pB{i}}{}} \edgeE{} \ideal[M]{}_{\pB{i}} \arrow[r] & \moplus\limits_{\vertex \in \tmeshInteriorComponentV{, \pB{i}}{}} \vertexE{} \ideal[M]{}_{\pB{i}} \arrow[r] & 0\;.
	\end{tikzcd}
\end{equation*}
\begin{proposition}\label{prop:H2_constantComplex}
	\begin{equation*}
		\dimwp{H_2\left(\degBa{\constantComplex}{i}\right)}_\abdegree = 
		\begin{dcases}
			\dimwp{\ideal[L]{}_{\pB{\nlevels}}}_\abdegree\;, & i = \nlevels+1\\
			0\;, & 1 \leq i \leq \nlevels
		\end{dcases}\;.
	\end{equation*}
\end{proposition}
\begin{proof}
	Let $p = \sum_{\face} \faceE{}p_\face$, $p_\face \in \ideal[L]{}_{\pB{i-1}}$, be in the kernel of $\boundary$, i.e.,
	\begin{equation*}
		0 = \boundary(p) = \sum_{\edge\in \tmeshInteriorComponentE{,\pB{i}}{}} \edgeE{} \sum_{\face\in\tmeshComponentF{,\pB{i}}{}} \kronFE{}{} p_\face \Leftrightarrow \forall \edge\in \tmeshInteriorComponentE{,\pB{i}}{}\;,\; \sum_{\face\in\tmeshComponentF{,\pB{i}}{}} \kronFE{}{} p_\face \in \ideal[L]{}_{\pB{i}}\;.
	\end{equation*}
	If any $\face$ and $\face'$ share an edge $\edge$, $\kronFE{}{} = -\krond{\face'}{\edge}$. Therefore, if both $\face$ and $\face'$ also belong to $\tmeshComponentF{,\pB{i}}{}$, we must have $p_\face - p_{\face'} \in \ideal[L]{}_{\pB{i}}$. Then,
	\begin{itemize}
		\item $i = \nlevels+1$: Following Assumption \ref{ass:simplyConnectedDomain}, all edges in $\tmeshInteriorComponentE{,\pB{\nlevels+1}}{}$ are shared by exactly two faces in $\tmeshComponentF{,\pB{\nlevels+1}}{} \equiv \tmeshF$. 
		Therefore, all $p_\face$ must correspond to the same global polynomial in $\ideal[L]{}_{\pB{\nlevels}}$ for all faces in $\tmeshF$.
		
		\item $i < \nlevels+1$: There exists at least one edge in $\tmeshInteriorComponentE{,\pB{i}}{}$ that is contained in only one face in $\tmeshComponentF{,\pB{i}}{}$. Therefore, all $p_\face$ must be polynomials in $\ideal[L]{}_{\pB{i}}$, and therefore must be zero in $\ideal[L]{}_{\pB{i-1}}/\ideal[L]{}_{\pB{i}}$.
	\end{itemize}
\end{proof}

\begin{definition}[Number of relative holes in $\domainComponent{\pB{i}}{}$]
	We define $\holes{\pB{i}}{}$ to be the number of linearly independent, non-trivial cycles in $\domainComponent{\pB{i}}{}$ relative to $\boundary\domainComponent{\pB{i}}{}\cap\domainBnd$,
	\begin{equation*}
	\holes{\pB{i}}{} := \rankwp{H_1\left(\domainComponent{\pB{i}}{},\boundary\domainComponent{\pB{i}}{}\cap\domainBnd\right)}\;.
	\end{equation*}
\end{definition}

\begin{proposition}\label{prop:H1_constantComplex}
	\begin{equation*}
	\dimwp{H_1\left( \degBa{\constantComplex}{i} \right)}_\abdegree = \holes{\pB{i}}{}\dimwp{\ideal[M]{}_{\pB{i}}}_\abdegree\;.
	\end{equation*}
\end{proposition}
\begin{proof}
	The entire kernel of
	\begin{equation*}
		\boundary : \moplus\limits_{\edge \in \tmeshInteriorComponentE{,\pB{i}}{}} \edgeE{} \ideal[M]{}_{\pB{i}} \rightarrow \moplus\limits_{\vertex \in \tmeshInteriorComponentV{,\pB{i}}{}} \vertexE{} \ideal[M]{}_{\pB{i}}
	\end{equation*}
	can be generated by ($\RR$-linear combinations of) $c_{k, i}$ of the form
	\begin{equation*}
		c_{k, i} = p_{i-1}\alpha_{k} + \sum_{\edge \in \tmeshInteriorComponentE{,\pB{i}}{}} \edgeE{} p_{\edge, i}\;,
	\end{equation*}
	where
	\begin{itemize}
		\item $\alpha_{k} = \sum_{\edge\in\tmeshInteriorComponentE{,\pB{i}}{}}\edgeE{}o_{\edge}$, $o_{\edge} \in \ZZ$, is a relative cycle, i.e., $\boundary \alpha_{k} = 0$; and,
		\item $p_{i-1} \in \ideal[L]{}_{\pB{i-1}}$ and $p_{\edge, i} \in \ideal[L]{}_{\pB{i}}$.
	\end{itemize}
	Then, we only need to see how many such $c_{k, i}$ are linearly independent and not nullhomologous. In particular, if $p_{i-1} \notin \ideal[L]{}_{\pB{i}}$, $c_{k, i}$ is nullhomologous iff there exists some $d_{k, i}$ of the form
	\begin{equation*}
	d_{k, i} = p_{i-1}\beta_{k} + \sum_{\face \in \tmeshComponentF{,\pB{i}}{}} \faceE{} p_{\face, i}
	\end{equation*}
	such that $\boundary\beta_{k} = \alpha_k$, where
	\begin{itemize}
		\item $\beta_{k} = \sum_{\face\in\tmeshComponentF{,\pB{i}}{}}\faceE{}o_{\face}$, $o_{\face} \in \ZZ$; and,
		\item $p_{\face, i} \in \ideal[L]{}_{\pB{i}}$.
	\end{itemize}
	Then, $c_{k, i}$ is not nullhomologous in $H_1\left( \degBa{\constantComplex}{i} \right)$ iff $\alpha_k$ is not nullhomologous in $H_1\left(\domainComponent{\pB{i}}{},\boundary\domainComponent{\pB{i}}{}\cap\domainBnd\right)$.
\end{proof}

\begin{definition}[Number of relative connected components in $\domainComponent{\pB{i}}{}$]
	We define $\nComponent{\pB{i}}$ to be the number of connected components in $\domainComponent{\pB{i}}{}$ relative to $\boundary\domainComponent{\pB{i}}{}\cap\domainBnd$,
	\begin{equation*}
		\nComponent{\pB{i}} := \rankwp{H_0\left(\domainComponent{\pB{i}}{},\boundary\domainComponent{\pB{i}}{}\cap\domainBnd\right)}\;.
	\end{equation*}
\end{definition}
\begin{proposition}\label{prop:H0_constantComplex}
	\begin{equation*}
		\dimwp{H_0\left( \degBa{\constantComplex}{i}\right)}_\abdegree = \nComponent{\pB{i}}\dimwp{\ideal[M]{}_{\pB{i}}}_\abdegree\;.
	\end{equation*}
\end{proposition}
\begin{proof}
	All $\vertexE{} p_\vertex$, $\vertex \in \tmeshInteriorComponentV{,\pB{i}}{}$, $p_\vertex \in \ideal[L]{}_{\pB{i-1}}$, are in the kernel of $\boundary$. Let vertex $\vertex_0$, edges $\edge_{1},\dots,\edge_{k} \in \tmeshInteriorComponentE{,\pB{i}}{}$, $p_{\edge_l} \in \ideal[L]{}_{\pB{i-1}}$, and $o_1,\dots,o_k \in \ZZ$ be such that
	\begin{itemize}
		\item $\forall l \in \{1,\dots,k\}$, $p_{\edge_l} - p_\vertex \in \ideal[L]{}_{\pB{i}}$; and,
		\item $\vertexE{} = [\vertex_0] + \boundary\left(\sum_{l=1}^k \edgeE{l}o_l\right)\;.$
	\end{itemize}
	Assuming $p_\vertex \notin \ideal[L]{}_{\pB{i}}$, $\vertexE{}p_\vertex$ is nullhomologous only if $\vertex$ and $\vertex_0$ belong to the same connected component of $\domainComponent{\pB{i}}{}$ and $\vertex_0 \in \domainBnd$. Else, $\vertexE{0} p_\vertex$ would be a generator of $H_0\left(\degBa{\constantComplex}{i} \right)$ for the particular connected component of $\domainComponent{\pB{i}}{}$ that it belongs to. Then, for each $p_\vertex$, the number of such generators is exactly equal to $\nComponent{\pB{i}}$, and the claim follows.
\end{proof}

\begin{example}\label{ex:tmesh_homology}
	Consider the setup in Example \ref{ex:tmesh_activeMesh} and Figure \ref{fig:tmesh_levels}. Then, for $\abdegree = (5, 5)$, it can be verified that:
	\begin{equation*}
		\dimwp{H_2(\degBa{\constantComplex}{i})}_\abdegree = \begin{dcases}
			16\;, & i = 3\\
			0\;, & i = 1, 2
		\end{dcases}\;,
	\end{equation*}
	\begin{equation*}
		\dimwp{H_1(\degBa{\constantComplex}{i})}_\abdegree = 0\;,\qquad i = 1, 2, 3\;,
	\end{equation*}
	\begin{equation*}
		\dimwp{H_0(\degBa{\constantComplex}{i})}_\abdegree = 
		\begin{dcases}
			0\;, & i = 3\\
			9\;, & i = 2\\
			11\;, & i = 1
		\end{dcases}\;.
	\end{equation*}
\end{example}

%% file: ch_topological_complexes_ideal.tex
\section[]{The $0$-homology of $\degBa{\idealComplex}{i}$}\label{sec:ideals}
As per our objectives stated at the end of Section \ref{ss:approach_summary}, only the characterization of the $0$-homology of $\degBa{\idealComplex}{i}$
remains, and we collect the associated results in this section.
Similarly to the previous section, we will do so keeping in mind the simplified form of $\degBa{\idealComplex}{i}$ that follows from Proposition \ref{prop:exact_stair},
\begin{equation*}
	\begin{tikzcd}[column sep=small]
		\degBa{\idealComplex}{i}~:~ 0\arrow[r] & \moplus\limits_{\edge \in \tmeshInteriorComponentE{,\pB{i}}{}} \edgeE{} \ideal_\edge \cap \ideal[L]{}_{\pB{i-1}} + \ideal[L]{}_{\pB{i}} / \ideal[L]{}_{\pB{i}} \arrow[r] & \moplus\limits_{\vertex \in \tmeshInteriorComponentV{,\pB{i}}{}} \vertexE{} \ideal_\vertex \cap \ideal[L]{}_{\pB{i-1}} + \ideal[L]{}_{\pB{i}} / \ideal[L]{}_{\pB{i}} \arrow[r] & 0\;.
	\end{tikzcd}
\end{equation*}
We first provide a lower bound on the dimension of $H_0\left(\degBa{\idealComplex}{i}\right)$ that holds for special choices of $\bdegree, \bsmooth$ and $\abdegree$.
\begin{proposition}\label{prop:H0_idealComplex_lower_bound}
	Let $\abdegree \in \ZZP^2$ and $1 \leq i \leq \nlevels$.
	\begin{enumerate}[label=(\alph*)]
		\item If $\abdegree - \bdegree_{i}$ is entry-wise greater than or equal to $(\bsmooth_{\vertex, h}, \bsmooth_{\vertex, v})$ for each $\vertex \in \tmeshInteriorComponentV{,\pB{i}}{}$, then
		\begin{equation*}
			\dimwp{H_0\left(\degBa{\idealComplex}{i}\right)}_\abdegree \geq \dimwp{H_0\left(\degBa{\constantComplex}{i}\right)}_\abdegree\;.
		\end{equation*}
		In particular, the map from $H_0\left(\degBa{\idealComplex}{i}\right)_\abdegree$ to $H_0\left(\degBa{\constantComplex}{i}\right)_\abdegree$ in Equation \eqref{eq:long_ex_seq_homology} is a surjection and $H_0\left(\degBa{\quotientComplex}{i}\right)_\abdegree$ vanishes.
		\item If $\abdegree - \bdegree_{i-1}$ is entry-wise smaller than or equal to $(\bsmooth_{\vertex, h}, \bsmooth_{\vertex, v})$ for each $\vertex \in \tmeshInteriorComponentV{,\pB{i}}{}$, then
		\begin{equation*}
			\dimwp{H_0\left(\degBa{\idealComplex}{i}\right)}_\abdegree = 0\;.
		\end{equation*}
		In particular, the map from $H_0\left(\degBa{\constantComplex}{i}\right)_\abdegree$ to $H_0\left(\degBa{\quotientComplex}{i}\right)_\abdegree$ in Equation \eqref{eq:long_ex_seq_homology} is an isomorphism and $H_1\left(\degBa{\quotientComplex}{i}\right)_\abdegree$ vanishes.
	\end{enumerate}
\end{proposition}
\begin{proof}
	~\par
	\begin{enumerate}[label=(\alph*)]
		\item If $\abdegree - \bdegree_{i}$ is entry-wise greater than or equal to $(\bsmooth_{\vertex, h}, \bsmooth_{\vertex, v})$, then Corollary \ref{cor:vertex_ideal_dim} implies that
		the $\abdegree^{th}$ graded piece of $\moplus_{\vertex \in \tmeshInteriorComponentV{,\pB{i}}{}}\vertexE{} \degB{\pRingH}{\vertex}{i}/\degB{\ideal}{\vertex}{i}$ vanishes and so does $H_0\left(\degBa{\quotientComplex}{i}\right)_\abdegree$.
		
		\item If $\abdegree - \bdegree_{i-1}$ is entry-wise smaller than or equal to $(\bsmooth_{\vertex, h}, \bsmooth_{\vertex, v})$, then Corollary \ref{cor:vertex_ideal_dim} implies that
		the $\abdegree^{th}$ graded piece of $\moplus_{\vertex \in \tmeshInteriorComponentV{,\pB{i}}{}}\vertexE{} \degB{\ideal}{\vertex}{i}$ vanishes and so does $H_0\left(\degBa{\idealComplex}{i}\right)_\abdegree$.
	\end{enumerate}
	The claims then follow from Equation \eqref{eq:long_ex_seq_homology}.
\end{proof}

Let us define the graded multiplication map $\phi_{\pB{i}, \vertex}$ for $\vertex \in \tmeshInteriorComponentV{,\pB{i}}{}$,
\begin{equation}
\begin{split}
\phi_{\pB{i}, \vertex}~:~\moplus_{\edge \in \tmeshInteriorComponentE{,\pB{i}}{}} &\halfEdge{\vertex}{\edge} \shiftideal{M}{\pB{i}}{ -\bideg{\edge}} \rightarrow \vertexE{} \degB{\ideal}{\vertex}{i}\\
&\halfEdge{\vertex}{\edge}p \mapsto \vertexE{} p \Delta_\edge\;,
\end{split}
\end{equation}
where $\halfEdge{\vertex}{\edge}$ is a half-edge element, with $\halfEdge{\vertex}{\edge} := 0$ when $\kronEV{}{} = 0$ or when $\vertex \in \domainBnd$.

\begin{lemma}\label{lem:redIdealComplex_surjection0}
	The map $\phi_{\pB{i}, \vertex}$ is surjective.
\end{lemma}
\begin{proof}
	The claim follows from the isomorphism $\ideal_\edge \isomorphic \pRingH(-\bdegree(\edge) -\bideg{\edge})$ and the surjective map
	\begin{equation*}
	\begin{split}
		\moplus_{\edge \in \tmeshInteriorComponentE{,\pB{i}}{}} &\halfEdge{\vertex}{\edge} \ideal_\edge \cap \ideal[L]{}_{\pB{i-1}} \rightarrow \vertexE{} \ideal_\vertex \cap \ideal[L]{}_{\pB{i-1}}\\
		&\halfEdge{\vertex}{\edge}p \mapsto \vertexE{} p\;.
	\end{split}
	\end{equation*}
\end{proof}

Define $E_{h, \pB{i}}(\vertex)$, $E_{v, \pB{i}}(\vertex)$ as the sets of horizontal and vertical edges in $\tmeshInteriorComponentE{,\pB{i}}{}$ that contain $\vertex \in \tmeshInteriorComponentV{,\pB{i}}{}$, respectively, and let $E_{\pB{i}}(\vertex) = E_{h, \pB{i}}(\vertex) \cup E_{v, \pB{i}}(\vertex)$.
Let $P_{\pB{i}}(\vertex)$ be the set that contains edge-pairs $(\edge, \edge')$ containing $\vertex$, both either in $E_{h, \pB{i}}(\vertex)$ or $E_{v, \pB{i}}(\vertex)$; we identify $(\edge,\edge')$ with $(\edge',\edge)$.
Note that, depending on the index $i$, $P_{\pB{i}}(\vertex)$ may be empty, or may contain either one or two elements.
When the vertex $\vertex$ is obvious from the context, we will exclude it from the above notation to simplify the reading (and writing!) of the text.

\begin{proposition}\label{prop:halfedge_map_kernel}
	\begin{equation*}
	\begin{split}
	\kerwp{\phi_{\pB{i}, \vertex}} = 
		&\sum_{(\edge,\edge')\in P_{\pB{i}}} \left(\halfEdge{\vertex}{\edge} - \halfEdge{\vertex}{\edge'} \right) \shiftideal{M}{\pB{i}}{ -\bideg{\edge}} 
		+ \sum_{\substack{\edge\in E_{v, \pB{i}}\\\edge'\in E_{h, \pB{i}}}} \left(\halfEdge{\vertex}{\edge}\Delta_{\edge'}  - \halfEdge{\vertex}{\edge'}\Delta_{\edge} \right) \shiftideal{M}{\pB{i}}{ -\bideg{\vertex}}\;.
	\end{split}
	\end{equation*}
\end{proposition}
\begin{proof}
	With respect to the T-mesh $\tmeshComponent{\pB{i}}{}$, let $\vertex$ be a crossing vertex. (The proof for when $\vertex$ is a T-junction will proceed analogously.) Let $\edge_1,\edge_2 \in E_{h, \pB{i}}$, $\edge_3,\edge_4 \in E_{v, \pB{i}}$, and $P_{\pB{i}} = \big\{(\edge_1,\edge_2), (\edge_3,\edge_4)\big\}$, and let $p_j \in \shiftideal{M}{\pB{i}}{ -\bideg{\edge_j}}$. Then,
	\begin{equation*}
	\begin{split}
		\phi_{\pB{i}, \vertex}\left( \sum_{j=1}^4 \halfEdge{\vertex}{\edge_j}p_j \right) &= \vertexE{}\sum_{j=1}^4 p_j \Delta_{\edge_j}\;,\\
		&= \vertexE{}\left( p_1 + p_2 \right)\Delta_{\edge_1} + \vertexE{}\left( p_3 + p_4\right)\Delta_{\edge_3}\;,
	\end{split}
	\end{equation*}
	where $\Delta_{\edge_1} = \Delta_{\edge_2}$, and $\Delta_{\edge_3} = \Delta_{\edge_4}$, and $\Delta_{\edge_1}$ and $\Delta_{\edge_3}$ are relatively prime. Therefore, the kernel of the map is generated by,
	\begin{equation*}
	\begin{split}
		\left(\halfEdge{\vertex}{\edge_1} - \halfEdge{\vertex}{\edge_2}\right)p_{12}\;,~
		\left(\halfEdge{\vertex}{\edge_3} - \halfEdge{\vertex}{\edge_4}\right)p_{34}\;,~
		\left(\halfEdge{\vertex}{\edge_h}\Delta_{\edge_v} - \halfEdge{\vertex}{\edge_v}\Delta_{\edge_h}\right)p_{hv}\;,
	\end{split}
	\end{equation*}
	where,
	\begin{equation*}
	\begin{split}
		&p_{12} \in \shiftideal{M}{\pB{i}}{ -\bideg{\edge_1}}\;,\;
		p_{34} \in \shiftideal{M}{\pB{i}}{ -\bideg{\edge_3}}\;,\;
		p_{hv} \in \shiftideal{M}{\pB{i}}{ -\bideg{\vertex}}\;.
	\end{split}
	\end{equation*}
\end{proof}

Using $\phi_{\pB{i}, \vertex}$, we can define a map $\phi_{\pB{i}}$ as
\begin{equation}
	\phi_{\pB{i}}~:~\moplus_{\vertex \in \tmeshInteriorComponentV{,\pB{i}}{}}\moplus_{\edge \in \tmeshInteriorComponentE{,\pB{i}}{}}\halfEdge{\vertex}{\edge}  \shiftideal{M}{\pB{i}}{ -\bideg{\edge}} \rightarrow \moplus_{\vertex \in \tmeshInteriorComponentV{,\pB{i}}{}} \vertexE{} \degB{\ideal}{\vertex}{i}\;,
\end{equation}
with kernel,
\begin{equation}
	\kerwp{\phi_{\pB{i}}} = \sum_{\vertex \in \tmeshInteriorComponentV{,\pB{i}}{}} \kerwp{\phi_{\pB{i}, \vertex}}\;.
	\label{eq:kernelOfPhi}
\end{equation}

Next, let us consider the diagram
\begin{equation}
	\begin{tikzcd}[row sep=large]
		\moplus\limits_{\edge \in \tmeshInteriorComponentE{,\pB{i}}{}} \edgeE{} \shiftideal{M}{\pB{i}}{ -\bideg{\edge}} \arrow[r,"\tilde{\boundary}"] \arrow[d,"\psi_{\pB{i}}"] & \moplus\limits_{\vertex \in \tmeshInteriorComponentV{,\pB{i}}{}} \moplus\limits_{\edge \in \tmeshInteriorComponentE{,\pB{i}}{}} \halfEdge{\vertex}{\edge} \shiftideal{M}{\pB{i}}{ -\bideg{\edge}}  \arrow[d,"\phi_{\pB{i}}"] \\
		\moplus\limits_{\edge \in \tmeshInteriorComponentE{,\pB{i}}{}} \edgeE{} \ideal_{\edge, \pB{i}} \arrow[r,"\boundary"] & \moplus\limits_{\vertex\in \tmeshInteriorComponentV{,\pB{i}}{}} \vertexE{} \ideal_{\vertex, \pB{i}}
	\end{tikzcd}\;,
\label{eq:H0_idealComplex_commDiag}
\end{equation}
where the maps $\tilde{\boundary}$ and ${\boundary}$ are the restrictions of the following maps to the active T-mesh,
\begin{align}
	&\tilde{\boundary}~:~\edgeE{} \mapsto \sum_{\vertex}\kronEV{}{}\halfEdge{\vertex}{\edge}\;,\qquad
	&{\boundary}~:~\edgeE{} \mapsto \sum_{\vertex}\kronEV{}{}\vertexE{}\;,
\end{align}
and the graded map $\psi_{\pB{i}}$ is defined as
\begin{equation*}
 \begin{split}
&\psi_{\pB{i}}~:~\edgeE{} \mapsto \edgeE{}\Delta_\edge\;.
 \end{split}
\end{equation*}
\begin{lemma}\label{lem:redIdealComplex_surjection1}
	The map $\psi_{\pB{i}}$ is surjective.
\end{lemma}
\begin{proof}
	The claim follows from the isomorphism $\ideal_\edge \isomorphic \pRingH(-\bdegree(\edge) -\bideg{\edge})$.
\end{proof}

\begin{lemma}\label{lem:H0_idealComplex_halfedges}
	\begin{equation*}
	\begin{split}
		H_0\left(\degBa{\idealComplex}{i} \right)
			&\isomorphic \moplus_{\vertex \in \tmeshInteriorComponentV{,\pB{i}}{}} \moplus_{\edge \in \tmeshInteriorComponentE{,\pB{i}}{}}\halfEdge{\vertex}{\edge} \shiftideal{M}{\pB{i}}{ -\bideg{\edge}} 
			\bigg/ \left(\kerwp{\phi_{\pB{i}}} + \tilde{\boundary}\left( \moplus_{\edge \in \tmeshInteriorComponentE{,\pB{i}}{}}\edgeE{} \shiftideal{M}{\pB{i}}{ -\bideg{\edge}} \right)\right)\;.
	\end{split}
	\end{equation*}
\end{lemma}
\begin{proof}
	The diagram in Equation \eqref{eq:H0_idealComplex_commDiag} commutes. Indeed,
	\begin{equation*}
	\begin{tikzcd}
		\edgeE{}p \rar[maps to] \dar[maps to] &  \sum_{\vertex}\halfEdge{\vertex}{\edge}\kronEV{}{}p \dar[maps to]\\
		\edgeE{}p\Delta_\edge \rar[maps to] & \sum_{\vertex}\vertexE{}\kronEV{}{}p\Delta_\edge
	\end{tikzcd}\;.
	\end{equation*}
	Then, the result follows from surjectivities of $\phi_{\pB{i}}$ and $\psi_{\pB{i}}$ (Lemmas \ref{lem:redIdealComplex_surjection0} and \ref{lem:redIdealComplex_surjection1}, respectively), and surjectivity of the induced morphism,
	\begin{equation*}
		\begin{split}
			\phi_\star~:~&\moplus_{\vertex \in \tmeshInteriorComponentV{,\pB{i}}{}}\moplus_{\edge \in \tmeshInteriorComponentE{,\pB{i}}{}}\halfEdge{\vertex}{\edge}\shiftideal{M}{\pB{i}}{ -\bideg{\edge}} ~\big/~ \imwp{\tilde{\boundary}} \rightarrow \moplus\limits_{\vertex \in \tmeshInteriorComponentV{,\pB{i}}{}} \vertexE{} \degB{\ideal}{\vertex}{i} ~\big/~ \imwp{\boundary}\;,\\
			& \sum_{\vertex}\sum_{\edge}\halfEdge{\vertex}{\edge}p_{\vertex\edge} + \imwp{\tilde{\boundary}} \mapsto \phi_{\pB{i}}\left(\sum_{\vertex}\sum_{\edge}\halfEdge{\vertex}{\edge}p_{\vertex\edge}\right) + \imwp{\boundary}\;.
		\end{split}
	\end{equation*}
	Indeed, the kernel of $\phi_\star$ is exactly $\kerwp{\phi_{\pB{i}}} + \imwp{\tilde{\boundary}} / \imwp{\tilde{\boundary}}$ and we have the isomorphism, 
	\begin{equation*}
		\left(\moplus_{\vertex \in \tmeshInteriorComponentV{,\pB{i}}{}}\moplus_{\edge \in \tmeshInteriorComponentE{,\pB{i}}{}}\halfEdge{\vertex}{\edge}\shiftideal{M}{\pB{i}}{ -\bideg{\edge}} ~\big/~ \imwp{\tilde{\boundary}}\right) ~\big/~ \kerwp{\phi_\star} \isomorphic \moplus\limits_{\vertex \in \tmeshInteriorComponentV{,\pB{i}}{}} \vertexE{} \degB{\ideal}{\vertex}{i} ~\big/~ \imwp{\boundary} = H_0 \left( \degBa{\idealComplex}{i} \right)\;.
	\end{equation*}
\end{proof}

Before proceeding, we first introduce the concept of maximal segments for the T-mesh $\tmeshComponent{\pB{i}}{}$. This will help us further simplify the half-edge based form of $H_0 \left( \degBa{\idealComplex}{i} \right)$ from Proposition \ref{lem:H0_idealComplex_halfedges}.

\begin{definition}[Active maximal segments]\label{def:max_segs}
	Given index $i \in \{ 1, \dots, \nlevels+1 \}$, the set of active horizontal (resp. vertical) maximal segments $\MSH{\pB{i}}{}$ (resp. $\MSV{\pB{i}}{}$) is the set containing maximal connected unions of edges in $\tmeshComponentEH{,\pB{i}}{}$ (resp. $\tmeshComponentEV{,\pB{i}}{}$).
\end{definition}
The set of all active maximal segments will be denoted by $\MSA{\pB{i}}{} = \MSH{\pB{i}}{} \cup \MSV{\pB{i}}{}$, while the set of active maximal segments that do not intersect the boundary will be denoted by $\IMSA{\pB{i}}{}$; with some abuse of notation, we will refer to these maximal segments as ``interior maximal segments.''
By definition of the smoothness distribution, we can unambiguously define $\bideg{\ms} = \bideg{\edge}$ and $\Delta_{\ms} = \Delta_{\edge}$ for any edge $\edge \subseteq \ms \in \IMSA{\pB{i}}{}$.
\begin{proposition}\label{prop:H0_idealComplex_ms}
	\begin{equation*}
		\begin{split}
			H_0\left( \degBa{\idealComplex}{i} \right) \isomorphic 
			&\moplus\limits_{\ms \in \IMSA{\pB{i}}{}} \msE{} \shiftideal{M}{\pB{i}}{ -\bideg{\ms}}\;\bigg/\;
			\sum_{\ms_h \in \IMSH{\pB{i}}{}}\sum_{\ms_v \in \IMSV{\pB{i}}{}}\sum_{\substack{\vertex \in \tmeshInteriorComponentV{,\pB{i}}{}\\\ms_h \cap \ms_v = \{\vertex\}}} \left(\msE{h}\Delta_{\ms_v}  - \msE{v} \Delta_{\ms_h} \right) \shiftideal{M}{\pB{i}}{ -\bideg{\vertex}}\;.
		\end{split}
	\end{equation*}
\end{proposition}
\begin{proof}
	Using Equation \eqref{eq:kernelOfPhi} and Lemma \ref{lem:H0_idealComplex_halfedges}, we define
	\begin{equation*}
		\begin{split}
			&K' =  \sum_{\edge \in \tmeshInteriorComponentE{,\pB{i}}{}} \left(\sum_\vertex\kronEV{}{}\halfEdge{\vertex}{\edge}\right) \shiftideal{M}{\pB{i}}{ -\bideg{\edge}} + \sum_{\vertex \in \tmeshInteriorComponentV{,\pB{i}}{}}\sum_{(\edge, \edge') \in P_{\pB{i}}} \left(\halfEdge{\vertex}{\edge} - \halfEdge{\vertex}{\edge'} \right) \shiftideal{M}{\pB{i}}{ -\bideg{\edge}}\;,\\
			&K = K' +
			\sum_{\vertex \in \tmeshInteriorComponentV{,\pB{i}}{}} \sum_{\substack{\edge\in E_{v, \pB{i}}\\\edge'\in E_{h, \pB{i}}}} \left(\halfEdge{\vertex}{\edge}\Delta_{\edge'}  - \halfEdge{\vertex}{\edge'}\Delta_\edge \right) \shiftideal{M}{\pB{i}}{ -\bideg{\vertex}}\;,\\
			&B = \moplus_{\vertex \in \tmeshInteriorComponentV{,\pB{i}}{}} \moplus_{\edge \in \tmeshInteriorComponentE{,\pB{i}}{}} \halfEdge{\vertex}{\edge} \shiftideal{M}{\pB{i}}{ -\bideg{\edge}}\;.
		\end{split}
	\end{equation*}
	The first term of $K'$ corresponds to relations yielding the identification of $\halfEdge{\vertex}{\edge}$ with $\edgeE{}$. The second term of $K'$ corresponds to relations yielding the identification of $\halfEdge{\vertex}{\edge}$ with $\halfEdge{\vertex}{\edge'}$ whenever $\edge, \edge' \subset \ms$. Therefore, $K'$ leads to the identification of all edges that belong to the same maximal segment $\ms \in \MSA{\pB{i}}{}$.
	
	Keeping the above in mind, and since $B / K \isomorphic (B/K') / (K/K')$, we take the quotient with $K'$. The required description is obtained by noticing that, since $\halfEdge{\vertex}{\edge} = 0$ if $\vertex \in \boundary\domain$, terms corresponding to $\msE{}$ must be zero in the quotient for all active maximal segments that intersect $\boundary\domain$.
\end{proof}

%% file: ch_dimension_formula.tex
\section[]{Bounds on the dimension of $\splSpaceH^\bsmooth_\bdegree$}\label{sec:dim_formula}
We will use Proposition \ref{prop:homo_dim} in this section to provide upper and lower bounds on the dimensions of graded pieces of $\splSpaceH^\bsmooth_{\bdegree}$.
We start by providing the following lower and upper bounds on the spline space dimension.
Some of the results presented here will assume that the condition of sufficiency in Proposition \ref{prop:H0_idealComplex_lower_bound}(a) is satisfied. Therefore, for the sake of convenience, we define the following configuration so that we can refer to it later.
\begin{configuration}\label{config:practical_smoothness}
  The degree and smoothness distributions are such that, for all $i$,
  $\abdegree - \bdegree_{i}$ is entry-wise greater than or equal to
  $(\bsmooth_{\vertex, h}, \bsmooth_{\vertex, v})$ for each $\vertex
  \in \tmeshInteriorComponentV{,\pB{i}}{}$.
\end{configuration}

\begin{theorem}[Lower bound for general smoothness distributions]\label{thm:lower_bound_general}
	\begin{equation*}
		\begin{split}
			\dimwp{\splSpaceH^\bsmooth_\bdegree}_\abdegree &\geq \euler{\quotientComplex}_\abdegree - \sum_{i=1}^{\nlevels+1}\nComponent{\pB{i}}\dimwp{\ideal[M]{}_{\pB{i}}}_\abdegree\;.
		\end{split}
	\end{equation*}
\end{theorem}
\begin{proof}
	The lower bound can be arrived at in exactly the same way as Proposition \ref{prop:homo_dim} but with a slightly different point of departure.
	Instead of using the short exact sequence in Equation \eqref{eq:exact_stair}, we embed the complex  $\quotientComplex$ directly in the short exact sequence $0 \rightarrow \idealComplex \rightarrow \constantComplex \rightarrow \quotientComplex \rightarrow 0$,
	\begin{equation*}
		\begin{tikzcd}
			~ & ~ & 0 \arrow[d,""] & 0 \arrow[d,""] & ~ \\
			\idealComplex~:~ & 0\arrow[r] \arrow[d]& \moplus\limits_{\edge \in \tmeshInteriorE} \edgeE{} \ideal_\edge \arrow[r] \arrow[d] & \moplus\limits_{\vertex \in \tmeshInteriorV} \vertexE{} \ideal_\vertex \arrow[r] \arrow[d] & 0 \\
			\constantComplex~:~ & \moplus\limits_{\face \in \tmeshF} \faceE{} \pRingH_\face \arrow[r] \arrow[d] & \moplus\limits_{\edge \in \tmeshInteriorE} \edgeE{} \pRingH_\edge \arrow[r] \arrow[d] & \moplus\limits_{\vertex \in \tmeshInteriorV} \vertexE{} \pRingH_\vertex \arrow[r] \arrow[d] & 0 \\
			\quotientComplex~:~ & \moplus\limits_{\face \in \tmeshF} \faceE{} \pRingH_\face \arrow[r] & \moplus\limits_{\edge \in \tmeshInteriorE} \edgeE{} \pRingH_\edge/\ideal_\edge \arrow[r] \arrow[d] & \moplus\limits_{\vertex \in \tmeshInteriorV} \vertexE{} \pRingH_\vertex/ \ideal_\vertex \arrow[r] \arrow[d] & 0 \\
			~ & ~ & 0 & 0 & ~
		\end{tikzcd}
	\end{equation*}
	In a manner similar to the proof of Proposition \ref{prop:H0_constantComplex}, it is easy to establish that
	\begin{equation*}
		\dimwp{H_0(\constantComplex)}_\abdegree = \sum_{i=1}^{\nlevels+1}\dimwp{H_0(\degBa{\constantComplex}{i})}_\abdegree\;,\quad
		\dimwp{H_1(\constantComplex)}_\abdegree = \sum_{i=1}^{\nlevels+1}\dimwp{H_1(\degBa{\constantComplex}{i})}_\abdegree\;.
	\end{equation*}
	Then, following the same steps as in Section \ref{ss:approach_summary} but for the diagram above, one can derive the following equation,
	\begin{equation*}
		\dimwp{\splSpaceH^\bsmooth_\bdegree}_\abdegree = \euler{\quotientComplex}_\abdegree + \dimwp{H_0(\idealComplex)}_\abdegree - \dimwp{H_0(\constantComplex)}_\abdegree\;.
	\end{equation*}
	which yields the claimed lower bound.
\end{proof}

Before presenting a sharper lower bound on the spline space dimension, let us present simple upper bounds on $\dimwp{\im{\hat{\boundary}_i}}_\abdegree$, which are themselves often pessimistic in practical configurations. The following can be readily derived using Equation \eqref{eq:long_ex_seq_homology}.
\begin{corollary}\label{cor:dimension_boundary_hat}
	~\par
	\begin{enumerate}[label=(\alph*)]
		\item In general,
		\begin{equation*}
		\begin{split}
		\dimwp{\im{\hat{\boundary}_i}}_\abdegree \leq 
		\min\bigg\{
		&\dimwp{H_2(\degAa{\quotientComplex}{i-1})}_\abdegree,\\
		&\dimwp{H_0(\degBa{\idealComplex}{i})}_\abdegree - \nComponent{\pB{i}}\dimwp{\ideal[M]{}_{\pB{i}}}_\abdegree + \dimwp{H_0(\degBa{\quotientComplex}{i})}_\abdegree
		\bigg\}\;.
		\end{split}			
		\end{equation*}
		\item For Configuration \ref{config:practical_smoothness},
		\begin{equation*}
		\begin{split}
		\dimwp{\im{\hat{\boundary}_i}}_\abdegree \leq 
		\min\bigg\{
		&\dimwp{H_2(\degAa{\quotientComplex}{i-1})}_\abdegree,~\dimwp{H_0(\degBa{\idealComplex}{i})}_\abdegree - \nComponent{\pB{i}}\dimwp{\ideal[M]{}_{\pB{i}}}_\abdegree
		\bigg\}\;.
		\end{split}			
		\end{equation*}
	\end{enumerate}
\end{corollary}

\begin{theorem}[Lower bound for practical smoothness distributions]\label{thm:lower_bound_special}
	For Configuration \ref{config:practical_smoothness},
	\begin{equation*}
		\dimwp{\splSpaceH^\bsmooth_\bdegree}_\abdegree \geq \euler{\quotientComplex}_\abdegree 
		\;.
	\end{equation*}
\end{theorem}
\begin{proof}
	Since the conditions of sufficiency in Proposition \ref{prop:H0_idealComplex_lower_bound}(a) is assumed to be satisfied for all $1 \leq i \leq \nlevels$, $H_0(\degBa{\idealComplex}{i})_\abdegree$ surjects onto $H_0(\degBa{\constantComplex}{i})_\abdegree$.
	For $i = \nlevels+1$, the dimension of $H_0(\degBa{\idealComplex}{i})_\abdegree$ is trivially bounded from below by $0 = \dimwp{H_0(\degBa{\constantComplex}{i})}_\abdegree$.
	From Corollary \ref{cor:dimension_boundary_hat}, this implies that
	\begin{equation*}
		-\dimwp{\im{\hat{\boundary}_i}}_\abdegree \geq \nComponent{\pB{i}}\dimwp{\ideal[M]{}_{\pB{i}}}_\abdegree - \dimwp{H_0(\degBa{\idealComplex}{i})}_\abdegree\;.
	\end{equation*}
	The claim follows from Proposition \ref{prop:homo_dim} and  Proposition \ref{prop:H0_constantComplex}.
\end{proof}

\begin{theorem}[Upper bound for general smoothness distributions]\label{thm:upper_bound_general}
	\begin{equation*}
	\begin{split}
	\dimwp{\splSpaceH^\bsmooth_\bdegree}_\abdegree &\leq \euler{\quotientComplex}_\abdegree + 
	\sum_{i=1}^{\nlevels+1}\dimwp{H_0(\degBa{\idealComplex}{i})}_\abdegree  - \nComponent{\pB{i}}\dimwp{\ideal[M]{}_{\pB{i}}}_\abdegree \;.
	\end{split}
	\end{equation*}
\end{theorem}
\begin{proof}
	Since $\dimwp{\im{\hat{\boundary}_i}}_\abdegree \geq 0$, the claim follows from Proposition \ref{prop:homo_dim}.
\end{proof}

It only remains to derive upper bounds on $\dimwp{H_0(\degBa{\idealComplex}{i})}_\abdegree$ and we do so next.
Given a particular $i$, we bound the dimensions of graded pieces of $H_0\left( \degBa{\idealComplex}{i} \right)$ from above
by introducing an ordering on the active interior maximal segments, i.e., on the elements of $\IMSA{\pB{i}}{}$ and by utilizing the representation of $H_0\left( \degBa{\idealComplex}{i} \right)$ from Proposition \ref{prop:H0_idealComplex_ms}.

\begin{definition}[Ordering of ${\IMSA{\pB{i}}{}}$]
	Given $i$, let $\xi_{\pB{i}}$ be an ordering on $\IMSA{\pB{i}}{}$, i.e., an injective map from $\IMSA{\pB{i}}{}$ to $\NN$.
	Given $\xi_{\pB{i}}$ and $\ms \in \IMSA{\pB{i}}{}$, define $\Gamma_{[i]} (\ms) \subset \MSA{\pB{i}}{}$ as the set of maximal segments $\ms'$ that intersect $\ms$ non-trivially and such that either $\xi_{\pB{i}}(\ms)  > \xi_{\pB{i}}(\ms')$ or $\ms' \cap \boundary\domain \neq \emptyset$.
\end{definition}
Hereafter, we will assume that given $i$ the ordering $\xi_{\pB{i}}$ is fixed. We will abuse the notation by using $\ms > \ms'$ to mean the same thing as $\xi_{\pB{i}}(\ms)  > \xi_{\pB{i}}(\ms')$.
Let us define the modules
\begin{align}
	M_{\pB{i}} &:= \moplus\limits_{\ms \in \IMSA{\pB{i}}{}} \msE{} \shiftideal{M}{\pB{i}}{ -\bideg{\ms}}\;,\label{eq:H0_numerator}\\
	D_{\pB{i}} &:= \sum_{\ms_h \in \IMSH{\pB{i}}{}}\sum_{\ms_v \in \IMSV{\pB{i}}{}}\sum_{\substack{\vertex \in \tmeshInteriorComponentV{,\pB{i}}{}\\\ms_h \cap \ms_v = \{\vertex\}}} \left(\msE{h}\Delta_{\ms_v}  - \msE{v} \Delta_{\ms_h} \right) \shiftideal{M}{\pB{i}}{ -\bideg{\vertex}}\;.
\label{eq:H0_denominator}
\end{align}
For $p = \sum_{\ms}\msE{} p_{\ms} \in M_{\pB{i}}$, we define its initial, denoted $\initial{p}$, as $\elem{\ms'}p_{\ms'}$ if, out of all $\ms$ such that $p_\ms \neq 0$, $\ms'$ has the biggest index according to $\xi_{\pB{i}}$.
\begin{lemma}\label{lem:H0_idealComplex_initial}
	\begin{equation*}
		H_0\left(\degBa{\idealComplex}{i}\right) \isomorphic M_{\pB{i}} / \initial{D_{\pB{i}}}\;.
	\end{equation*}
\end{lemma}
\begin{proof}
	The claim is a standard result; see \cite{schenck2003computational}, for example.
\end{proof}

\begin{proposition}\label{prop:H0_idealComplex_upperBound}
	\begin{equation*}
	\begin{split}
		\dimwp{H_0(\degBa{\idealComplex}{i})}_\abdegree \leq \sum_{\ms \in \IMSA{\pB{i}}{}} 
		\left(\dimwp{\shiftideal{M}{\pB{i}}{ -\bideg{\ms}}}_\abdegree - \dimwp{\sum_{\ms' \in \Gamma_{\pB{i}}(\ms) } \Delta_{\ms'}\shiftideal{M}{\pB{i}}{ -\bideg{\ms}-\bideg{\ms'}}}_\abdegree\right)\;.
	\end{split}
	\end{equation*}
\end{proposition}
\begin{proof}
	From Lemma \ref{lem:H0_idealComplex_initial}, if we can provide a lower bound on the dimension of  $\initial{D_{\pB{i}}}$, then we can provide an upper bound on $\dimwp{H_0(\degBa{\idealComplex}{i})}_\abdegree$.
	
	Notice that $\initial{D_{\pB{i}}}$ is going to be at least partially generated by the initials of its generators. Looking at the generators, for each $\ms \in \IMSA{\pB{i}}{}$ the contributions only come from the $\ms' \in \Gamma_{\pB{i}}(\ms)$. The claim follows.
\end{proof}

In practice, we have observed that if each connected component of $\domainComponent{\pB{i}}{}$ intersects $\boundary\domain$, the upper bound in Proposition \ref{prop:H0_idealComplex_upperBound} is usually optimal.
Here we use ``optimal'' in the sense of the upper bound coinciding with the exact dimension of $H_0(\degBa{\idealComplex}{i})$ in bi-degree $\abdegree$.
However, if $\boundary\domainComponent{\pB{i}}{}$ does not intersect $\boundary\domain$, it turns out that the upper bound in Proposition \ref{prop:H0_idealComplex_upperBound} is a poor estimate.
In other words, the initials of the generators of $D_{\pB{i}}$ used in Proposition \ref{prop:H0_idealComplex_upperBound} do a bad job of approximating its dimension.
Nonetheless, we can significantly improve upon this estimate by \emph{systematically} adding some new generators to the ones used previously.
In particular, doing so will allow us to compute the dimension of  $H_0(\degBa{\idealComplex}{i})$ exactly even when there exist connected components of $\domainComponent{\pB{i}}{}$ that do not intersect $\boundary\domain$; see Sections \ref{sec:stable_dimension} and \ref{sec:examples} for examples.
In turn, this will enable the use of the sufficient conditions outlined later in Section \ref{sec:stable_dimension} for computing the exact spline space dimension in many cases.

Following the above comments, for a given maximal segment $\ms \in \MSA{\pB{i}}{}$, consider the particular connected component of $\domainComponent{\pB{i}}{}$ that $\ms$ belongs to.
Let us focus on enlarging the set of generators involving $\ms$ in Equation \eqref{eq:H0_denominator}.
Define $\Upsilon_{\pB{i}}(\ms)$ as the following set of maximal segment pairs,
\begin{equation}
\Upsilon_{\pB{i}}(\ms) := \left\{ (\ms_1, \ms_2) \in \IMSA{\pB{i}}{}\times\IMSA{\pB{i}}{}~:~ \ms_2 \cap \ms \neq \emptyset \neq \ms_2 \cap \ms_1,~\bsmooth(\ms) \geq \bsmooth(\ms_1),~\ms > \ms_1\right\}\;.
\end{equation}
Note that in the above definition $\ms$ and $\ms_1$ must be parallel and $\ms_2$ must be perpendicular to both.
In Equation \eqref{eq:H0_denominator}, the generators of $D_{\pB{i}}$ already contain explicit relations between $\msE{}$ and $\msE{2}$, and between $\msE{1}$ and $\msE{2}$ for $(\ms_1, \ms_2) \in \Upsilon_{\pB{i}}(\ms)$.
Then, these generators can be manipulated to give a generator involving only $\msE{}$ and $\msE{1}$.

\vspace{-2\baselineskip}
\leavevmode{
	\begin{wrapfigure}[4]{r}{0.20\textwidth}
		\begin{tikzpicture}
		\coordinate (c0) at (0,0);
		\coordinate (c1) at (2,0);
		\coordinate (c2) at (0,2);
		\coordinate (c3) at (2,2);
		\draw[eThickness] (c0) -- (c1);
		\draw[eThickness] (c2) -- (c3);
		\draw[eThickness] (c0) -- (c2);
		\node[circle,fill=black,black,inner sep=0pt,minimum size=2pt] at (c0) {};
		\node[circle,fill=black,black,inner sep=0pt,minimum size=2pt] at (c2) {};
		\node[label={west:$\ms_2$}] at ($(c0)!0.5!(c2)$) {};
		\node[label={south:$\ms_1$}] at ($(c0)!0.5!(c1)$) {};
		\node[label={north:$\ms$}] at ($(c2)!0.5!(c3)$) {};
		\draw[<->] ($(c0)!0.74!(c2)$) arc (-90:0:14pt);
		\draw[<->] ($(c0)!0.26!(c1)$) arc (0:90:14pt);
		\draw[<->, dashed] ($(c0)!0.5!(c1)$) -- ($(c2)!0.5!(c3)$);
		\end{tikzpicture}
	\end{wrapfigure}
	\begin{lemma}\label{lem:new_relations}
		Let $1 \leq i \leq \nlevels$.
		Given maximal segment $\ms \in \IMSA{\pB{i}}{}$, let $(\ms_1, \ms_2) \in \Upsilon_{\pB{i}}(\ms)$. Recalling Equation \eqref{eq:deg_def_diff}, define $\mbf{\Delta p}_{\ms}$ as
		\begin{equation*}
		\mbf{\Delta p}_{\ms} := 
			\begin{dcases}
			(\Delta n_{i1}, 0)\;,& \ms \in \MSH{\pB{i}}{j}\;,\\
			(0, \Delta n_{i2})\;,& \ms \in \MSV{\pB{i}}{j}\;.
		\end{dcases}
		\end{equation*}
		Then, there is a polynomial $\hat{\Delta}$ such that
		\begin{equation*}
		\mbf{u}^{\mbf{\Delta p}_\ms}\Delta_{\ms_2}\left(\msE{} - \elem{\ms_1}\hat{\Delta}\right) \shiftideal{M}{\pB{i}}{ -\bideg{\ms}-\bideg{\ms_2}-\mbf{\Delta p}_{\ms}} \subset {D}_{\pB{i}}\;.
		\end{equation*}
\end{lemma}}
\begin{proof}
	Without loss of generality, let $\ms, \ms_1 \in \MSV{\pB{i}}{}$ and define $\smooth_{\ms\ms_1} := \bsmooth(\ms) - \bsmooth(\ms_1) \geq 0$. By definition of $D_{\pB{i}}$, the following are two of its generators,
	\begin{align*}
	&\left(\elem{\ms_2} \Delta_{\ms_1} - \elem{\ms_1} \Delta_{\ms_2} \right)\shiftideal{M}{\pB{i}}{ -\bideg{\ms_1}-\bideg{\ms_2}}\;,\tag*{($\star_1$)}\\
	&\left(\elem{\ms_2} \Delta_{\ms} - \msE{} \Delta_{\ms_2}\right) \shiftideal{M}{\pB{i}}{ -\bideg{\ms}-\bideg{\ms_2}}\;. \tag*{($\star_2$)}
	\end{align*}
	
	Let $\Delta_{\ms_1} = s^{\bsmooth(\ms_1)+1}$ and $\Delta_{\ms} = (s + au)^{(\bsmooth(\ms)+ 1)}$, $a \in \RR$. We can write $\Delta_{\ms} = \Delta_{\ms_1}\hat{\Delta} + \Delta'$, where the term with the highest power of $s$ in $\Delta'$ is a multiple of $s^{\bsmooth(\ms_1)}u^{\smooth_{\ms\ms_1}+1}$.
	Then, we can combine the two generators in Equations ($\star_1$) and ($\star_2$) to yield
	\begin{equation*}
	\begin{split}
	(\star_2) - (\star_1)\times\hat{\Delta} = \left(\elem{\ms_2} \Delta' - \msE{} \Delta_{\ms_2}+\elem{\ms_1}\Delta_{\ms_2}\hat{\Delta}\right) \shiftideal{M}{\pB{i}}{ -\bideg{\ms}-\bideg{\ms_2}} \subset {D}_{\pB{i}}\;.
	\end{split}
	\end{equation*}
	We can further reduce the above relation to
	\begin{equation*}
	\begin{split}
	&\mbf{u}^{\mbf{\Delta p}_{\ms}}\Delta_{\ms_2}\left(\msE{} - \elem{\ms_1}\hat{\Delta}\right) \shiftideal{M}{\pB{i}}{ -\bideg{\ms}-\bideg{\ms_2}-\mbf{\Delta p}_{\ms}} \subset {D}_{\pB{i}}
	\end{split}
	\end{equation*}
	because $\shiftideal{L}{\pB{i}}{ -\bideg{\ms_2}} = \mbf{u}^{\bdegreev_i}\shiftideal{L}{\pB{i-1}}{ -\bideg{\ms_2}-\bdegreev_i}$ from Equation \eqref{eq:ideal_i_im1} and $\Delta'$ is a multiple of $u$.
\end{proof}

Lemma \ref{lem:new_relations} shows that, even if $\ms_2 > \ms$, it may be a part of the contribution that $\msE{}$ makes toward $\initial{D_{\pB{i}}}$.
Given enough new generators of ${D}_{\pB{i}}$ of this form, we can go a step further and identify some additional generators.
To this end, given a bi-degree $\abdegree$, define the set ${\Upsilon}_{\pB{i}}(\ms,\abdegree)$ as
\begin{equation}
\begin{split}
\Upsilon_{\pB{i}}(\ms, \abdegree) := \bigg\{ \Upsilon \subset \Upsilon_{\pB{i}}(\ms)~:~
&\sum_{(\cdot,\ms') \in \Upsilon} \Delta_{\ms'} \ideal[M]{}_{\pB{i}}(-\bideg{\ms'}-\bideg{\ms}-\mbf{\Delta p}_{\ms})_{\abdegree}\\
&\qquad\qquad\qquad\qquad\qquad= 
\shiftideal{M}{\pB{i}}{-\bideg{\ms}-\mbf{\Delta p}_{\ms}}_{\abdegree} \bigg\}\;.
\end{split}
\label{eq:contributions_partial}
\end{equation}
\begin{corollary}\label{cor:new_relations_weight}
	Let $\abdegree = (\adegree, \adegree') \in \ZZP^2$, $\ms \in \IMSA{\pB{i}}{}$, $1 \leq i \leq \nlevels$, and consider $\Upsilon \subset \Upsilon_{\pB{i}}(\ms)$.
	Then $\Upsilon \in \Upsilon_{\pB{i}}(\ms, \abdegree)$ if
	\begin{equation*}
		\begin{dcases}
		\sum_{\substack{\ms'\\(\cdot, \ms') \in \Upsilon}} (m - \degree_{i1}-\bsmooth(\ms'))_+ \geq m - \degree_{i1} + 1\;, & \ms \in \MSH{\pB{i}}{}\;,\\
		\sum_{\substack{\ms'\\(\cdot, \ms') \in \Upsilon}} (m' - \degree_{i2}-\bsmooth(\ms'))_+ \geq m' - \degree_{i2} + 1\;, & \ms \in \MSV{\pB{i}}{}\;.
		\end{dcases}
	\end{equation*}
\end{corollary}
\begin{proof}
	The proof follows directly from Proposition \ref{prop:space_sum_dim} since $\degree_{(i-1)j} + \degreev_{ij} = \degree_{ij}$, $j = 1, 2$.
\end{proof}
Using the above definition, we define the set $\Theta_{\pB{i}}(\ms, \abdegree)$ as
\begin{equation}
\begin{split}
\Theta_{\pB{i}}(\ms,\abdegree) := \bigg\{ (\ms_1, \ms_2) \in &\IMSA{\pB{i}}{}\times\IMSA{\pB{i}}{}~:~ \ms_1 \cap \ms \neq \emptyset \neq \ms_2 \cap \ms\;,~\bsmooth(\ms_2) \geq \bsmooth(\ms_1),\\
&\exists \Upsilon \in \Upsilon_{\pB{i}}(\ms, \abdegree),~~ \forall (\ms_3, \cdot) \in \Upsilon,~~\ms_2 > \ms_1 > \ms_3
\bigg\}\;.
\end{split}
\label{eq:contributions_total}
\end{equation}
Given a maximal segment $\ms$ such that $\Upsilon_{\pB{i}}(\ms,\abdegree)$ is not empty, we can identify further contributions to the initial $\initial{D_{\pB{i},\abdegree}}$ from maximal segments that intersect $\ms$.
The next result elucidates our reasoning.

\vspace{-2\baselineskip}
\leavevmode{
	\begin{wrapfigure}[4]{r}{0.20\textwidth}
		\begin{tikzpicture}
		\coordinate (c0) at (0,0);
		\coordinate (c1) at (2,0);
		\coordinate (c2) at (0,2);
		\coordinate (c3) at (2,2);
		\draw[eThickness] (c0) -- (c1);
		\draw[eThickness] (c2) -- (c3);
		\draw[eThickness] (c0) -- (c2);
		\node[circle,fill=black,black,inner sep=0pt,minimum size=2pt] at (c0) {};
		\node[circle,fill=black,black,inner sep=0pt,minimum size=2pt] at (c2) {};
		\node[label={west:$\ms$}] at ($(c0)!0.5!(c2)$) {};
		\node[label={south:$\ms_1$}] at ($(c0)!0.5!(c1)$) {};
		\node[label={north:$\ms_2$}] at ($(c2)!0.5!(c3)$) {};
		\draw[->] ($(c0)!0.74!(c2)$) arc (-90:0:14pt) node [midway, right] {$\alpha_{\ms_2\ms\ms_1}$};
		\draw[<-] ($(c0)!0.26!(c1)$) arc (0:90:14pt) node [midway, right] {$\alpha_{\ms_1\ms\ms_2}$};
		\end{tikzpicture}
	\end{wrapfigure}
\begin{lemma}\label{lem:new_relations_transversal}
	For  bi-degree $\abdegree \in \ZZP^2$ and fixed maximal segment $\ms$, let $(\ms_1, \ms_2) \in \Theta_{\pB{i}}(\ms,\abdegree)$. Define
	\begin{equation*}
	\alpha_{\ms_1\ms\ms_2} := \begin{dcases}
	1\;, & \mbf{\Delta p}_\ms = \mbf{0}\;,\\
	0\;, & \mbf{\Delta p}_\ms \neq \mbf{0}\;,
	\end{dcases}\qquad
	\alpha_{\ms_2\ms\ms_1} := 1\;.
	\end{equation*}
	Then, we have
	\begin{equation*}
	\begin{split}
	&\msE{1}\alpha_{\ms_1\ms\ms_2}\Delta_{\ms} \shiftideal{M}{\pB{i}}{-\bideg{\ms_1}-\bideg{\ms}}_\abdegree \subset \initial{D_{\pB{i},\abdegree}}\;,\\
	&\msE{2}\alpha_{\ms_2\ms\ms_1}\Delta_{\ms} \shiftideal{M}{\pB{i}}{-\bideg{\ms_2}-\bideg{\ms}}_\abdegree \subset \initial{D_{\pB{i},\abdegree}}\;.
	\end{split}
	\end{equation*}
\end{lemma}}
\begin{proof}
	From Lemma \ref{lem:new_relations}, there exists $\Upsilon \in \Upsilon_{\pB{i}}(\ms, \abdegree)$ and polynomials $\eta_{\ms_3}$ such that
	\begin{equation*}
	g_0 = \msE{}\mbf{u}^{\mbf{\Delta p}_\ms}\ideal[M]{}_{\pB{i}}(-\bideg{\ms}-\mbf{\Delta p}_\ms)_\abdegree
	- \sum_{(\ms_3, \ms_4) \in \Upsilon} \elem{\ms_3} \mbf{u}^{\mbf{\Delta p}_\ms}\eta_{\ms_3} \Delta_{\ms_4}\ideal[M]{}_{\pB{i}}(-\bideg{\ms_4}-\bideg{\ms}-\mbf{\Delta p}_\ms)_{\abdegree}
	\end{equation*}
	is a relation in ${D_{\pB{i},\abdegree}}$.
	Then, as in the proof of Lemma \ref{lem:new_relations}, consider the relations between $\msE{1}$ and $\msE{}$, and between $\msE{2}$ and $\msE{}$; denote these with $g_1$ and $g_2$, respectively.
	Thereafter, if $\mbf{\Delta p}_\ms = \mbf{0}$, eliminate $\msE{}$ from both $g_1$ and $g_2$ using $g_0$ to get new relations $\overline{g}_1$ and $\overline{g}_2$; the claim will follow since $\ms_1, \ms_2 > \ms_3$ for all $(\ms_3, \cdot) \in \Upsilon$.
	
	If, on the other hand, $\mbf{\Delta p}_\ms \neq \mbf{0}$, combine $g_1$ and $g_2$ to get a new relation $g_3 = g_2 - g_1 \hat{\Delta}$.
	For an appropriate choice of $\hat{\Delta}$ (as in the proof of Lemma \ref{lem:new_relations}), it will be possible to eliminate $\msE{}$ from $g_3$, and the initial of $g_3$ will thus be a part of the contribution that $\msE{2}$ makes toward $\initial{D_{\pB{i},\abdegree}}$.
\end{proof}

Finally, we can toss in all the new generators identified in Lemmas \ref{lem:new_relations} and \ref{lem:new_relations_transversal} with the original set of generators in Equation \eqref{eq:H0_denominator}, and define
\begin{equation}
\ideal[D]{}_{\pB{i}, \abdegree}^{\ms} :=
\begin{array}{l}
\sum_{\ms' \in \Gamma_{\pB{i}}(\ms) } \Delta_{\ms'}\shiftideal{M}{\pB{i}}{-\bideg{\ms}-\bideg{\ms'}}_\abdegree\\
+ \sum_{(\cdot, \ms_1) \in \Upsilon_{\pB{i}}(\ms) } \mbf{u}^{\mbf{\Delta p}_\ms} \Delta_{\ms_1}\shiftideal{M}{\pB{i}}{-\bideg{\ms}-\bideg{\ms_1}-\mbf{\Delta p}_\ms}_\abdegree\\
+ \sum_{(\ms, \ms_2) \in \Theta_{\pB{i}}(\ms_3, \abdegree) } \alpha_{\ms\ms_3\ms_2}\Delta_{\ms_3} \shiftideal{M}{\pB{i}}{-\bideg{\ms}-\bideg{\ms_3}}_\abdegree\\
+ \sum_{(\ms_4, \ms) \in \Theta_{\pB{i}}(\ms_5, \abdegree) } \alpha_{\ms_4\ms_5\ms}\Delta_{\ms_5} \shiftideal{M}{\pB{i}}{-\bideg{\ms}-\bideg{\ms_5}}_\abdegree\;.
\end{array}
\label{eq:ms_contribution}
\end{equation}
This is the contribution of $\msE{}$ toward the initial of $D_{\pB{i}}$ corresponding to the generators identified in Equation \eqref{eq:H0_denominator} and Lemmas \ref{lem:new_relations} and \ref{lem:new_relations_transversal}.
Then, we can use the above definition and Lemma \ref{lem:H0_idealComplex_initial} to provide an upper bound on the dimension of $H_0(\degBa{\idealComplex}{i})$ that improves upon the one presented in Proposition \ref{prop:H0_idealComplex_upperBound}.
\begin{corollary}\label{cor:H0_idealComplex_upperBound_new_relations}
	\begin{equation*}
	\begin{split}
	\dimwp{H_0(\degBa{\idealComplex}{i})}_\abdegree &\leq 
	\sum_{\ms \in \IMSA{\pB{i}}{}} 
	\dimwp{\shiftideal{M}{\pB{i}}{-\bideg{\ms}}}_\abdegree
	- \dimwp{\ideal[D]{}_{\pB{i},\abdegree}^{\ms}}\;.
	\end{split}
	\end{equation*}
\end{corollary}

%% file: ch_stable_dimension_configs.tex
\section{Configurations with stable dimension}\label{sec:stable_dimension}
In this section, we outline sufficient conditions that guarantee that the dimension of the spline space $\splSpaceH^\bsmooth_{\bdegree, \abdegree}$ can be computed exactly.
We will work in a setting where Configuration \ref{config:practical_smoothness} is true, i.e., where the bounds from Theorems \ref{thm:lower_bound_special} and \ref{thm:upper_bound_general} hold.
Note that supplementary Macaulay2 scripts accompanying some of the examples presented in this section and the next can be downloaded from \cite{M2Scripts}.
\begin{assumption}
	For the rest of this section, we assume that Configuration \ref{config:practical_smoothness} holds, i.e., the lower and upper bounds from Corollary \ref{thm:lower_bound_special} and Theorem \ref{thm:upper_bound_general} hold.
\end{assumption}
\begin{theorem}\label{thm:homo_dim_zero}
	For all $i$, if $\dimwp{H_0\left(\degBa{\idealComplex}{i}\right)}_\abdegree = \dimwp{H_0\left(\degBa{\constantComplex}{i}\right)}_\abdegree$, then $\homoDim^\bsmooth_{\bdegree, \abdegree} = 0$, i.e.,
	\begin{equation*}
		\dimwp{\splSpaceH^\bsmooth_{\bdegree}}_{\abdegree} = \euler{\quotientComplex}_\abdegree\;.
	\end{equation*}
\end{theorem}

As per Theorem \ref{thm:homo_dim_zero}, if the upper bound in Corollary \ref{cor:H0_idealComplex_upperBound_new_relations} equals $\dimwp{H_0\left(\degBa{\constantComplex}{i}\right)}_\abdegree$ then the dimension of the spline space can be exactly determined.
Before stating sufficient conditions on $\tmesh, \bdegree, \bsmooth, \abdegree$ that make this possible, let us first consider an example where we explicitly use the results from the previous section to compute the spline space dimension.

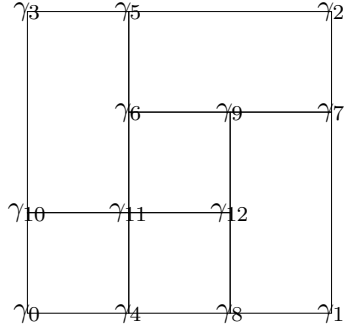
\begin{figure}[t]
	\centering
	\tikzsetnextfilename{./tikz/images/tmesh_type11}%
	\input{tmesh_type11}%

	\caption{The T-mesh used in Example \ref{ex:using_new_relations} to explicitly demonstrate how results from Section 3 can be used to find the dimension of $\splSpaceH^\bsmooth_{\bdegree, \abdegree}$.}
	\label{fig:using_new_relations}
\end{figure}
\begin{example}\label{ex:using_new_relations}
	Consider the T-mesh shown in Figure \ref{fig:using_new_relations} and let $\bsmooth(\edge) = 1$ for all interior edges. Let us consider two different degree deficit distributions on this mesh and find the dimension of the resulting spline space. In all of the following cases,
	the bi-smoothness for each maximal segment is simply $\bideg{\ms} = (2, 0)$ or $(0, 2)$. We will also use the following fact that is implied by Proposition \ref{prop:space_sum_dim} for real numbers $a_1 \neq a_2$ and $\abdegree \geq (3, 3)$,
	\begin{equation*}
		\begin{split}
			\dimwp{\shiftideal{M}{\pB{i}}{0, -2}}_\abdegree &= \dimwp{(s+a_1u)^2\shiftideal{M}{\pB{i}}{ -2, -2} + (s+a_2u)^2\shiftideal{M}{\pB{i}}{ -2, -2}}_\abdegree\;,\\
			\dimwp{\shiftideal{M}{\pB{i}}{-2, 0}}_\abdegree &= \dimwp{(t+a_1v)^2\shiftideal{M}{\pB{i}}{ -2, -2} + (t+a_2v)^2\shiftideal{M}{\pB{i}}{ -2, -2}}_\abdegree\;.
		\end{split}
		\tag*{($\star$)}
	\end{equation*}
	\begin{enumerate}[label=(\alph*)]
		\item Let $\bdegree(\face) = (1, 1)$ for all faces $\face \in \tmesh$ except for the face bounded by vertices $\vertex_4,\vertex_8,\vertex_{12},\vertex_{11}$; on the latter face the degree deficit is chosen to be $(0,0)$. Let us choose the associated degree-deficit sequence as $\bdegree_0 = (0,0) < \bdegree_1 = (1, 1)$ so that $\nlevels = 1$. We will choose $\abdegree = (3, 3)$. The following results follow on the different active meshes $\tmeshComponent{\pB{i}}{}$, $1 \leq i \leq \nlevels+1 = 2$.
		\begin{itemize}
			\item $[i = 2]$: From Proposition \ref{prop:H0_constantComplex}, $\dimwp{H_0\left(\degBa{\constantComplex}{i}\right)}_\abdegree = 0$. Furthermore, from Proposition \ref{prop:H0_idealComplex_ms}, there are no interior maximal segments therefore $\dimwp{H_0\left(\degBa{\idealComplex}{i}\right)}_\abdegree = 0$.
			\item $[i = 1]$: From Proposition \ref{prop:H0_constantComplex}, $\dimwp{H_0\left(\degBa{\constantComplex}{i}\right)}_\abdegree = 0$. Furthermore, there is a single active interior maximal segment $\ms = \vertex_{11}\vertex_{12}$, and the non-interior maximal segments $\ms_1 = \vertex_4\vertex_{11}$ and $\ms_2 = \vertex_8\vertex_{12}$ intersect it.
			Therefore, $\ms_1, \ms_2 \in \Gamma_{\pB{i}}(\ms)$ and from Proposition \ref{prop:H0_idealComplex_upperBound} and Equation $(\star)$ above, we get
			\begin{equation*}
				\dimwp{H_0\left(\degBa{\idealComplex}{i}\right)}_\abdegree \leq 
				\dimwp{\shiftideal{M}{\pB{i}}{0, -2}}_\abdegree - \dimwp{\sum_{\ms' \in \{\ms_1, \ms_2\}} \Delta_{\ms'}\shiftideal{M}{\pB{i}}{ -2, -2}}_\abdegree = 0\;.
			\end{equation*}
		\end{itemize}
		From the above we can see that Theorem \ref{thm:homo_dim_zero} applies. Thus, the dimension of $\splSpaceH^\bsmooth_{\bdegree, \abdegree}$ can be determined to be exactly $17$.
		It can be observed using the Macaulay2 script \cite[ex0a.m2]{M2Scripts} that $\abdegree = (3, 3)$ is also the smallest degree for which we get an increase in the dimension because of the non-uniformity in polynomial degrees.
		
		\item Let $\bdegree(\face) = (1, 1)$ for all faces $\face \in \tmesh$ except for the face bounded by vertices $\vertex_{6},\vertex_9,\vertex_{12},\vertex_{11}$; on the latter face the degree deficit is chosen to be $(0,0)$. Let us choose the associated degree-deficit sequence as $\bdegree_0 = (0,0) < \bdegree_1 = (1, 1)$ so that $\nlevels = 1$. We will choose $\abdegree = (4, 4)$. The following results follow on the different active meshes $\tmeshComponent{\pB{i}}{}$, $1 \leq i \leq \nlevels+1 = 2$.
		\begin{itemize}
			\item $[i = 2]$: From Proposition \ref{prop:H0_constantComplex}, $\dimwp{H_0\left(\degBa{\constantComplex}{i}\right)}_\abdegree = 0$. Furthermore, from Proposition \ref{prop:H0_idealComplex_ms}, there are no interior maximal segments therefore $\dimwp{H_0\left(\degBa{\idealComplex}{i}\right)}_\abdegree = 0$.
			
			\item $[i = 1]$: From Proposition \ref{prop:H0_constantComplex}, $\dimwp{H_0\left(\degBa{\constantComplex}{i}\right)}_\abdegree = 9$. In this case, there are $4$ active maximal segments $\ms_1 = \vertex_{11}\vertex_{12}$, $\ms_2 = \vertex_{12}\vertex_{9}$, $\ms_3 = \vertex_{9}\vertex_{6}$, $\ms_1 = \vertex_{6}\vertex_{11}$.
			Let us order these maximal segments as $\ms_3 > \ms_2 > \ms_4 > \ms_1$.
			Therefore,
			\begin{equation*}
				\Gamma_{\pB{i}}(\ms_1) = \emptyset\;,\quad
				\Gamma_{\pB{i}}(\ms_2) = \Gamma_{\pB{i}}(\ms_4) = \{\ms_1\}\;,\quad
				\Gamma_{\pB{i}}(\ms_3) = \{\ms_2, \ms_4\}\;.
			\end{equation*}
			Once again, from Proposition \ref{prop:H0_idealComplex_upperBound} and Equation $(\star)$, there is no contribution to $\dimwp{H_0\left(\degBa{\idealComplex}{i}\right)}_\abdegree$ from $\ms_3$.
			Furthermore, from Lemma \ref{lem:new_relations} and Corollary \ref{cor:new_relations_weight}, we can also verify that $\mbf{\Delta p}_{\ms_3} = (1, 0)$ and $\Upsilon_{\pB{i}}(\ms_3, \abdegree) = \{ \{(\ms_1, \ms_2), (\ms_1,\ms_4) \} \}$. Therefore, from Lemma \ref{lem:new_relations_transversal}, we can state that,
			\begin{equation*}
				\msE{2}\Delta_{\ms_3} \shiftideal{M}{\pB{i}}{-2, -2}_\abdegree \subset \initial{D_{\pB{i}}}\;.
			\end{equation*}
			However, since $\ms_1 \in \Gamma_{\pB{i}}(\ms_2)$, we also have the containment
			\begin{equation*}
				\msE{2}\Delta_{\ms_1} \shiftideal{M}{\pB{i}}{-2, -2}_\abdegree \subset \initial{D_{\pB{i}}}\;.
			\end{equation*}
			Therefore, again from Equation $(\star)$, there is no contribution to $\dimwp{H_0\left(\degBa{\idealComplex}{i}\right)}_\abdegree$ from $\ms_2$.
			Thus, the only contributions to the upper bound in Proposition \ref{prop:H0_idealComplex_upperBound} come from $\ms_1$ and $\ms_4$,
			\begin{equation*}
			\begin{split}
				\dimwp{H_0\left(\degBa{\idealComplex}{i}\right)}_\abdegree &\leq
				\dimwp{\msE{1}\shiftideal{M}{\pB{i}}{0, -2}}_\abdegree + \dimwp{\msE{4}\shiftideal{M}{\pB{i}}{-2, 0}}_\abdegree\\
				&\qquad\qquad\qquad - \dimwp{\msE{4}\Delta_{\ms_1}\shiftideal{M}{\pB{i}}{-2, -2}}_\abdegree\\
				& = 7 + 7 - 5 = 9 = \dimwp{H_0\left(\degBa{\constantComplex}{i}\right)}_\abdegree\;.
			\end{split}
			\end{equation*}
			Thus, $\dimwp{H_0\left(\degBa{\idealComplex}{i}\right)}_\abdegree = \dimwp{H_0\left(\degBa{\constantComplex}{i}\right)}_\abdegree$.
		\end{itemize}
		From the above we can see that Theorem \ref{thm:homo_dim_zero} applies. Thus, the dimension of $\splSpaceH^\bsmooth_{\bdegree, \abdegree}$ can be found to be $41$.
		As in part (a), it can be observed using the Macaulay2 script \cite[ex0b.m2]{M2Scripts} that $\abdegree = (4, 4)$ is again the smallest degree for which we get an increase in the dimension because of the non-uniformity in polynomial degrees.
	\end{enumerate}
\end{example}

Following Equation \eqref{eq:ms_contribution}, let us define the set $\Lambda_{\pB{i}}(\ms, \abdegree)$ to be
\begin{equation}
\begin{split}
	\Lambda_{\pB{i}}(\ms, \abdegree) := \bigg\{ \ms'~:~ &\ms' \in \Gamma_{\pB{i}}(\ms)\;, \text{~or~}
	(\cdot, \ms) \in \Theta_{\pB{i}}(\ms', \abdegree)\;, \text{~or~}\\
	&(\ms, \ms'') \in \Theta_{\pB{i}}(\ms', \abdegree)\text{~and~}\alpha_{\ms\ms'\ms''} = 1
	\bigg\}\;.
\end{split}
\label{eq:full_transversal_contribution}
\end{equation}

\begin{definition}[Maximal-segment weights]\label{def:max_seg_wt}
	For $\ms \in \IMSA{\pB{i}}{}$ and $\abdegree = (m, m') \in \ZZP^2$, the weight of $\ms$ is denoted by $\weight{i}{\ms}{\abdegree}$, and it is defined to be
	\begin{equation*}
	\weight{i}{\ms}{\abdegree} := 
	\begin{dcases}
		\sum_{\ms' \in \Lambda_{[i]} (\ms,\abdegree)} (m - \degree_{(i-1)1}-\bsmooth(\ms'))_+\;, & \ms \in \MSH{\pB{i}}{}\;,\\
		\sum_{\ms' \in \Lambda_{[i]} (\ms,\abdegree)} (m' - \degree_{(i-1)2}-\bsmooth(\ms'))_+\;, & \ms \in \MSV{\pB{i}}{}\;.
	\end{dcases}
	\end{equation*}
\end{definition}

\begin{lemma}\label{lem:suff_ms_wts}
	For $\abdegree = (m, m') \in \ZZP^2$,
	\begin{equation*}
		\weight{i}{\ms}{\abdegree} \geq
		\begin{dcases}
			m - \degree_{(i-1)1}+1\;, & \ms \in \MSH{\pB{i}}{}\\
			m' - \degree_{(i-1)2}+1\;, & \ms \in \MSV{\pB{i}}{}
		\end{dcases}
		\Rightarrow
		\ideal[D]{}_{\pB{i},\abdegree}^{\ms} =
		\ideal[M]{}_{\pB{i}}(-\bideg{\ms})_{\abdegree}
		\;.
	\end{equation*}
\end{lemma}

\begin{example}\label{ex:using_new_relations_revisit}
	Let us revisit Example \ref{ex:using_new_relations}(b).
	For $i = 2$ there are no interior maximal segments; so, let us look at the case of $i = 1$. We have
	\begin{equation*}
	\begin{split}
		&\Gamma_{\pB{i}}(\ms_3) = \{\ms_2, \ms_4\}\;,\quad
		\Gamma_{\pB{i}}(\ms_2) = \{\ms_1\}\;,\quad
		\Gamma_{\pB{i}}(\ms_4) = \{\ms_1\}\;,\quad
		\Theta_{\pB{i}}(\ms_3,\abdegree) = \{(\ms_2, \ms_4)\}\;.
	\end{split}
	\end{equation*}
	Thus, from Equation \ref{eq:full_transversal_contribution}, we have
	\begin{equation*}
	\begin{split}
		&\Lambda_{\pB{i}}(\ms_1,\abdegree) = \emptyset\;,\quad
		\Lambda_{\pB{i}}(\ms_2,\abdegree) = \{ \ms_1, \ms_3 \}\;,\quad
		\Lambda_{\pB{i}}(\ms_3,\abdegree) = \{ \ms_2, \ms_4 \}\;,\quad
		\Lambda_{\pB{i}}(\ms_4,\abdegree) = \{ \ms_1 \}\;,
	\end{split}
	\end{equation*}
	and from Definition \ref{def:max_seg_wt}, for $\ms \in \{\ms_2, \ms_3\}$ we obtain
	\begin{equation*}
		\weight{i}{\ms}{\abdegree} = 2 \times (4-1-1) = 4
		\;.
	\end{equation*}
	Then, from Lemma \ref{lem:suff_ms_wts}, we see that $\ms_2$ and $\ms_3$ do not contribute to $\dimwp{H_0\left(\degBa{\idealComplex}{i}\right)}_\abdegree$.
\end{example}

%% file: tmesh_type11.tex
\begin{tikzpicture}
	 \begin{scope}
		\node (v0) at (0.0000000000000000,0.0000000000000000) {$\vertex_{0}$};
		\node (v1) at (4.0000000000000000,0.0000000000000000) {$\vertex_{1}$};
		\node (v2) at (4.0000000000000000,4.0000000000000000) {$\vertex_{2}$};
		\node (v3) at (0.0000000000000000,4.0000000000000000) {$\vertex_{3}$};
		\node (v4) at (1.3333333333333333,0.0000000000000000) {$\vertex_{4}$};
		\node (v5) at (1.3333333333333333,4.0000000000000000) {$\vertex_{5}$};
		\node (v6) at (1.3333333333333333,2.6666666666666665) {$\vertex_{6}$};
		\node (v7) at (4.0000000000000000,2.6666666666666665) {$\vertex_{7}$};
		\node (v8) at (2.6666666666666665,0.0000000000000000) {$\vertex_{8}$};
		\node (v9) at (2.6666666666666665,2.6666666666666665) {$\vertex_{9}$};
		\node (v10) at (0.0000000000000000,1.3333333333333333) {$\vertex_{10}$};
		\node (v11) at (1.3333333333333333,1.3333333333333333) {$\vertex_{11}$};
		\node (v12) at (2.6666666666666665,1.3333333333333333) {$\vertex_{12}$};

		 \draw[step=1,eThickness] (v4.center) -- (v11.center) -- (v10.center) -- (v0.center) -- cycle;
		 \draw[step=1,eThickness] (v8.center) -- (v12.center) -- (v11.center) -- (v4.center) -- cycle;
		 \draw[step=1,eThickness] (v9.center) -- (v7.center) -- (v2.center) -- (v5.center) -- (v6.center) -- cycle;
		 \draw[step=1,eThickness] (v1.center) -- (v7.center) -- (v9.center) -- (v12.center) -- (v8.center) -- cycle;
		 \draw[step=1,eThickness] (v11.center) -- (v6.center) -- (v5.center) -- (v3.center) -- (v10.center) -- cycle;
		 \draw[step=1,eThickness] (v12.center) -- (v9.center) -- (v6.center) -- (v11.center) -- cycle;
	 \end{scope}
\end{tikzpicture}

%% file: ch_examples.tex
\section{Examples}\label{sec:examples}
In this section, we provide three examples that illustrate how the theory developed in this document can be used to compute the spline space dimension in the presence of non-uniform degrees.
The first two examples show configurations where Theorem \ref{thm:homo_dim_zero} applies, i.e., where the dimension can be computed exactly.
The last example serves to counter the expectation that Theorem \ref{thm:homo_dim_zero} will apply in all circumstances.

\begin{figure}[t!]
	\centering
	\tikzsetnextfilename{./tikz/images/tmesh_type16a}%
	\input{tmesh_type16a}%

	\caption{A non-uniform degree $C^1$ spline space consisting of quadratic and cubic polynomial pieces is built on the above mesh. Example \ref{ex:test_1} shows that the dimension of the spline space can be computed using Theorem \ref{thm:homo_dim_zero}.}
	\label{fig:test_1}
\end{figure}
\begin{example}\label{ex:test_1}
	Consider the T-mesh shown in Figure \ref{fig:test_1}. Let us build a $C^1$ spline space on this mesh, i.e., $\bsmooth(\edge) = 1$ for all interior edges $\edge$.
	The degree deficit on the shaded faces is chosen to be $(0,0)$ and on the white faces it is chosen to be $(1, 1)$.
	We choose $\bdegree_0 = (0,0)$ and $\bdegree_1 = (1, 1)$, i.e., $\nlevels = 1$, and choose $\abdegree = (3, 3)$.
	Let us examine the active T-meshes $\tmeshComponent{\pB{i}}{}$ in the following for $i = 1, 2$.
	\begin{itemize}
		\item $[i = 2]$: The only interior maximal segments in $\tmeshComponent{\pB{i}}{}$ are
		\begin{equation*}
			\ms_1 = \vertex_{20}\vertex_{22}\;,\;
			\ms_2 = \vertex_{23}\vertex_{25}\;,\;
			\ms_3 = \vertex_{26}\vertex_{28}\;,\;
			\ms_4 = \vertex_{29}\vertex_{31}\;.
		\end{equation*}
		Let us order these interior maximal segments as $\ms_1 < \ms_2 < \ms_3 < \ms_4$.
		Then notice that the cardinality of $\Gamma_{\pB{i}}(\ms_j)$ is $3$ for each $\ms_j$.
		Therefore, we can compute the weight of each interior maximal segment to be $\weight{i}{\ms_j}{\abdegree} \geq 3$ for all $j$.
		Then, using Lemma \ref{lem:suff_ms_wts}, we see that
		\begin{equation*}
			\dimwp{H_0\left(\degBa{\idealComplex}{i}\right)}_\abdegree = \dimwp{H_0\left(\degBa{\constantComplex}{i}\right)}_\abdegree = 0\;.
		\end{equation*}
		
		\item $[i = 1]$: All maximal segments in $\tmeshComponent{\pB{i}}{}$ are interior maximal segments, let us denote them as below,
		\begin{equation*}
		\begin{split}
			&\ms_1 = \vertex_{11}\vertex_{13}\;,\;
			\ms_2 = \vertex_{22}\vertex_{18}\;,\;
			\ms_3 = \vertex_{29}\vertex_{31}\;,\;
			\ms_4 = \vertex_{12}\vertex_{17}\;,\;
			\ms_5 = \vertex_{26}\vertex_{28}\;,\\
			&\ms_6 = \vertex_{11}\vertex_{23}\;,\;
			\ms_7 = \vertex_{29}\vertex_{22}\;,\;
			\ms_8 = \vertex_{23}\vertex_{25}\;,\;
			\ms_9 = \vertex_{28}\vertex_{18}\;.
		\end{split}
		\end{equation*}
		Let us order these interior maximal segments so that $\ms_j < \ms_k$ when $j < k$.
		The sets $\Gamma_{\pB{i}}(\ms_j)$ can be seen to be
		\begin{equation*}
		\begin{split}
			&\Gamma_{\pB{i}}(\ms_1) = \emptyset\;,\;
			\Gamma_{\pB{i}}(\ms_2) = 
			\Gamma_{\pB{i}}(\ms_3) = 
			\Gamma_{\pB{i}}(\ms_4) = 
			\Gamma_{\pB{i}}(\ms_5) = 
			\Gamma_{\pB{i}}(\ms_6) = \{ \ms_1 \}\;,\\
			&\Gamma_{\pB{i}}(\ms_7) = \{ \ms_2, \ms_3 \}\;,\;
			\Gamma_{\pB{i}}(\ms_8) = \{ \ms_2, \ms_3, \ms_4, \ms_5, \ms_6 \}\;,\;
			\Gamma_{\pB{i}}(\ms_9) = \{ \ms_2, \ms_4, \ms_5 \}\;.
		\end{split}
		\end{equation*}
		Next, it can be seen from Corollary \ref{cor:new_relations_weight} that
		\begin{equation*}
			\Upsilon_{\pB{i}}(\ms_8, \abdegree) \ni \Upsilon := \{ (\ms_1, \ms_2),  (\ms_1, \ms_3), (\ms_1, \ms_4), (\ms_1, \ms_5), (\ms_1, \ms_6)\}\;.
		\end{equation*}
		Then, from Lemma \ref{lem:new_relations_transversal}, we can build the sets $\Lambda_{\pB{i}}(\ms_j)$ for $j = 3, \dots, 6$ such that
		\begin{equation*}
			\begin{split}
				&\{ \ms_1, \ms_8 \} \subseteq 
				\Lambda_{\pB{i}}(\ms_3)\;,
				\Lambda_{\pB{i}}(\ms_4)\;,
				\Lambda_{\pB{i}}(\ms_5)\;,
				\Lambda_{\pB{i}}(\ms_6)\;.
			\end{split}
		\end{equation*}
		Then, we see that the weight $\weight{i}{\ms_j}{\abdegree} \geq 4$ for all $j = 3, \dots, 9$, and thus the upper bound on the dimension of $H_0\left(\degBa{\idealComplex}{i}\right)$ can be computed to be the following,
		\begin{equation*}
		\begin{split}
			\dimwp{H_0\left(\degBa{\idealComplex}{i}\right)}_\abdegree &\leq
			\dimwp{\msE{1}\shiftideal{M}{\pB{i}}{0, -2}}_\abdegree + \dimwp{\msE{2}\shiftideal{M}{\pB{i}}{-2, 0}}_\abdegree\\
		&\qquad\qquad - \dimwp{\msE{2}\Delta_{\ms_1}\shiftideal{M}{\pB{i}}{-2, -2}}_\abdegree\\
		& = 5 + 5 - 3 = 7 = \dimwp{H_0\left(\degBa{\constantComplex}{i}\right)}_\abdegree\;.
		\end{split}
		\end{equation*}
		Therefore, $\dimwp{H_0\left(\degBa{\idealComplex}{i}\right)}_\abdegree = \dimwp{H_0\left(\degBa{\constantComplex}{i}\right)}_\abdegree$.
	\end{itemize}
	From the above we can see that Theorem \ref{thm:homo_dim_zero} applies and
	\begin{equation*}
		\dimwp{\splSpaceH^\bsmooth_{\bdegree}}_{\abdegree} = \euler{\quotientComplex}_\abdegree = 37\;.
	\end{equation*}
	In particular, there are $30$ splines on $\tmeshComponent{\pB{2}}{}$ and $7$ on $\tmeshComponent{\pB{1}}{}$.
	The reader can use the accompanying Macaulay2 script \cite[ex1.m2]{M2Scripts} to confirm that $\abdegree = (3, 3)$ is the smallest bi-degree for which non-uniformity in degrees leads to an increase in the dimension.
\end{example}

\begin{figure}[t!]
	\centering
	\tikzsetnextfilename{./tikz/images/tmesh_type17a}%
	\input{tmesh_type17a}%

	\caption{A non-uniform degree $C^2$ spline space consisting of cubic and quartic polynomial pieces is built on the above mesh. Example \ref{ex:test_2} shows that the dimension of the spline space can be computed using Theorem \ref{thm:homo_dim_zero}.}
	\label{fig:test_2}
\end{figure}
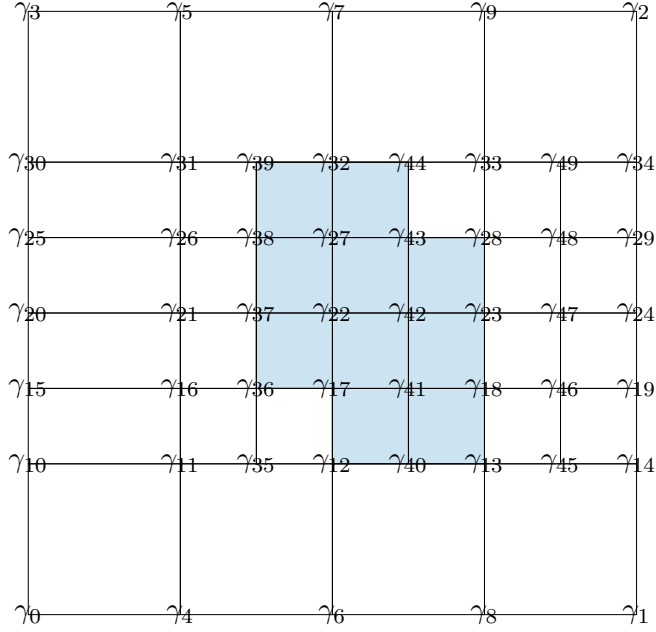
\begin{example}\label{ex:test_2}
	Consider the T-mesh shown in Figure \ref{fig:test_1}. Let us build a $C^2$ spline space on this mesh, i.e., $\bsmooth(\edge) = 2$ for all interior edges $\edge$.
	The degree deficit on the shaded faces is chosen to be $(0,0)$ and on the white faces it is chosen to be $(1, 1)$.
	We choose $\bdegree_0 = (0,0)$ and $\bdegree_1 = (1, 1)$, i.e., $\nlevels = 1$, and choose $\abdegree = (4, 4)$.
	Let us examine the active T-meshes $\tmeshComponent{\pB{i}}{}$ in the following for $i = 1, 2$.
	\begin{itemize}
		\item $[i = 2]$: The only interior maximal segments in $\tmeshComponent{\pB{i}}{}$ are
		\begin{equation*}
			\ms_1 = \vertex_{35}\vertex_{39}\;,\;
			\ms_2 = \vertex_{40}\vertex_{44}\;,\;
			\ms_3 = \vertex_{45}\vertex_{49}\;.
		\end{equation*}
		We can order them in any manner with respect to each other since they don't intersect each other.
		The weight of each interior maximal segment can be computed to be $\weight{i}{\ms_j}{\abdegree} = 5$ for all $j$.
		Then, using Lemma \ref{lem:suff_ms_wts}, we see that
		\begin{equation*}
			\dimwp{H_0\left(\degBa{\idealComplex}{i}\right)}_\abdegree = \dimwp{H_0\left(\degBa{\constantComplex}{i}\right)}_\abdegree = 0\;.
		\end{equation*}
		
		\item $[i = 1]$: All maximal segments in $\tmeshComponent{\pB{i}}{}$ are interior maximal segments, let us denote them as below,
		\begin{equation*}
		\begin{split}
			&\ms_1 = \vertex_{12}\vertex_{32}\;,\;
			\ms_2 = \vertex_{36}\vertex_{18}\;,\;
			\ms_3 = \vertex_{12}\vertex_{13}\;,\;
			\ms_4 = \vertex_{37}\vertex_{23}\;,\;
			\ms_5 = \vertex_{38}\vertex_{28}\;,\\
			&
			\ms_6 = \vertex_{39}\vertex_{44}\;,\;
			\ms_7 = \vertex_{36}\vertex_{39}\;,\;
			\ms_8 = \vertex_{40}\vertex_{44}\;,\;
			\ms_9 = \vertex_{13}\vertex_{28}\;.
		\end{split}
		\end{equation*}
		Let us order these interior maximal segments so that $\ms_j < \ms_k$ when $j < k$.
		The sets $\Gamma_{\pB{i}}(\ms_j)$ can be seen to be
		\begin{equation*}
		\begin{split}
		&\Gamma_{\pB{i}}(\ms_1) = \emptyset\;,\;
		\Gamma_{\pB{i}}(\ms_2) = 
		\Gamma_{\pB{i}}(\ms_3) = 
		\Gamma_{\pB{i}}(\ms_4) = 
		\Gamma_{\pB{i}}(\ms_5) = 
		\Gamma_{\pB{i}}(\ms_6) = \{ \ms_1 \}\;,\\
		&\Gamma_{\pB{i}}(\ms_7) = \{ \ms_2, \ms_4 , \ms_5, \ms_6\}\;,\;
		\Gamma_{\pB{i}}(\ms_8) = \{ \ms_2, \ms_3, \ms_4 , \ms_5, \ms_6\}\;,\;
		\Gamma_{\pB{i}}(\ms_9) = \{ \ms_2, \ms_3, \ms_4 , \ms_5\}\;.
		\end{split}
		\end{equation*}
		Next, it can be seen from Corollary \ref{cor:new_relations_weight} that
		\begin{equation*}
		\begin{split}
			&\Upsilon_{\pB{i}}(\ms_7, \abdegree) \ni \{ (\ms_1, \ms_2),  (\ms_1, \ms_4), (\ms_1, \ms_5), (\ms_1, \ms_6)\}\;,\\
			&\Upsilon_{\pB{i}}(\ms_8, \abdegree) \ni \{ (\ms_1, \ms_2),  (\ms_1, \ms_3), (\ms_1, \ms_4), (\ms_1, \ms_5), (\ms_1, \ms_6)\}\;,\\
			&\Upsilon_{\pB{i}}(\ms_9, \abdegree) \ni \{ (\ms_1, \ms_2),  (\ms_1, \ms_3), (\ms_1, \ms_4), (\ms_1, \ms_5)\}\;.
		\end{split}
		\end{equation*}
		Then, from Lemma \ref{lem:new_relations_transversal}, we can build the sets $\Lambda_{\pB{i}}(\ms_j)$ for $j = 3, \dots, 6$ such that
		\begin{equation*}
		\begin{split}
		&\{ \ms_1, \ms_8, \ms_9 \} \subseteq 
		\Lambda_{\pB{i}}(\ms_3)\;,\;
		\{ \ms_1, \ms_7, \ms_8 \} \subseteq 
		\Lambda_{\pB{i}}(\ms_4)\;,
		\Lambda_{\pB{i}}(\ms_5)\;,
		\Lambda_{\pB{i}}(\ms_6)\;.
		\end{split}
		\end{equation*}
		Then, we see that the weight $\weight{i}{\ms_j}{\abdegree} > 5$ for all $j = 3, \dots, 9$, and thus the upper bound on the dimension of $H_0\left(\degBa{\idealComplex}{i}\right)$ can be computed to be the following,
		\begin{equation*}
		\begin{split}
		\dimwp{H_0\left(\degBa{\idealComplex}{i}\right)}_\abdegree &\leq
		\dimwp{\msE{1}\shiftideal{M}{\pB{i}}{0, -3}}_\abdegree + \dimwp{\msE{2}\shiftideal{M}{\pB{i}}{-3, 0}}_\abdegree\\
		&\qquad\qquad- \dimwp{\msE{2}\Delta_{\ms_1}\shiftideal{M}{\pB{i}}{-3, -3}}_\abdegree\\
		& = 6 + 6 - 3 = 9 = \dimwp{H_0\left(\degBa{\constantComplex}{i}\right)}_\abdegree\;.
		\end{split}
		\end{equation*}
		Therefore, $\dimwp{H_0\left(\degBa{\idealComplex}{i}\right)}_\abdegree = \dimwp{H_0\left(\degBa{\constantComplex}{i}\right)}_\abdegree$.
	\end{itemize}
	From the above we can see that Theorem \ref{thm:homo_dim_zero} applies and
	\begin{equation*}
		\dimwp{\splSpaceH^\bsmooth_{\bdegree}}_{\abdegree} = \euler{\quotientComplex}_\abdegree = 75\;.
	\end{equation*}
	In particular, there are $66$ splines on $\tmeshComponent{\pB{2}}{}$ and $9$ on $\tmeshComponent{\pB{1}}{}$.
	The reader can use the accompanying Macaulay2 script \cite[ex2.m2]{M2Scripts} can be used to confirm that $\abdegree = (4, 4)$ is the smallest bi-degree for which non-uniformity in degrees leads to an increase in the dimension.
\end{example}

\begin{figure}[t!]
	\centering
	\tikzsetnextfilename{./tikz/images/tmesh_type18}%
	\input{tmesh_type18}%

	\caption{A non-uniform degree $C^2$ spline space consisting of cubic and quartic polynomial pieces is built on the above mesh. Example \ref{ex:test_3} shows that the dimension of the spline space coincides with the upper bound implied by Theorem \ref{thm:homo_dim_zero}.}
	\label{fig:test_3}
\end{figure}
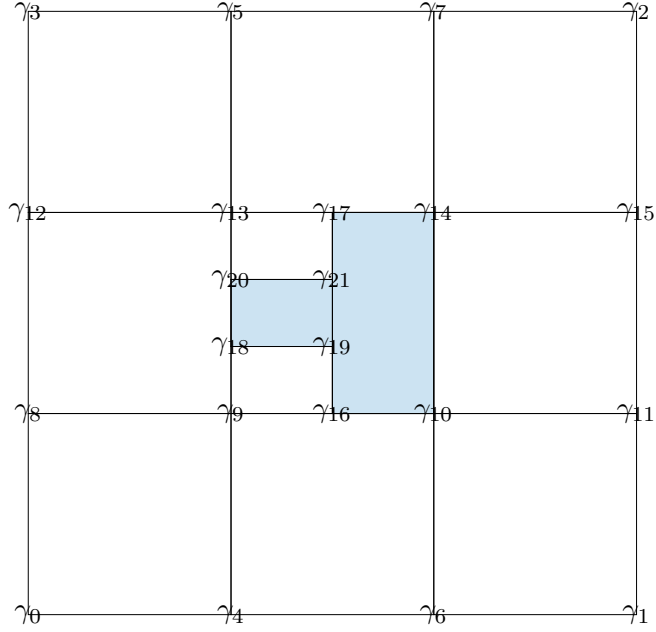
\begin{example}\label{ex:test_3}
	Consider the T-mesh shown in Figure \ref{fig:test_3}.
	Let us build a $C^2$ spline space on this mesh, i.e., $\bsmooth(\edge) = 2$ for all interior edges $\edge$.
	The degree deficit on the shaded faces is chosen to be $(0,0)$ and on the white faces it is chosen to be $(1, 1)$.
	We choose $\bdegree_0 = (0,0)$ and $\bdegree_1 = (1, 1)$, i.e., $\nlevels = 1$, and choose $\abdegree = (6, 6)$.
	Let us examine the active T-meshes $\tmeshComponent{\pB{i}}{}$ in the following for $i = 1, 2$.
	\begin{itemize}
		\item $[i = 2]$: The only interior maximal segments in $\tmeshComponent{\pB{i}}{}$ are
		\begin{equation*}
		\ms_1 = \vertex_{16}\vertex_{17}\;,\;
		\ms_2 = \vertex_{20}\vertex_{21}\;,\;
		\ms_3 = \vertex_{18}\vertex_{19}\;.
		\end{equation*}
		Let us order them as $\ms_1 < \ms_2 < \ms_3$.
		Then, for $\abdegree = (6, 6)$, it can be computed that $\weight{i}{\ms_j}{\abdegree} \geq 6$.
		Then, using Lemma \ref{lem:suff_ms_wts}, we see that
		\begin{equation*}
			\dimwp{H_0\left(\degBa{\idealComplex}{i}\right)}_\abdegree = \dimwp{H_0\left(\degBa{\constantComplex}{i}\right)}_\abdegree = 0\;.
		\end{equation*}
		Note that $\abdegree = (6, 6)$ is the smallest bi-degree in which $H_0\left(\degBa{\idealComplex}{i}\right)$ vanishes.
		
		\item $[i = 1]$: All maximal segments in $\tmeshComponent{\pB{i}}{}$ are interior maximal segments, let us denote them as below,
		\begin{equation*}
		\begin{split}
		&\ms_1 = \vertex_{16}\vertex_{17}\;,\;
		\ms_2 = \vertex_{16}\vertex_{10}\;,\;
		\ms_3 = \vertex_{17}\vertex_{14}\;,\;
		\ms_4 = \vertex_{10}\vertex_{14}\;,\\
		&
		\ms_5 = \vertex_{18}\vertex_{19}\;,\;
		\ms_6 = \vertex_{20}\vertex_{21}\;,\;
		\ms_7 = \vertex_{18}\vertex_{20}\;.
		\end{split}
		\end{equation*}
		Let us order these interior maximal segments so that $\ms_j < \ms_k$ when $j < k$.
		The sets $\Gamma_{\pB{i}}(\ms_j)$ can be seen to be
		\begin{equation*}
		\begin{split}
			&\Gamma_{\pB{i}}(\ms_1) = \emptyset\;,\;
			\Gamma_{\pB{i}}(\ms_2) = 
			\Gamma_{\pB{i}}(\ms_3) = 
			\Gamma_{\pB{i}}(\ms_5) = 
			\Gamma_{\pB{i}}(\ms_6) = \{ \ms_1 \}\;,\\
			&\Gamma_{\pB{i}}(\ms_4) = \{ \ms_2, \ms_3\}\;,\;
			\Gamma_{\pB{i}}(\ms_7) = \{ \ms_5, \ms_6\}\;.
		\end{split}
		\end{equation*}
		Next, it can be seen from Corollary \ref{cor:new_relations_weight} that
		\begin{equation*}
		\begin{split}
		&\Upsilon_{\pB{i}}(\ms_4, \abdegree) \ni \{ (\ms_1, \ms_2),  (\ms_1, \ms_3)\}\;,\\
		&\Upsilon_{\pB{i}}(\ms_7, \abdegree) \ni \{ (\ms_1, \ms_5),  (\ms_1, \ms_6)\}\;.
		\end{split}
		\end{equation*}
		Then, from Lemma \ref{lem:new_relations_transversal}, we can build the sets $\Lambda_{\pB{i}}(\ms_3)$ and $\Lambda_{\pB{i}}(\ms_6)$ such that
		\begin{equation*}
		\begin{split}
			&\Lambda_{\pB{i}}(\ms_3) = \{ \ms_1, \ms_4\}\;,\;
			\Lambda_{\pB{i}}(\ms_6) = \{ \ms_1, \ms_7\} \;.
		\end{split}
		\end{equation*}
		Then, we see that the weight $\weight{i}{\ms_j}{\abdegree} \geq 8$ for $j = 3, 4, 6, 7$.
		Thus the upper bound on the dimension of $H_0\left(\degBa{\idealComplex}{i}\right)$ can be computed to be the following,
		\begin{equation*}
		\begin{split}
		\dimwp{H_0\left(\degBa{\idealComplex}{i}\right)}_\abdegree &\leq
		\dimwp{\msE{1}\shiftideal{M}{\pB{i}}{0, -3}}_\abdegree +
		\dimwp{\msE{2}\shiftideal{M}{\pB{i}}{-3, 0}}_\abdegree\\ 
		&\qquad+ \dimwp{\msE{5}\shiftideal{M}{\pB{i}}{0, -3}}_\abdegree
		- \dimwp{\msE{2}\Delta_{\ms_1}\shiftideal{M}{\pB{i}}{-3, -3}}_\abdegree\\
		&\qquad- \dimwp{\msE{5}\Delta_{\ms_1}\shiftideal{M}{\pB{i}}{-3, -3}}_\abdegree\\
		& = 10 + 10 + 10 - 7 - 7 = 16\;.
		\end{split}
		\end{equation*}
		On the other hand, using Proposition \ref{prop:H0_constantComplex}, we can compute that
		\begin{equation*}
			\dimwp{H_0\left(\degBa{\constantComplex}{i}\right)}_\abdegree = 13\;.
		\end{equation*}
		Therefore, $\dimwp{H_0\left(\degBa{\idealComplex}{i}\right)}_\abdegree - \dimwp{H_0\left(\degBa{\constantComplex}{i}\right)}_\abdegree  \leq 3$.
	\end{itemize}
	From the above we see that Theorem \ref{thm:homo_dim_zero} does not apply.
	Therefore, let us use Theorems \ref{thm:lower_bound_special} and \ref{thm:upper_bound_general} to bound the spline space dimension from below and above.
	Computing the Euler characteristic of $\quotientComplex$ to be
	$\euler{\quotientComplex}_\abdegree = 143$, we use those theorems to obtain the following,
	\begin{equation*}
		143 \leq \dimwp{\splSpaceH^\bsmooth_{\bdegree}}_{\abdegree} \leq 146\;.
	\end{equation*}
	The reader can use the accompanying Macaulay2 script \cite[ex3.m2]{M2Scripts} to confirm that the dimension of the spline space is exactly $146$ for this configuration and thus coincides with the computed upper bound.
	This example serves to show that there exist configurations where a maximal segment ordering that allows us to use Theorem \ref{thm:homo_dim_zero} does not exist.
	Indeed, the accompanying script can be used to verify that $\dimwp{H_0\left(\degBa{\idealComplex}{i}\right)}_\abdegree - \dimwp{H_0\left(\degBa{\constantComplex}{i}\right)}_\abdegree = 3$ for all bi-degrees greater than or equal to $(6, 6)$.
\end{example}

%% file: tmesh_type16a.tex
\begin{tikzpicture}[scale=2]
\begin{scope}
\node (v0) at (0.0000000000000000,0.0000000000000000) {};
\node (v1) at (4.0000000000000000,0.0000000000000000) {};
\node (v2) at (4.0000000000000000,4.0000000000000000) {};
\node (v3) at (0.0000000000000000,4.0000000000000000) {};
\node (v4) at (1.0000000000000000,0.0000000000000000) {};
\node (v5) at (1.0000000000000000,4.0000000000000000) {};
\node (v6) at (2.0000000000000000,0.0000000000000000) {};
\node (v7) at (2.0000000000000000,4.0000000000000000) {};
\node (v8) at (3.0000000000000000,0.0000000000000000) {};
\node (v9) at (3.0000000000000000,4.0000000000000000) {};
\node (v10) at (0.0000000000000000,2.0000000000000000) {};
\node (v11) at (1.0000000000000000,2.0000000000000000) {};
\node (v12) at (2.0000000000000000,2.0000000000000000) {};
\node (v13) at (3.0000000000000000,2.0000000000000000) {};
\node (v14) at (4.0000000000000000,2.0000000000000000) {};
\node (v15) at (0.0000000000000000,3.0000000000000000) {};
\node (v16) at (1.0000000000000000,3.0000000000000000) {};
\node (v17) at (2.0000000000000000,3.0000000000000000) {};
\node (v18) at (3.0000000000000000,3.0000000000000000) {};
\node (v19) at (4.0000000000000000,3.0000000000000000) {};
\node (v20) at (1.0000000000000000,1.5000000000000000) {};
\node (v21) at (2.0000000000000000,1.5000000000000000) {};
\node (v22) at (3.0000000000000000,1.5000000000000000) {};
\node (v23) at (1.0000000000000000,2.5000000000000000) {};
\node (v24) at (2.0000000000000000,2.5000000000000000) {};
\node (v25) at (3.0000000000000000,2.5000000000000000) {};
\node (v26) at (1.5000000000000000,2.0000000000000000) {};
\node (v27) at (1.5000000000000000,2.5000000000000000) {};
\node (v28) at (1.5000000000000000,3.0000000000000000) {};
\node (v29) at (2.5000000000000000,1.5000000000000000) {};
\node (v30) at (2.5000000000000000,2.0000000000000000) {};
\node (v31) at (2.5000000000000000,2.5000000000000000) {};

\draw[eThickness] (v4.center) -- (v20.center) -- (v11.center) -- (v10.center) -- (v0.center) -- cycle;
\draw[eThickness] (v6.center) -- (v21.center) -- (v20.center) -- (v4.center) -- cycle;
\draw[eThickness] (v8.center) -- (v22.center) -- (v29.center) -- (v21.center) -- (v6.center) -- cycle;
\draw[eThickness] (v1.center) -- (v14.center) -- (v13.center) -- (v22.center) -- (v8.center) -- cycle;
\draw[eThickness] (v11.center) -- (v23.center) -- (v16.center) -- (v15.center) -- (v10.center) -- cycle;
\draw[eThickness, fill=myBlue!20] (v26.center) -- (v27.center) -- (v23.center) -- (v11.center) -- cycle;
\draw[eThickness, fill=myBlue!20] (v30.center) -- (v31.center) -- (v24.center) -- (v12.center) -- cycle;
\draw[eThickness] (v14.center) -- (v19.center) -- (v18.center) -- (v25.center) -- (v13.center) -- cycle;
\draw[eThickness] (v16.center) -- (v5.center) -- (v3.center) -- (v15.center) -- cycle;
\draw[eThickness] (v28.center) -- (v17.center) -- (v7.center) -- (v5.center) -- (v16.center) -- cycle;
\draw[eThickness] (v18.center) -- (v9.center) -- (v7.center) -- (v17.center) -- cycle;
\draw[eThickness] (v19.center) -- (v2.center) -- (v9.center) -- (v18.center) -- cycle;
\draw[eThickness] (v21.center) -- (v12.center) -- (v26.center) -- (v11.center) -- (v20.center) -- cycle;
\draw[eThickness] (v29.center) -- (v30.center) -- (v12.center) -- (v21.center) -- cycle;
\draw[eThickness] (v27.center) -- (v28.center) -- (v16.center) -- (v23.center) -- cycle;
\draw[eThickness, fill=myBlue!20] (v31.center) -- (v25.center) -- (v18.center) -- (v17.center) -- (v24.center) -- cycle;
\draw[eThickness, fill=myBlue!20] (v12.center) -- (v24.center) -- (v27.center) -- (v26.center) -- cycle;
\draw[eThickness, fill=myBlue!20] (v24.center) -- (v17.center) -- (v28.center) -- (v27.center) -- cycle;
\draw[eThickness, fill=myBlue!20] (v22.center) -- (v13.center) -- (v30.center) -- (v29.center) -- cycle;
\draw[eThickness, fill=myBlue!20] (v13.center) -- (v25.center) -- (v31.center) -- (v30.center) -- cycle;

\node[] at (v0) {$\vertex_{0}$};
\node[] at (v1) {$\vertex_{1}$};
\node[] at (v2) {$\vertex_{2}$};
\node[] at (v3) {$\vertex_{3}$};
\node[] at (v4) {$\vertex_{4}$};
\node[] at (v5) {$\vertex_{5}$};
\node[] at (v6) {$\vertex_{6}$};
\node[] at (v7) {$\vertex_{7}$};
\node[] at (v8) {$\vertex_{8}$};
\node[] at (v9) {$\vertex_{9}$};
\node[] at (v10) {$\vertex_{10}$};
\node[] at (v11) {$\vertex_{11}$};
\node[] at (v12) {$\vertex_{12}$};
\node[] at (v13) {$\vertex_{13}$};
\node[] at (v14) {$\vertex_{14}$};
\node[] at (v15) {$\vertex_{15}$};
\node[] at (v16) {$\vertex_{16}$};
\node[] at (v17) {$\vertex_{17}$};
\node[] at (v18) {$\vertex_{18}$};
\node[] at (v19) {$\vertex_{19}$};
\node[] at (v20) {$\vertex_{20}$};
\node[] at (v21) {$\vertex_{21}$};
\node[] at (v22) {$\vertex_{22}$};
\node[] at (v23) {$\vertex_{23}$};
\node[] at (v24) {$\vertex_{24}$};
\node[] at (v25) {$\vertex_{25}$};
\node[] at (v26) {$\vertex_{26}$};
\node[] at (v27) {$\vertex_{27}$};
\node[] at (v28) {$\vertex_{28}$};
\node[] at (v29) {$\vertex_{29}$};
\node[] at (v30) {$\vertex_{30}$};
\node[] at (v31) {$\vertex_{31}$};

\end{scope}
\end{tikzpicture}

%% file: tmesh_type17a.tex
\begin{tikzpicture}[scale=2]
	 \begin{scope}
		\node (v0) at (0.0000000000000000,0.0000000000000000) {};
		\node (v1) at (4.0000000000000000,0.0000000000000000) {};
		\node (v2) at (4.0000000000000000,4.0000000000000000) {};
		\node (v3) at (0.0000000000000000,4.0000000000000000) {};
		\node (v4) at (1.0000000000000000,0.0000000000000000) {};
		\node (v5) at (1.0000000000000000,4.0000000000000000) {};
		\node (v6) at (2.0000000000000000,0.0000000000000000) {};
		\node (v7) at (2.0000000000000000,4.0000000000000000) {};
		\node (v8) at (3.0000000000000000,0.0000000000000000) {};
		\node (v9) at (3.0000000000000000,4.0000000000000000) {};
		\node (v10) at (0.0000000000000000,1.0000000000000000) {};
		\node (v11) at (1.0000000000000000,1.0000000000000000) {};
		\node (v12) at (2.0000000000000000,1.0000000000000000) {};
		\node (v13) at (3.0000000000000000,1.0000000000000000) {};
		\node (v14) at (4.0000000000000000,1.0000000000000000) {};
		\node (v15) at (0.0000000000000000,1.5000000000000000) {};
		\node (v16) at (1.0000000000000000,1.5000000000000000) {};
		\node (v17) at (2.0000000000000000,1.5000000000000000) {};
		\node (v18) at (3.0000000000000000,1.5000000000000000) {};
		\node (v19) at (4.0000000000000000,1.5000000000000000) {};
		\node (v20) at (0.0000000000000000,2.0000000000000000) {};
		\node (v21) at (1.0000000000000000,2.0000000000000000) {};
		\node (v22) at (2.0000000000000000,2.0000000000000000) {};
		\node (v23) at (3.0000000000000000,2.0000000000000000) {};
		\node (v24) at (4.0000000000000000,2.0000000000000000) {};
		\node (v25) at (0.0000000000000000,2.5000000000000000) {};
		\node (v26) at (1.0000000000000000,2.5000000000000000) {};
		\node (v27) at (2.0000000000000000,2.5000000000000000) {};
		\node (v28) at (3.0000000000000000,2.5000000000000000) {};
		\node (v29) at (4.0000000000000000,2.5000000000000000) {};
		\node (v30) at (0.0000000000000000,3.0000000000000000) {};
		\node (v31) at (1.0000000000000000,3.0000000000000000) {};
		\node (v32) at (2.0000000000000000,3.0000000000000000) {};
		\node (v33) at (3.0000000000000000,3.0000000000000000) {};
		\node (v34) at (4.0000000000000000,3.0000000000000000) {};
		\node (v35) at (1.5000000000000000,1.0000000000000000) {};
		\node (v36) at (1.5000000000000000,1.5000000000000000) {};
		\node (v37) at (1.5000000000000000,2.0000000000000000) {};
		\node (v38) at (1.5000000000000000,2.5000000000000000) {};
		\node (v39) at (1.5000000000000000,3.0000000000000000) {};
		\node (v40) at (2.5000000000000000,1.0000000000000000) {};
		\node (v41) at (2.5000000000000000,1.5000000000000000) {};
		\node (v42) at (2.5000000000000000,2.0000000000000000) {};
		\node (v43) at (2.5000000000000000,2.5000000000000000) {};
		\node (v44) at (2.5000000000000000,3.0000000000000000) {};
		\node (v45) at (3.5000000000000000,1.0000000000000000) {};
		\node (v46) at (3.5000000000000000,1.5000000000000000) {};
		\node (v47) at (3.5000000000000000,2.0000000000000000) {};
		\node (v48) at (3.5000000000000000,2.5000000000000000) {};
		\node (v49) at (3.5000000000000000,3.0000000000000000) {};

		\draw[eThickness, fill=myBlue!20] (v39.center) -- (v36.center) -- (v17.center) -- (v12.center) -- (v13.center) -- (v28.center) -- (v43.center) -- (v44.center) -- cycle;
		 \draw[eThickness] (v4.center) -- (v11.center) -- (v10.center) -- (v0.center) -- cycle;
		 \draw[eThickness] (v6.center) -- (v12.center) -- (v35.center) -- (v11.center) -- (v4.center) -- cycle;
		 \draw[eThickness] (v8.center) -- (v13.center) -- (v40.center) -- (v12.center) -- (v6.center) -- cycle;
		 \draw[eThickness] (v1.center) -- (v14.center) -- (v45.center) -- (v13.center) -- (v8.center) -- cycle;
		 \draw[eThickness] (v11.center) -- (v16.center) -- (v15.center) -- (v10.center) -- cycle;
		 \draw[eThickness] (v35.center) -- (v36.center) -- (v16.center) -- (v11.center) -- cycle;
		 \draw[eThickness] (v40.center) -- (v41.center) -- (v17.center) -- (v12.center) -- cycle;
		 \draw[eThickness] (v45.center) -- (v46.center) -- (v18.center) -- (v13.center) -- cycle;
		 \draw[eThickness] (v16.center) -- (v21.center) -- (v20.center) -- (v15.center) -- cycle;
		 \draw[eThickness] (v36.center) -- (v37.center) -- (v21.center) -- (v16.center) -- cycle;
		 \draw[eThickness] (v41.center) -- (v42.center) -- (v22.center) -- (v17.center) -- cycle;
		 \draw[eThickness] (v46.center) -- (v47.center) -- (v23.center) -- (v18.center) -- cycle;
		 \draw[eThickness] (v21.center) -- (v26.center) -- (v25.center) -- (v20.center) -- cycle;
		 \draw[eThickness] (v37.center) -- (v38.center) -- (v26.center) -- (v21.center) -- cycle;
		 \draw[eThickness] (v42.center) -- (v43.center) -- (v27.center) -- (v22.center) -- cycle;
		 \draw[eThickness] (v47.center) -- (v48.center) -- (v28.center) -- (v23.center) -- cycle;
		 \draw[eThickness] (v26.center) -- (v31.center) -- (v30.center) -- (v25.center) -- cycle;
		 \draw[eThickness] (v38.center) -- (v39.center) -- (v31.center) -- (v26.center) -- cycle;
		 \draw[eThickness] (v43.center) -- (v44.center) -- (v32.center) -- (v27.center) -- cycle;
		 \draw[eThickness] (v48.center) -- (v49.center) -- (v33.center) -- (v28.center) -- cycle;
		 \draw[eThickness] (v31.center) -- (v5.center) -- (v3.center) -- (v30.center) -- cycle;
		 \draw[eThickness] (v39.center) -- (v32.center) -- (v7.center) -- (v5.center) -- (v31.center) -- cycle;
		 \draw[eThickness] (v44.center) -- (v33.center) -- (v9.center) -- (v7.center) -- (v32.center) -- cycle;
		 \draw[eThickness] (v49.center) -- (v34.center) -- (v2.center) -- (v9.center) -- (v33.center) -- cycle;
		 \draw[eThickness] (v12.center) -- (v17.center) -- (v36.center) -- (v35.center) -- cycle;
		 \draw[eThickness] (v17.center) -- (v22.center) -- (v37.center) -- (v36.center) -- cycle;
		 \draw[eThickness] (v22.center) -- (v27.center) -- (v38.center) -- (v37.center) -- cycle;
		 \draw[eThickness] (v27.center) -- (v32.center) -- (v39.center) -- (v38.center) -- cycle;
		 \draw[eThickness] (v13.center) -- (v18.center) -- (v41.center) -- (v40.center) -- cycle;
		 \draw[eThickness] (v18.center) -- (v23.center) -- (v42.center) -- (v41.center) -- cycle;
		 \draw[eThickness] (v23.center) -- (v28.center) -- (v43.center) -- (v42.center) -- cycle;
		 \draw[eThickness] (v28.center) -- (v33.center) -- (v44.center) -- (v43.center) -- cycle;
		 \draw[eThickness] (v14.center) -- (v19.center) -- (v46.center) -- (v45.center) -- cycle;
		 \draw[eThickness] (v19.center) -- (v24.center) -- (v47.center) -- (v46.center) -- cycle;
		 \draw[eThickness] (v24.center) -- (v29.center) -- (v48.center) -- (v47.center) -- cycle;
		 \draw[eThickness] (v29.center) -- (v34.center) -- (v49.center) -- (v48.center) -- cycle;
		
		\node[inner sep=0] at (v0) {$\vertex_{0}$};
		\node[inner sep=0] at (v1) {$\vertex_{1}$};
		\node[inner sep=0] at (v2) {$\vertex_{2}$};
		\node[inner sep=0] at (v3) {$\vertex_{3}$};
		\node[inner sep=0] at (v4) {$\vertex_{4}$};
		\node[inner sep=0] at (v5) {$\vertex_{5}$};
		\node[inner sep=0] at (v6) {$\vertex_{6}$};
		\node[inner sep=0] at (v7) {$\vertex_{7}$};
		\node[inner sep=0] at (v8) {$\vertex_{8}$};
		\node[inner sep=0] at (v9) {$\vertex_{9}$};
		\node[inner sep=0] at (v10) {$\vertex_{10}$};
		\node[inner sep=0] at (v11) {$\vertex_{11}$};
		\node[inner sep=0] at (v12) {$\vertex_{12}$};
		\node[inner sep=0] at (v13) {$\vertex_{13}$};
		\node[inner sep=0] at (v14) {$\vertex_{14}$};
		\node[inner sep=0] at (v15) {$\vertex_{15}$};
		\node[inner sep=0] at (v16) {$\vertex_{16}$};
		\node[inner sep=0] at (v17) {$\vertex_{17}$};
		\node[inner sep=0] at (v18) {$\vertex_{18}$};
		\node[inner sep=0] at (v19) {$\vertex_{19}$};
		\node[inner sep=0] at (v20) {$\vertex_{20}$};
		\node[inner sep=0] at (v21) {$\vertex_{21}$};
		\node[inner sep=0] at (v22) {$\vertex_{22}$};
		\node[inner sep=0] at (v23) {$\vertex_{23}$};
		\node[inner sep=0] at (v24) {$\vertex_{24}$};
		\node[inner sep=0] at (v25) {$\vertex_{25}$};
		\node[inner sep=0] at (v26) {$\vertex_{26}$};
		\node[inner sep=0] at (v27) {$\vertex_{27}$};
		\node[inner sep=0] at (v28) {$\vertex_{28}$};
		\node[inner sep=0] at (v29) {$\vertex_{29}$};
		\node[inner sep=0] at (v30) {$\vertex_{30}$};
		\node[inner sep=0] at (v31) {$\vertex_{31}$};
		\node[inner sep=0] at (v32) {$\vertex_{32}$};
		\node[inner sep=0] at (v33) {$\vertex_{33}$};
		\node[inner sep=0] at (v34) {$\vertex_{34}$};
		\node[inner sep=0] at (v35) {$\vertex_{35}$};
		\node[inner sep=0] at (v36) {$\vertex_{36}$};
		\node[inner sep=0] at (v37) {$\vertex_{37}$};
		\node[inner sep=0] at (v38) {$\vertex_{38}$};
		\node[inner sep=0] at (v39) {$\vertex_{39}$};
		\node[inner sep=0] at (v40) {$\vertex_{40}$};
		\node[inner sep=0] at (v41) {$\vertex_{41}$};
		\node[inner sep=0] at (v42) {$\vertex_{42}$};
		\node[inner sep=0] at (v43) {$\vertex_{43}$};
		\node[inner sep=0] at (v44) {$\vertex_{44}$};
		\node[inner sep=0] at (v45) {$\vertex_{45}$};
		\node[inner sep=0] at (v46) {$\vertex_{46}$};
		\node[inner sep=0] at (v47) {$\vertex_{47}$};
		\node[inner sep=0] at (v48) {$\vertex_{48}$};
		\node[inner sep=0] at (v49) {$\vertex_{49}$};

	 \end{scope}
\end{tikzpicture}

%% file: tmesh_type18.tex
\begin{tikzpicture}[scale=2]
	 \begin{scope}
		\node (v0) at (0.0000000000000000,0.0000000000000000) {};
		\node (v1) at (4.0000000000000000,0.0000000000000000) {};
		\node (v2) at (4.0000000000000000,4.0000000000000000) {};
		\node (v3) at (0.0000000000000000,4.0000000000000000) {};
		\node (v4) at (1.3333333333333333,0.0000000000000000) {};
		\node (v5) at (1.3333333333333333,4.0000000000000000) {};
		\node (v6) at (2.6666666666666665,0.0000000000000000) {};
		\node (v7) at (2.6666666666666665,4.0000000000000000) {};
		\node (v8) at (0.0000000000000000,1.3333333333333333) {};
		\node (v9) at (1.3333333333333333,1.3333333333333333) {};
		\node (v10) at (2.6666666666666665,1.3333333333333333) {};
		\node (v11) at (4.0000000000000000,1.3333333333333333) {};
		\node (v12) at (0.0000000000000000,2.6666666666666665) {};
		\node (v13) at (1.3333333333333333,2.6666666666666665) {};
		\node (v14) at (2.6666666666666665,2.6666666666666665) {};
		\node (v15) at (4.0000000000000000,2.6666666666666665) {};
		\node (v16) at (2.0000000000000000,1.3333333333333333) {};
		\node (v17) at (2.0000000000000000,2.6666666666666665) {};
		\node (v18) at (1.3333333333333333,1.7777777777777777) {};
		\node (v19) at (2.0000000000000000,1.7777777777777777) {};
		\node (v20) at (1.3333333333333333,2.2222222222222223) {};
		\node (v21) at (2.0000000000000000,2.2222222222222223) {};
		
		\draw[eThickness, fill=myBlue!20] (v18.center) -- (v19.center) -- (v16.center) -- (v10.center) -- (v14.center) -- (v17.center) -- (v21.center) -- (v20.center) -- cycle;

		 \draw[step=1,eThickness] (v4.center) -- (v9.center) -- (v8.center) -- (v0.center) -- cycle;
		 \draw[step=1,eThickness] (v6.center) -- (v10.center) -- (v16.center) -- (v9.center) -- (v4.center) -- cycle;
		 \draw[step=1,eThickness] (v1.center) -- (v11.center) -- (v10.center) -- (v6.center) -- cycle;
		 \draw[step=1,eThickness] (v9.center) -- (v18.center) -- (v20.center) -- (v13.center) -- (v12.center) -- (v8.center) -- cycle;
		 \draw[step=1,eThickness] (v16.center) -- (v19.center) -- (v18.center) -- (v9.center) -- cycle;
		 \draw[step=1,eThickness] (v11.center) -- (v15.center) -- (v14.center) -- (v10.center) -- cycle;
		 \draw[step=1,eThickness] (v13.center) -- (v5.center) -- (v3.center) -- (v12.center) -- cycle;
		 \draw[step=1,eThickness] (v17.center) -- (v14.center) -- (v7.center) -- (v5.center) -- (v13.center) -- cycle;
		 \draw[step=1,eThickness] (v15.center) -- (v2.center) -- (v7.center) -- (v14.center) -- cycle;
		 \draw[step=1,eThickness] (v10.center) -- (v14.center) -- (v17.center) -- (v21.center) -- (v19.center) -- (v16.center) -- cycle;
		 \draw[step=1,eThickness] (v19.center) -- (v21.center) -- (v20.center) -- (v18.center) -- cycle;
		 \draw[step=1,eThickness] (v21.center) -- (v17.center) -- (v13.center) -- (v20.center) -- cycle;
		\node[inner sep=0] at (v0) {$\vertex_{0}$};
		\node[inner sep=0] at (v1) {$\vertex_{1}$};
		\node[inner sep=0] at (v2) {$\vertex_{2}$};
		\node[inner sep=0] at (v3) {$\vertex_{3}$};
		\node[inner sep=0] at (v4) {$\vertex_{4}$};
		\node[inner sep=0] at (v5) {$\vertex_{5}$};
		\node[inner sep=0] at (v6) {$\vertex_{6}$};
		\node[inner sep=0] at (v7) {$\vertex_{7}$};
		\node[inner sep=0] at (v8) {$\vertex_{8}$};
		\node[inner sep=0] at (v9) {$\vertex_{9}$};
		\node[inner sep=0] at (v10) {$\vertex_{10}$};
		\node[inner sep=0] at (v11) {$\vertex_{11}$};
		\node[inner sep=0] at (v12) {$\vertex_{12}$};
		\node[inner sep=0] at (v13) {$\vertex_{13}$};
		\node[inner sep=0] at (v14) {$\vertex_{14}$};
		\node[inner sep=0] at (v15) {$\vertex_{15}$};
		\node[inner sep=0] at (v16) {$\vertex_{16}$};
		\node[inner sep=0] at (v17) {$\vertex_{17}$};
		\node[inner sep=0] at (v18) {$\vertex_{18}$};
		\node[inner sep=0] at (v19) {$\vertex_{19}$};
		\node[inner sep=0] at (v20) {$\vertex_{20}$};
		\node[inner sep=0] at (v21) {$\vertex_{21}$};

	 \end{scope}
\end{tikzpicture}

%% file: ch_conclusions.tex
\section{Conclusions}
Splines have been used for geometric modeling for several decades, and they are now rapidly becoming indispensable tools for performing approximation.
In order to efficiently alter the local resolution offered by a spline space, it is important to be able to perform local adaptivity.
While local mesh adaptivity on quadrilateral meshes has been widely studied since the introduction of T-splines \cite{sederberg2003t}, theoretical studies focused on splines that allow local degree adaptivity have been missing heretofore from the bivariate spline literature.
Since the possibility of using non-uniform bi-degree splines on T-meshes would enable powerful new paradigms of local refinement, we take a first step in this direction by analyzing the the dimension of such spline spaces.
In particular, using tools from homological algebra, we provide combinatorial bounds on the dimension.
We also outline sufficient conditions that guarantee that the spline space dimension is \emph{stable}, i.e., the dimension of the space is independent of the geometry of the T-mesh for a fixed topology.
Several examples are provided to show applicability of the theory developed here.

The results presented in this paper can be used for classifying spline spaces with stable dimension.
This is important for avoiding geometry-based linear dependency issues that may arise when performing approximation.
The ability to combinatorially compute the spline space dimension is also important because it can be used to determine when a given set of linearly independent splines spans the full spline space.
Conversely, given a constructive approach that aims to produce linearly independent splines over T-meshes using only local data, computation of the associated spline space dimension can help identify cases where the splines produced by the approach cannot be linearly independent.
This is crucial for devising constructive approaches that can be robustly employed for performing isogeometric analysis.

Current research on this topic is progressing along several lines of inquiry.
A follow up paper will discuss local refinement algorithms for both mesh sizes and polynomial degrees that ensure stability of the spline space dimension.
The construction of a suitable basis that possesses B-spline-like properties remains an open and essential question and will be a part of future efforts focused on formulation of constructive approaches.
A practical construction of this nature has been successfully devised for univariate non-uniform degree splines \cite{toshniwal2017multi,speleers_computation_2018,toshniwal_multi-degree_2018}.
A generalization of this univariate approach to the bivariate setting has been recently conjectured \cite{thomas_usplines_2018}.